\documentclass[reqno]{amsart}
\usepackage{amsmath, amssymb, amsthm, epsfig}
\usepackage{hyperref, latexsym}
\usepackage{url}
\usepackage[mathscr]{euscript}

\usepackage{color}
\usepackage{fullpage} 
\usepackage{setspace}

\DeclareMathOperator{\sgn}{\mathrm{sgn}}

\begin{document}
 \bibliographystyle{plain}

 \newtheorem{theorem}{Theorem}
 \newtheorem{lemma}[theorem]{Lemma}
 \newtheorem{proposition}[theorem]{Proposition}
 \newtheorem{corollary}[theorem]{Corollary}
 \theoremstyle{definition}
 \newtheorem{definition}[theorem]{Definition}
 \newtheorem{example}[theorem]{Example}
 \theoremstyle{remark}
 \newtheorem{remark}[theorem]{Remark}
 \newcommand{\mc}{\mathcal}
 \newcommand{\A}{\mc{A}}
 \newcommand{\B}{\mc{B}}
 \newcommand{\cc}{\mc{C}}
 \newcommand{\D}{\mc{D}}
 \newcommand{\E}{\mathbb{E}}
 \newcommand{\F}{\mc{F}}
 \newcommand{\G}{\mc{G}}
 \newcommand{\sH}{\mc{H}}
 \newcommand{\I}{\mc{I}}
 \newcommand{\J}{\mc{J}}
 \newcommand{\nn}{\mc{N}}
 \newcommand{\rr}{\mc{R}}
 \newcommand{\sS}{\mc{S}}
 \newcommand{\U}{\mc{U}}
 \newcommand{\X}{\mc{X}}
 \newcommand{\Y}{\mc{Y}}
 \newcommand{\C}{\mathbb{C}}
 \newcommand{\R}{\mathbb{R}}
 \newcommand{\N}{\mathbb{N}}
 \newcommand{\Q}{\mathbb{Q}}
 \newcommand{\Z}{\mathbb{Z}}
 \newcommand{\csch}{\mathrm{csch}}
 \newcommand{\tF}{\widehat{F}}
 \newcommand{\tG}{\widehat{G}}
 \newcommand{\tH}{\widehat{H}}
 \newcommand{\tf}{\widehat{f}}
 \newcommand{\ug}{\widehat{g}}
 \newcommand{\wg}{\widetilde{g}}
 \newcommand{\uh}{\widehat{h}}
 \newcommand{\wh}{\widetilde{h}}
 \newcommand{\wl}{\widetilde{l}}
 \newcommand{\tk}{\widehat{k}}
 \newcommand{\tK}{\widehat{K}}
 \newcommand{\tl}{\widehat{l}}
 \newcommand{\tL}{\widehat{L}}
 \newcommand{\tm}{\widehat{m}}
 \newcommand{\tM}{\widehat{M}}
 \newcommand{\tp}{\widehat{\varphi}}
 \newcommand{\tq}{\widehat{q}}
 \newcommand{\tT}{\widehat{T}}
 \newcommand{\tU}{\widehat{U}}
 \newcommand{\tu}{\widehat{u}}
 \newcommand{\tV}{\widehat{V}}
 \newcommand{\tv}{\widehat{v}}
 \newcommand{\tW}{\widehat{W}}
 \newcommand{\ba}{\boldsymbol{a}}
 \newcommand{\bal}{\boldsymbol{\alpha}}
 \newcommand{\bx}{\boldsymbol{x}}
 \newcommand{\p}{\varphi}
 \newcommand{\f}{\frac52}
 \newcommand{\g}{\frac32}
 \newcommand{\h}{\frac12}
 \newcommand{\hh}{\tfrac12}
 \newcommand{\ds}{\text{\rm d}s}
 \newcommand{\dt}{\text{\rm d}t}
  \newcommand{\dr}{\text{\rm d}r}
 \newcommand{\du}{\text{\rm d}u}
 \newcommand{\dv}{\text{\rm d}v}
 \newcommand{\dw}{\text{\rm d}w}
  \newcommand{\dz}{\text{\rm d}z}
 \newcommand{\dx}{\text{\rm d}x}
   \newcommand{\dxx}{\text{\rm d}{\bf x}}
 \newcommand{\dy}{\text{\rm d}y}
 \newcommand{\dl}{\text{\rm d}\lambda}
 \newcommand{\dmu}{\text{\rm d}\mu(\lambda)}
 \newcommand{\dnu}{\text{\rm d}\nu(\lambda)}
  \newcommand{\dnuN}{\text{\rm d}\nu_N(\lambda)}
\newcommand{\dnus}{\text{\rm d}\nu_{\sigma}(\lambda)}
 \newcommand{\dlnu}{\text{\rm d}\nu_l(\lambda)}
 \newcommand{\dnnu}{\text{\rm d}\nu_n(\lambda)}
\newcommand{\sech}{\text{\rm sech}}
\newcommand{\CC}{\mathbb{C}}
\newcommand{\NN}{\mathbb{N}}
\newcommand{\RR}{\mathbb{R}}
\newcommand{\ZZ}{\mathbb{Z}}
\newcommand{\thp}{\theta^+}
\newcommand{\thpn}{\theta^+_N}
\newcommand{\vthp}{\vartheta^+}
\newcommand{\vthpn}{\vartheta^+_N}
\newcommand{\ft}[1]{\widehat{#1}}
\newcommand{\support}[1]{\mathrm{supp}(#1)}
\newcommand{\gplus}{G^+_\lambda}
\newcommand{\ph}{\mathcal{H}}
\newcommand{\godd}{G^o_\lambda}
\newcommand{\mplus}{M_\lambda^+}
\newcommand{\lplus}{L_\lambda^+}
\newcommand{\modd}{M_\lambda^o}
\newcommand{\lodd}{L_\lambda^o}
\newcommand{\sgp}{x_+^0}
\newcommand{\Tl}{T_\lambda}
\newcommand{\Tlc}{T_{\lambda,c}}
\newcommand{\El}{{E_\lambda}}

 \newcommand{\im}{{\rm Im}\,}
  \newcommand{\re}{{\rm Re}\,}

 \newcommand{\T}{\mc{T}}
 \newcommand{\M}{\mc{M}}
 \renewcommand{\L}{\mc{L}}
 \newcommand{\K}{\mc{K}}
 \renewcommand{\H}{\mc{H}}
 \newcommand{\Bc}{\mathcal{B}}

\newcommand{\TT}{\mathfrak{T}}
\newcommand{\MM}{\mathfrak{M}}
\newcommand{\LL}{\mathfrak{L}}
\newcommand{\KK}{\mathfrak{K}}
\newcommand{\GG}{\mathfrak{G}}
\newcommand{\HH}{\mathfrak{H}}

 \renewcommand{\d}{\text{\rm d}}
 
  \renewcommand{\SS}{\mathbb{S}}
 
  \newcommand{\z}{{\bf z}}
    \newcommand{\x}{{\bf x}}
      \newcommand{\y}{{\bf y}}
        \renewcommand{\t}{{\bf t}}
         \renewcommand{\v}{{\bf v}}
    \newcommand{\xxi}{{\bf \xi}}      

 \newcommand{\ov}{\overline}

 \newcommand{\uc}{\partial \mathbb{D}}
\newcommand{\ud}{\mathbb{D}}
\newcommand{\dbpn}{\mathcal{H}_n(P)}

\newcommand{\newr}[1]{\textcolor{red}{#1}}
\newcommand{\new}[1]{\textcolor{black}{#1}}

 \def\today{\ifcase\month\or
  January\or February\or March\or April\or May\or June\or
  July\or August\or September\or October\or November\or December\fi
  \space\number\day, \number\year}

\title[Extremal functions]{Extremal functions in de Branges \\
and Euclidean spaces}
\author[Carneiro and Littmann]{Emanuel Carneiro and  Friedrich Littmann}

\date{\today}

\address{IMPA - Instituto Nacional de Matem\'{a}tica Pura e Aplicada, Estrada Dona Castorina, 110, Rio de Janeiro, Brazil 22460-320.}
\email{carneiro@impa.br}
\address{Department of Mathematics, North Dakota State University, Fargo, ND 58105-5075.}
\email{friedrich.littmann@ndsu.edu}

\begin{abstract}
In this work we obtain optimal majorants and minorants of exponential type for a wide class of radial functions on $\R^N$.  These extremal functions minimize the $L^1(\R^N, |x|^{2\nu + 2 - N}\dx)$-distance to the original function, where $\nu >-1$ is a free parameter. To achieve this result we develop new interpolation tools to solve an associated extremal problem for the exponential function $\mc{F}_{\lambda}(x) = e^{-\lambda|x|}$, where $\lambda >0$, in the general framework of de Branges spaces of entire functions. We then specialize the construction to a particular family of homogeneous de Branges spaces to approach the multidimensional Euclidean case. Finally, we extend the result from the exponential function to a class of subordinated radial functions via integration on the parameter $\lambda >0$ against suitable measures. Applications of the results presented here include multidimensional versions of Hilbert-type inequalities, extremal one-sided approximations by trigonometric polynomials for a class of even periodic functions and extremal one-sided approximations by polynomials for a class of functions on the sphere $\SS^{N-1}$ with an axis of symmetry.
\end{abstract}

\maketitle

\numberwithin{equation}{section}

\section{Introduction}

In the remarkable work \cite{HV}, Holt and Vaaler constructed majorants and minorants of prescribed exponential type for characteristic functions of Euclidean balls in $\R^N$, in a way to minimize the $L^1(\R^N)$-error. This was the first time that the methods of Beurling and Selberg were extended to a multidimensional setting (see \S 1.2.1 below for a brief historical account on this problem and its applications). Their main insight was to reformulate and solve the extremal problem for the signum function within the framework of Hilbert spaces of entire functions as developed by L. de Branges \cite{B}. The purpose of the present work is to extend this theory to a wide class of radial functions on $\R^N$, by further exploiting the connection with de Branges spaces, and provide some applications. The central role that was played by the signum function in \cite{HV} is now played by the exponential function $\mc{F}_{\lambda}(x) = e^{-\lambda|x|}$, where $\lambda >0$.

\subsection{An extremal problem in de Branges spaces} We briefly review the basics of de Branges' theory of Hilbert spaces of entire functions \cite{B}.  A function $F$ analytic in the open upper half plane 
$$\U = \{z \in \C;\ \im(z) >0\}$$ 
has {\it bounded type} if it can be written as the quotient of two functions that are analytic and bounded in $\U$. If $F$ has bounded type in $\U$ then, according to \cite[Theorems 9 and 10]{B}, we have
\begin{equation*}
\limsup_{y \to \infty} \, y^{-1}\log|F(iy)| = v(F) <\infty.
\end{equation*}
The number $v(F)$ is called the {\it mean type} of $F$. We say that an entire function $F:\C \to \C$, not identically zero, has {\it exponential type} if 
\begin{equation}\label{Intro_exp_type}
\limsup_{|z|\to \infty} |z|^{-1}\log|F(z)| = \tau(F) <\infty.
\end{equation}
In this case, the nonnegative number $\tau(F)$ is called the {\it exponential type of $F$}. If $F:\C \to \C$ is entire we define $F^*:\C \to \C$ by $F^*(z) = \overline{F(\overline{z})}$. We say that $F$ is {\it real entire} if $F$ restricted to $\R$ is real valued.

\smallskip

A {\it Hermite-Biehler function} $E:\C \to \C$ is an entire function that satisfies the inequality
\begin{align}\label{HB-condition}
|E^*(z)|<|E(z)|
\end{align}
for all $z\in\U$. We define the {\it de Branges space} $\H(E)$ to be the space of entire functions $F:\C \to \C$ such that
\begin{equation*}
\|F\|_E^2 := \int_{-\infty}^\infty |F(x)|^{2} \, |E(x)|^{-2} \, \dx <\infty\,,
\end{equation*}
and such that $F/E$ and $F^*/E$ have bounded type and nonpositive mean type in $\U$. This is a Hilbert space with respect to the inner product 
\begin{equation*}
\langle F, G \rangle_E := \int_{-\infty}^\infty F(x) \,\overline{G(x)} \, |E(x)|^{-2} \, \dx.
\end{equation*}
The Hilbert space $\H(E)$ has the special property that, for each $w \in \C$, the map $F \mapsto F(w)$ is a continuous linear functional on $\H(E)$. Therefore, there exists a function $z \mapsto K(w,z)$ in $\H(E)$ such that 
\begin{equation}\label{Intro_rep_pro}
F(w) = \langle F, K(w,\cdot) \rangle_E\,.
\end{equation} 
The function $K(w,z)$ is called the {\it reproducing kernel} of $\H(E)$. If we write
\begin{equation}\label{Intro_def_A_B}
A(z) := \frac12 \big\{E(z) + E^*(z)\big\} \ \ \ {\rm and}  \ \ \ B(z) := \frac{i}{2}\big\{E(z) - E^*(z)\big\},
\end{equation}
then $A$ and $B$ are real entire functions and $E(z) = A(z) -iB(z)$. The reproducing kernel is then given by \cite[Theorem 19]{B}
\begin{equation*}
\pi (z - \ov{w})K(w,z) = B(z)A(\ov{w}) - A(z)B(\ov{w}),
\end{equation*}
or alternatively by
\begin{equation}\label{Intro_Def_K_0}
2\pi i (\ov{w}-z)K(w,z) = E(z)E^*(\ov{w}) - E^*(z)E(\ov{w}). 
\end{equation}
When $z = \ov{w}$ we have
\begin{equation}\label{Intro_Def_K}
\pi K(\ov{z}, z) = B'(z)A(z) - A'(z)B(z).
\end{equation}
Observe from \eqref{Intro_rep_pro} that 
\begin{align*}
0 \leq \|K(w, \cdot)\|_E^2 = \langle K(w, \cdot), K(w, \cdot) \rangle_E = K(w,w),
\end{align*}
and it is not hard to show (see \cite[Lemma 11]{HV}) that $K(w,w)=0$ if and only if $w \in\R$ and $E(w) = 0$.

\smallskip

We now consider Hermite-Biehler functions $E$ satisfying the following properties:
\begin{enumerate}
\item[(P1)] $E$ has bounded type in $\U$;
\smallskip
\item[(P2)] $E$ has no real zeros; 
\smallskip
\item[(P3)] $z \mapsto E(iz)$ is a real entire function;
\smallskip
\item[(P4)] $A, B \notin \H(E)$. 
\end{enumerate} 
\smallskip
By Krein's theorem (reviewed in Lemma \ref{theorem-krein} below) we see that if $E$ satisfies (P1), then $E$ has exponential type and $\tau(E) = v(E)$. The main result of this paper is the solution of the following extremal problem.

\begin{theorem}\label{Intro_Thm1}
Let $\lambda>0$. Let $E$ be a Hermite-Biehler function satisfying properties {\rm (P1) \!-\! (P4)} above. Assume also that
\begin{equation*}
\int_{-\infty}^{\infty} e^{-\lambda |x|}\, |E(x)|^{-2}\, \dx < \infty.
\end{equation*}
The following properties hold:
\smallskip
\begin{enumerate}
\item[(i)] If $L:\C \to \C$ is an entire function of exponential type at most $2 \tau(E)$ such that $L(x) \leq e^{-\lambda |x|}$ for all $x \in \R$ then 
\begin{equation}\label{Intro-L-integral}
\int_{-\infty}^{\infty} L(x) \,|E(x)|^{-2}\,\dx \leq \sum_{A(\xi)=0} \frac{e^{-\lambda|\xi|}}{K(\xi,\xi)},
\end{equation}
where the sum on the right-hand side of \eqref{Intro-L-integral} is finite. Moreover, there exists an entire function $z\mapsto L(A^2, \lambda, z)$ of exponential type at most $2 \tau(E)$ such that $L(A^2, \lambda, x) \leq e^{-\lambda |x|}$ for all $x \in \R$ and equality in \eqref{Intro-L-integral} holds.
\smallskip
\item[(ii)] If $M:\C \to \C$ is an entire function of exponential type at most $2 \tau(E)$ such that $M(x) \geq e^{-\lambda |x|}$ for all $x \in \R$ then 
\begin{equation}\label{Intro-M-integral}
\int_{-\infty}^{\infty} M(x) \,|E(x)|^{-2}\,\dx \geq \sum_{B(\xi)=0} \frac{e^{-\lambda|\xi|}}{K(\xi,\xi)},
\end{equation}
where the sum on the right-hand side of \eqref{Intro-M-integral} is finite. Moreover, there exists an entire function $z\mapsto M(B^2, \lambda, z)$ of exponential type at most $2 \tau(E)$ such that $M(B^2, \lambda, x) \geq e^{-\lambda |x|}$ for all $x \in \R$ and equality in \eqref{Intro-M-integral} holds.
\end{enumerate}
\end{theorem}
The proof of this result will be carried out in Sections \ref{polya} and \ref{debranges-section}. We will make use of a precise qualitative description of the representations of reciprocals of Laguerre-P\'{o}lya functions as Laplace transforms \cite[Chapters II to V]{HW}. This will ultimately provide the interpolation tools to generate the extremal minorants and majorants we seek. 

\smallskip

The extremal functions $z\mapsto L(A^2, \lambda, z)$ and $z\mapsto M(B^2, \lambda, z)$ in Theorem \ref{Intro_Thm1} depend implicitly on $E$, for they actually depend only on the functions $A^2$ and $B^2$, respectively. Our particular choice of notation will be clarified in Section \ref{polya}. Under hypotheses (P1) - (P3) we shall see that $A^2$ and $B^2$ are even functions in the Laguerre-P\'{o}lya class.

\subsection{Multidimensional extremal problems} 

\subsubsection{Preliminaries} Our notation and terminology for Fourier analysis is standard and follows \cite{SW}. From now on we shall denote vectors in $\R^N$ or $\C^N$ with bold font (e.g. $\x$, $\y$, $\z$) and numbers in $\R$ or $\C$ with regular font (e.g. $x,y,z$). For $\z = (z_1, z_2, \ldots, z_N) \in \C^N$ we let $|\cdot|$ denote the usual Hermitian norm $|\z| = (|z_1|^2 + \ldots + |z_N|^2)^{1/2}$ and define a second norm $\|\cdot\|$ by 
$$\|\z\| = \sup \left\{ \left|\sum_{n=1}^N z_n\,t_n\right|; \ \t \in \R^N \ {\rm and} \ |\t|\leq 1\right\}.$$
Let $F: \C^N \to \C$ be an entire function of $N$ complex variables which is not identically zero. We say that $F$ has {\it exponential type} if
$$\limsup_{\|\z\|\to \infty} \|\z\|^{-1}\,\log|F(\z)| =: \tau(F) < \infty.$$
When $N=1$ this is the classical definition of exponential type given in \eqref{Intro_exp_type} and, when $N \geq 2$, our definition is a particular case of a more general definition of exponential type with respect to a compact, convex and symmetric set $K \subset \R^N$ (cf. \cite[pp. 111-112]{SW}). In our case this convex set $K$ is the unit Euclidean ball.

\smallskip

Let $\nu > -1$ and $\delta > 0$ be given parameters. Given a function $\mc{F}:\R^N \to \R$ we address here the problem of finding entire functions $\mc {L}:\C^N \to \C$ and $\mc{M}:\C^N \to \C$ such that 
\begin{enumerate}
\item[(i)] $\mc{L}$ and $\mc{M}$ have exponential type at most $\delta$.
\smallskip
\item[(ii)] $\mc{L}$ and $\mc{M}$ are real valued on $\R^N$ and
\begin{equation*}\label{Intro_EP1}
\mc{L}(\x) \leq \mc{F}(\x) \leq \mc{M}(\x)
\end{equation*}
for all $\x \in \R^N$.
\smallskip
\item[(iii)] Subject to (i) and (ii), the values of the integrals
\begin{equation*}\label{Intro_EP2}
\int_{\R^N} \big\{{\mc M}(\x) - \mc{F}(\x)\big\}\,|\x|^{2\nu + 2 - N}\,\d\x 
\end{equation*}
and 
\begin{equation*}\label{Intro_EP3}
\int_{\R^N} \big\{\mc{F}(\x) - {\mc L}(\x)\big\}\,|\x|^{2\nu + 2 - N}\,\d\x
\end{equation*}
are minimized.
\end{enumerate}
\smallskip

In the setting $N=1$ and $\nu = -1/2$ this is a very classical problem in harmonic analysis. It was introduced by A. Beurling (unpublished work) in the late 1930's for $f(x) = \sgn (x)$ in connection with bounds for almost periodic functions. Later, with the observation that $\chi_{[a,b]}(x) = \tfrac{1}{2}\{\sgn(x-a) + \sgn(b-x)\}$, A. Selberg constructed majorants and minorants of exponential type for characteristic functions of intervals, a simple yet very useful tool for number theoretical applications, cf.\ \cite{BMV, Ga, GG, S1, S2}. The survey \cite{V} by J. D. Vaaler is an excellent source for a historical perspective on these extremal functions and their early applications. Other interesting classical works related to this theory include \cite{GV, K2, M, M2, Na}. Among the more recent works we highlight \cite{Car, CV3, Gan, GL, ILS, L1, L3, Lu}.

\smallskip

The unweighted one-dimensional extremal problem is relatively well understood. Besides the many applications listed in the references above, there are general frameworks to construct optimal majorants and minorants of exponential type for certain classes of even, odd and truncated functions \cite{CL, CL2, CLV, CV2}. These frameworks include, for instance, the functions $f_{a,b}(x) = \log\big((x^2 + b^2)/(x^2 + a^2)\big)$, where $0 \leq a < b$\,; $g(x) = \arctan(1/x) - x/(1+x^2)$ and $h(x) = 1 - x\arctan(1/x)$. The extremal functions associated to $f_{a,b}$, $g$ and $h$ have been recently used in connection to the theory of the Riemann zeta-function to obtain improved bounds for $\zeta(s)$ in the critical strip \cite{CC,CS}, and improved bounds for the argument function $S(t)$ and its antiderivative $S_1(t)$ on the critical line \cite{CCM}, all under the Riemann hypothesis. The machinery of de Branges spaces, that shall be explored in this paper,  has also proven useful in the recent works  \cite{CCLM, L4} in connection to extremals for characteristic functions of intervals and bounds for the pair correlation of zeros of the Riemann zeta-function.

\smallskip

As pointed out in the beginning of this introduction, in the multidimensional setting the picture changes considerably, as there is in the literature only the previous work of Holt and Vaaler \cite{HV} for characteristic functions of Euclidean balls addressing this problem. A periodic analogue of this work with applications to Erd\"{o}s-Tur\'{a}n inequalities was considered by Li and Vaaler in \cite{LV}. 

\subsubsection{Base case: the exponential function} We shall use our Theorem \ref{Intro_Thm1} to obtain the solution of the multidimensional Beurling-Selberg extremal problem posed above for a wide class of radial functions. The first step in this direction will be to specialize the general form of Theorem \ref{Intro_Thm1} to a particular family of homogeneous de Branges spaces, suitable to treat the power weights $|\x|^{2\nu +2 - N}$. We now present the basic elements associated to these spaces. Further details will be presented in Section \ref{Hom_spaces}.

\smallskip

For $\nu > -1$ let $A_{\nu}:\C \to \C$ and $B_{\nu}:\C \to \C$ be real entire functions defined by 
\begin{equation}\label{Intro_A_nu}
A_{\nu}(z) = \sum_{n=0}^{\infty} \frac{(-1)^n \big(\tfrac12 z\big)^{2n}}{n!(\nu +1)(\nu +2)\ldots(\nu+n)}
\end{equation}
and
\begin{equation}\label{Intro_B_nu}
B_{\nu}(z) = \sum_{n=0}^{\infty} \frac{(-1)^n \big(\tfrac12 z\big)^{2n+1}}{n!(\nu +1)(\nu +2)\ldots(\nu+n+1)}.
\end{equation}
These functions are related to the classical Bessel functions by 
\begin{align}
A_{\nu}(z) &= \Gamma(\nu +1) \left(\tfrac12 z\right)^{-\nu} J_{\nu}(z),\label{Intro_Bessel1}\\
B_{\nu}(z) & = \Gamma(\nu +1) \left(\tfrac12 z\right)^{-\nu} J_{\nu+1}(z).\label{Intro_Bessel2}
\end{align}
The functions $A_{\nu}$ and $B_{\nu}$ are the suitable generalizations of $\cos z$ and $\sin z$ for our purposes (note that $A_{-1/2}(z) = \cos z$ and $B_{-1/2}(z) = \sin z$). Note also that both $A_{\nu}$ and $B_{\nu}$ have only real, simple zeros and have no common zeros. Furthermore, they satisfy the following system of differential equations 
\begin{equation}\label{Intro_Diff_Eqs}
A_{\nu}'(z) = - B_{\nu}(z) \ \ \ {\rm and} \ \ \ B_{\nu}'(z) = A_{\nu}(z) - (2\nu +1)z^{-1}B_{\nu}(z),
\end{equation}
which can be seen directly from \eqref{Intro_A_nu} and \eqref{Intro_B_nu}. We shall see that the entire function 
$$E_{\nu}(z) := A_{\nu}(z) - iB_{\nu}(z)$$
 is a Hermite-Biehler function that satisfies properties (P1) \!-\! (P4) listed above, and therefore the general machinery of Theorem \ref{Intro_Thm1} will be available. We denote by $K_{\nu}$ the reproducing kernel of the space $\mc{H}(E_{\nu})$.

\smallskip

For $\delta >0$ and $N \in \N$ we denote by $\E_{\delta}^{N}$ the set of all entire functions $F:\C^N \to \C$ of exponential type at most $\delta$. Given a real valued function $\mc{F}:\R^N \to \R$ we write
\begin{align*}
\E_{\delta}^{N-} (\mc{F}) &:= \big\{\mc{L}\in \E_{\delta}^N;\ \mc{L}(\x) \leq \mc{F}(\x), \forall \x \in \R^N \big\},\\[0.5em]
\E_{\delta}^{N+}(\mc{F}) &:= \big\{\mc{M} \in \E_{\delta}^N;\ \mc{M}(\x) \geq \mc{F}(\x), \forall \x \in \R^N \big\}.\\[-1em]
\end{align*}
We are now able to state our second result.

\begin{theorem}\label{Thm1}
Let $\lambda >0$ and $\mc{F}_{\lambda}(\x) = e^{-\lambda |\x|}$. Let
\begin{align*}
U_{\nu}^{N-}(\delta, \lambda)  & = \inf\left\{ \int_{\R^N} \left\{e^{-\lambda|\x|} - \mc{L}(\x)\right\}\,|\x|^{2\nu+2 - N}\,\dxx; \  \mc{L} \in \E_{\delta}^{N-} (\mc{F}_{\lambda}) \right\},\\
U_{\nu}^{N+}(\delta, \lambda)  & = \inf\left\{ \int_{\R^N} \left\{\mc{M}(\x) - e^{-\lambda|\x|}\right\}\,|\x|^{2\nu+2 - N}\,\dxx; \  \mc{M} \in \E_{\delta}^{N+} (\mc{F}_{\lambda}) \right\}.
\end{align*}
The following properties hold:
\begin{enumerate}
\item[(i)] If $\kappa > 0$ then $U_{\nu}^{N\pm}(\delta, \lambda) = \kappa^{2\nu +2} \,U_{\nu}^{N\pm}(\kappa\delta, \kappa\lambda)$.
\smallskip
\item[(ii)] For any $\delta >0$ we have
\begin{equation*}\label{Intro_rel_N_1}
U_{\nu}^{N\pm}(\delta, \lambda) = \tfrac12\, \omega_{N-1}\, U_{\nu}^{1\pm}(\delta, \lambda),
\end{equation*}
where $\omega_{N-1} = 2 \pi^{N/2} \,\Gamma(N/2)^{-1}$ is the surface area of the unit sphere in $\R^N$.
\smallskip
\item[(iii)] When $N=1$ and $\delta = 2$ we compute explicitly:
\begin{align}
U_{\nu}^{1-}(2, \lambda)&= \frac{2 \,\Gamma(2\nu +2)}{\lambda^{2\nu +2}}- \sum_{A_{\nu}(\xi)=0}\frac{e^{-\lambda|\xi|}}{c_{\nu}\,K_{\nu}(\xi, \xi)}, \label{Intro_value_min}\\
U_{\nu}^{1+}(2, \lambda)& = \sum_{B_{\nu}(\xi)=0}\frac{e^{-\lambda|\xi|}}{c_{\nu}\,K_{\nu}(\xi, \xi)}\, - \, \frac{2 \,\Gamma(2\nu +2)}{\lambda^{2\nu +2}},\label{Intro_value_max}
\end{align}
where $c_{\nu} = \pi\, 2^{-2\nu -1}\, \Gamma(\nu +1)^{-2}.$

\smallskip

\item[(iv)] There exists a pair of real entire functions of $N$ complex variables $($which are radial when restricted to $\R^N$$)$ $\z \mapsto \mc{L}_{\nu}(\delta, \lambda, \z) \in \E_{\delta}^{N-} (\mc{F}_{\lambda}) $ and $\z \mapsto \mc{M}_{\nu}(\delta, \lambda, \z) \in \E_{\delta}^{N+} (\mc{F}_{\lambda})$ such that
\begin{equation*}
U_{\nu}^{N-}(\delta, \lambda) = \int_{\R^N} \left\{e^{-\lambda|\x|} - \mc{L}_{\nu}(\delta, \lambda, \x) \right\}\,|\x|^{2\nu+2-N}\,\dxx
\end{equation*}
and
\begin{equation*}
U_{\nu}^{N+}(\delta, \lambda) = \int_{\R^N} \left\{\mc{M}_{\nu}(\delta, \lambda, \x) - e^{-\lambda|\x|}  \right\}\,|\x|^{2\nu+2-N}\,\d\x.
\end{equation*}
\end{enumerate}
\end{theorem}

It is clear from parts (i) and (ii) of Theorem \ref{Thm1} that we can calculate $U_{\nu}^{N\pm}(\delta, \lambda)$ for any $N \geq 1$ and $\delta >0$ from the explicit formulas for the case $N=1$ and $\delta = 2$ presented in part (iii). Theorem \ref{Thm1} provides thus a complete description of the situation. Its proof will be presented in Section  \ref{Hom_spaces}.
 
\smallskip

For general $\nu$ the qualitative behavior of the zeros of $A_{\nu}$ and $B_{\nu}$ is not so different from that of the zeros of $\cos z$ and $\sin z$, and we have good asymptotics for the sums over the zeros in \eqref{Intro_value_min} and \eqref{Intro_value_max}. In fact, note that $A_{\nu}$ is an even function with $A_{\nu}(0)\neq 0$, while $B_{\nu}$ is an odd function with a simple zero at the origin and, from \eqref{Intro_Bessel1} and \eqref{Intro_Bessel2}, the positive zeros of $A_{\nu}$ and $B_{\nu}$ are the positive zeros of $J_{\nu}$ and $J_{\nu+1}$, respectively. From \cite[Section 15.4]{W} we note that, for large $m$, the Bessel function $J_\nu(x)$ has exactly $m$ zeros in $\big(0,m\pi+\tfrac{\pi\nu}{2}+\tfrac{\pi}{4}\big]$. Also, from \eqref{Intro_Def_K}, \eqref{Intro_Bessel1}, \eqref{Intro_Bessel2} and \eqref{Intro_Diff_Eqs} we find, for $\xi >0$, that
\begin{align}\label{Intro_asymp1}
K_{\nu}(\xi, \xi) = 2^{-1}\, c_{\nu}^{-1}\, \xi^{-2\nu -1} \Big\{ \xi J_{\nu}(\xi)^2 + \xi J_{\nu+1}(\xi)^2 - \!(2\nu+1)\,J_{\nu}(\xi)\,J_{\nu+1}(\xi)\Big\}.
\end{align}
From the well known asymptotics for Bessel functions
\begin{equation}\label{Asymptotic_Bessel_functions}
J_{\nu}(\xi) =  \sqrt{\frac{2}{\pi \xi}} \,  \cos\left(\xi - \tfrac{\pi \nu}{2} - \tfrac{\pi}{4}\right) + O_{\nu}\big(\xi^{-{3/2}}\big)
\end{equation}
as $\xi \to \infty$, we have
\begin{equation}\label{Intro_asymp2}
\lim_{\xi \to \infty} \Big\{ \xi J_{\nu}(\xi)^2 + \xi J_{\nu+1}(\xi)^2 - \!(2\nu+1)\,J_{\nu}(\xi)\,J_{\nu+1}(\xi)\Big\} = 2 \pi^{-1},
\end{equation}
and then it follows from \eqref {Intro_value_min} - \eqref{Intro_asymp2} that 
\begin{align*}
\big|\,U_{\nu}^{1-}(2, \lambda) \big| & \ll_{\nu} \frac{1}{\lambda^{2\nu +2}} + \sum_{n=1}^{\infty} n^{2\nu+1}\,e^{-\lambda j_{\nu,n}} \\
& \ll_{\nu} \frac{1}{\lambda^{2\nu +2}}\,
\end{align*}
as $\lambda \to \infty$, where $\{j_{\nu,n}\}_{n=1}^{\infty}$ is the sequence of positive zeros of $J_{\nu}$ (in this work we use Vinogradov's notation $f \ll g$ (or $f = O(g)$) to mean that $|f|$ is less than or equal to a constant times $|g|$, and indicate the parameters of dependence of the constant in the subscript). Analogously, from \eqref {Intro_value_max} we find
\begin{align*}
\big|U_{\nu}^{1+}(2,\lambda) \big|& \ll_{\nu} \frac{1}{c_{\nu} K_{\nu}(0,0)} +  \sum_{n=1}^{\infty} n^{2\nu+1}\,e^{-\lambda j_{\nu+1,n}}  +  \frac{1}{\lambda^{2\nu +2}}\\
&  \ll_{\nu} 1
\end{align*}
as $\lambda \to \infty$. To complete the asymptotic description we verify in the beginning of Section \ref{Sec_Radial_Functions} that 
\begin{equation}\label{Intro_open_asymp}
\big|U_{\nu}^{1\pm}(2,\lambda) \big| \ll_{\nu}  \lambda
\end{equation}
as $\lambda \to 0$. 

\subsubsection{Extension to a class of radial functions} Motivated by the asymptotic description of \eqref {Intro_value_min} and \eqref {Intro_value_max} presented above, we consider two classes of nonnegative Borel measures $\mu$ on $(0,\infty)$ for a given parameter $\nu>-1$. For the minorant problem we consider the class of measures $\mu$ satisfying the condition
\begin{equation}\label{Intro_mu_1}
\int_{0}^{\infty} \frac{\lambda}{1 + \lambda^{2\nu + 3}}\,\dmu < \infty,
\end{equation}
whereas for the majorant problem we consider the class of measures $\mu$ satisfying the more restrictive condition
\begin{equation}\label{Intro_mu_2}
\int_{0}^{\infty} \frac{\lambda}{1 + \lambda}\,\dmu < \infty.
\end{equation}
For $\mu$ satisfying \eqref{Intro_mu_1} or \eqref{Intro_mu_2} we define the radial function $\mc{G}_{\mu}: \R^N \to \R \cup \{\infty\}$ by 
\begin{equation}\label{Intro_def_g_mu}
\mc{G}_{\mu}(\x) = \int_0^{\infty}\left\{ e^{-\lambda |\x|} - e^{-\lambda}\right\}\dmu.
\end{equation}
Observe that we might have $\mc{G}_{\mu}(0) = +\infty$ in case $\mu$ satisfies \eqref{Intro_mu_1} but not \eqref{Intro_mu_2}, and in this situation only the minorant problem will be well-posed. Notice also the introduction of the term $-e^{-\lambda}$ to generate a decay of order $\lambda$ at the origin and make use of the full class of measures satisfying \eqref{Intro_mu_1} or \eqref{Intro_mu_2}. This idea is reminiscent of \cite{CV2, CV3}. For $\kappa >0$ we define a new measure $\mu_{\kappa}$ on the Borel subsets $E \subseteq (0,\infty)$ by putting
\begin{align}\label{def_mu_kappa}
\mu_{\kappa}(E) := \mu(\kappa E),
\end{align}
where 
$$\kappa E = \{ \kappa x; \ x \in E\}.$$
Note that $\mu_{\kappa}$ satisfies \eqref{Intro_mu_1} or \eqref{Intro_mu_2} whenever $\mu$ does. We are now able to state our third result.

\begin{theorem}\label{Thm2}
Let $\mu$ be a nonnegative Borel measure on $(0,\infty)$ satisfying \eqref{Intro_mu_1} for the minorant problem or \eqref{Intro_mu_2} for the majorant problem. Let $\mc{G}_{\mu}$ be defined by \eqref{Intro_def_g_mu} and write
\begin{align*}
U_{\nu}^{N-}(\delta, \mu)  & = \inf\left\{ \int_{\R^N} \big\{\mc{G}_{\mu}(\x) - \mc{L}(\x)\big\}\,|\x|^{2\nu+2-N}\,\dxx; \  \mc{L} \in \E_{\delta}^{N-} (\mc{G}_{\mu}) \right\},\\
U_{\nu}^{N+}(\delta, \mu) & = \inf\left\{ \int_{\R^N} \big\{\mc{M}(\x) - \mc{G}_{\mu}(\x)\big\}\,|\x|^{2\nu+2-N}\,\dxx; \  \mc{M} \in \E_{\delta}^{N+}(\mc{G}_{\mu})  \right\}.
\end{align*}
The following properties hold:
\smallskip
\begin{enumerate}
\item[(i)] If $0 < \kappa$ then $U_{\nu}^{N\pm}(\delta, \mu) = \kappa^{2\nu +2} \,U_{\nu}^{N\pm}(\kappa\delta, \mu_{\kappa^{-1}})$.
\smallskip
\item[(ii)] For any $\delta >0$ we have
\begin{equation*}
U_{\nu}^{N\pm}(\delta, \mu)= \int_0^{\infty} U_{\nu}^{N\pm}(\delta, \lambda)\,\dmu = \tfrac12 \, \omega_{N-1}\int_0^{\infty} U_{\nu}^{1\pm}(\delta, \lambda)\,\dmu.
\end{equation*}
where $\omega_{N-1} = 2 \pi^{N/2} \,\Gamma(N/2)^{-1}$.

\smallskip

\item[(iii)] There exists a pair of real entire functions of $N$ complex variables $($which are radial when restricted to $\R^N$$)$ $\z \mapsto \mc{L}_{\nu}(\delta, \mu, \z) \in \E_{\delta}^{N-} (\mc{G}_{\mu}) $ and $\z \mapsto \mc{M}_{\nu}(\delta, \mu, \z) \in \E_{\delta}^{N+} (\mc{G}_{\mu})$ such that
\begin{equation*}
U_{\nu}^{N-}(\delta, \mu) = \int_{\R^N} \big\{\mc{G}_{\mu}(\x) - \mc{L}_{\nu}(\delta, \mu, \x) \big\}\,|\x|^{2\nu+2-N}\,\dxx
\end{equation*}
and
\begin{equation*}
U_{\nu}^{N+}(\delta, \mu) = \int_{\R^N} \big\{\mc{M}_{\nu}(\delta, \mu, \x) - \mc{G}_{\mu}(\x)  \big\}\,|\x|^{2\nu+2-N}\,\dxx.
\end{equation*}
\end{enumerate}
\end{theorem}

The subtle change of notation from $U_{\nu}^{N\pm}(\delta, \lambda)$ in Theorem \ref{Thm1} to $U_{\nu}^{N\pm}(\delta, \mu)$ in Theorem \ref{Thm2} should cause no confusion since $\lambda>0$ will always represent a real valued parameter while $\mu$ will always represent a nonnegative Borel measure on $(0,\infty)$. From parts (i) and (ii) of Theorem \ref{Thm2} it is clear that we can calculate $U_{\nu}^{N\pm}(\delta, \mu)$ for any $N \geq 1$ and $\delta>0$ from the explicit formulas in part (iii) of Theorem \ref{Thm1} for the case $N=1$ and $\delta =2$.

\smallskip

Interesting examples of measures that we can consider are the power measures $\d\mu_{(\alpha)} (\lambda)= \lambda^{\alpha}\,\dl$. For $-2 < \alpha < 2\nu + 1$ these measures satisfy  \eqref{Intro_mu_1} and Theorem \ref{Thm2} provides an extremal minorant of exponential type for the family of radial functions
\begin{equation}\label{Intro_power_meas}
\mc{G}_{\mu_{(\alpha)}}(\x) =  \int_0^{\infty}\big\{ e^{-\lambda |\x|} - e^{-\lambda}\big\}\,\lambda^{\alpha}\,\dl = \Gamma(1+\alpha) \big\{|\x|^{-\alpha -1} - 1\big\}\,,
\end{equation}
when $\alpha \neq -1$, and
\begin{equation*}
\mc{G}_{\mu_{(-1)}}(\x) =  \int_0^{\infty}\big\{ e^{-\lambda |\x|} - e^{-\lambda}\big\}\,\lambda^{-1}\,\dl = -\log |\x|\,,
\end{equation*}
when $\alpha = -1$. If $-2 < \alpha < -1$ then these measures satisfy the more restrictive condition \eqref{Intro_mu_2} and Theorem \ref{Thm2} also provides an extremal majorant of exponential type for the corresponding functions in \eqref{Intro_power_meas}.

\smallskip

The proof of Theorem \ref{Thm2} will be carried out in Section \ref{Sec_Radial_Functions}. For optimality considerations we shall require a general decomposition of a function in $L^1\big(\R, |E_{\nu}(x)|^{-2}\,\dx\big)$ as a difference of two squares in $\mc{H}(E_{\nu})$.  We accomplish this in Section \ref{Diff_Squares}, where we prove two results of independent interest: (i) a general version of a theorem of Plancherel and P\'{o}lya \cite[\S 32]{PP} for the homogeneous spaces $\mc{H}(E_{\nu})$ (Theorem \ref{PP-in-HE}); and (ii) a general method for obtaining majorants and minorants of exponential type $\big($non-optimal with respect to the metric $L^1\big(\R, |E_{\nu}(x)|^{-2}\,\dx\big)$$\big)$ for functions that are locally of bounded variation (Theorem \ref{general-majorant}).

\subsection{Some applications} 

\subsubsection{Multidimensional Hilbert-type inequalities}

The solution of the extremal problem (in the classical case $N=1$ and $\nu = -1/2$) for a function $f:\R\to \R$ leads to upper and lower bounds for certain Hermitian forms involving the kernel $\widehat{f}$. When $f(x) = \sgn(x)$ this gives us the classical Hilbert's inequality (see \cite{V})
$$\left| \sum_{\stackrel{j,l=1}{j\neq l}}^{M} \frac{a_j \,\overline{a_l}}{j-l} \right| \leq \pi\sum_{j=1}^M |a_j|^2\,,$$
where $\{a_j\}_{j=1}^M$ is a sequence of complex numbers. This inequality was first proved by Hilbert (with the constant replaced by $2\pi$) and Schur (with the sharp constant $\pi$), and then generalized by Montgomery and Vaughan  \cite{MV} to replace the sequence of integers appearing in the denominator by a general sequence of well-spaced real numbers. For different kernels $\widehat{f}$, the corresponding Hilbert-type inequalities have appeared for instance in \cite{Car, CLV, CV2, GV, HV, L3}.

\smallskip

In our case, the extremal functions given by Theorems \ref{Thm1} and \ref{Thm2} provide generalizations of the Hilbert-type inequalities to a multidimensional context. We keep denoting 
$$\mc{F}_{\lambda} (\x) = e^{-\lambda |\x|},$$
and recall that its Fourier transform on $\R^N$ is the Poisson kernel \cite[pp. 6]{SW} given by 
\begin{equation}\label{Intro_FTF}
\widehat{\mc{F}}_{\lambda}(\y) = C_N \frac{\lambda}{(\lambda^2 + 4\pi^2 |\y|^2)^{(N+1)/2}}\,,
\end{equation}
with $C_N = 2^N \, \pi^{(N-1)/2}\, \Gamma(\tfrac{N+1}{2})$. For the next result we restrict ourselves to the case when $2\nu +2 - N$ is a nonnegative even integer, which we call $2r$. If a nonnegative Borel measure $\mu$ on $(0,\infty)$ satisfies \eqref{Intro_mu_1} or \eqref{Intro_mu_2}, by directly differentiating \eqref{Intro_FTF}, we observe that the function 
\begin{equation}\label{Intro_Q_a_y}
\mc{Q}_{\mu,r}(\y):=\big(4\pi^2\big)^{-r}\int_0^{\infty} (-\Delta)^{r} \widehat{\mc{F}}_{\lambda}(\y) \,\dmu
\end{equation}
is well-defined for $\y \neq0$, where $\Delta$ denotes the usual $N$-dimensional Laplacian.

\begin{theorem}\label{Thm4}
Let $\delta >0$ and $2\nu + 2 - N = 2r$ be a nonnegative even integer. Let $\{a_j\}_{j=1}^M$ be a sequence of complex numbers and $\{\y_j\}_{j=1}^M$ be a sequence of well-spaced vectors in $\R^N$, in the sense that $|\y_j - \y_l|\geq \delta$ for any $j \neq l$. The following propositions hold:
\smallskip
\begin{enumerate}
\item[(i)] If $\mu$ is a nonnegative Borel measure on $(0,\infty)$ satisfying \eqref{Intro_mu_1} and $\mc{Q}_{\mu,r}$ is defined as in \eqref{Intro_Q_a_y} then
\begin{equation*}
-U_{\nu}^{N-}(2\pi\delta, \mu)\,\sum_{j=1}^M |a_j|^2\leq \sum_{\stackrel{j,l=1}{j\neq l}}^{M} a_j\,\overline{a_l} \,\mc{Q}_{\mu,r}(\y_j - \y_l).
\end{equation*}
\item[(ii)] If $\mu$ is a nonnegative Borel measure on $(0,\infty)$ satisfying \eqref{Intro_mu_2} and $\mc{Q}_{\mu,r}$ is defined as in \eqref{Intro_Q_a_y} then
\begin{equation*}
\sum_{\stackrel{j,l=1}{j\neq l}}^{M} a_j\,\overline{a_l} \,\mc{Q}_{\mu,r}(\y_j - \y_l) \leq U_{\nu}^{N+}(2\pi\delta,\mu)\,\sum_{j=1}^M |a_j|^2.
\end{equation*}
\end{enumerate}
\end{theorem}

\smallskip

In particular we observe that when $r=0$ and $\d \mu_{(\alpha)}(\lambda) = \lambda^{\alpha}\,\dl$, with $-2 < \alpha < 2\nu + 1 = N-1$, we have 
$$\mc{Q}_{\mu_{(\alpha)},0}(\y) =  C_{\alpha}\, |\y|^{-N  + \alpha+1},$$
where $C_\alpha = \pi^{\alpha + 1 - \frac{N}{2}} \,\Gamma(\alpha +1)\, \Gamma\left(\frac{N - \alpha -1}{2}\right)\,\Gamma\left(\frac{\alpha +1}{2}\right)^{-1}$ (see \cite[Chapter V, Lemma 1]{St}). In these cases, the inequalities of Theorem \ref{Thm4} are related to the discrete Hardy-Littlewood-Sobolev inequalities (see \cite[pp. 288]{HPL}). Their one-dimensional versions have appeared in \cite{CLV,CV2}.

\subsubsection{Periodic analogues} In the last section of the paper we provide further applications of the interpolation theory developed in Section \ref{polya} to solve some extremal problems in a periodic setting. We address the following problems:

\smallskip

\noindent (i) Given a periodic function $f:\R/\Z \to \R$ and an even probability measure $\vartheta$ on $\R/\Z$, find a trigonometric polynomial of degree at most $n$ that majorizes/minorizes $f$ in a way to minimize the $L^1(\R/\Z, \d\vartheta)$-error; 

\smallskip

\noindent (ii) Given a function $F: \SS^{N-1} \to \R$ and a probability measure $\sigma$ on the sphere $\SS^{N-1}$, find a polynomial (in $N$ variables) of degree at most $n$ that majorizes/minorizes $F$ in a way to minimize the $L^1(\SS^{N-1}, \d\sigma)$-error.

\smallskip

In \cite{LV}, Li and Vaaler solved (i) for the sawtooth function 
$$\psi(x) = 
\left\{
\begin{array}{lc}
x - \lfloor x \rfloor - \hh&, \ {\rm if} \ x \notin \Z;\\
0&, \ {\rm if} \ x \in \Z;
\end{array}
\right.
$$ 
and for characteristic functions of intervals and (ii) for characteristic functions of spherical caps, all with respect to Jacobi measures. Here we extend this construction to a wide class of even functions in (i) (that includes for instance the function $\varphi(x) = -\log|2 \sin \pi x|$, the harmonic conjugate of the sawtooth function) and a wide class of functions with an axis of symmetry in (ii), all with respect to more general measures. To accomplish this, besides the general interpolation machinery of Section \ref{polya}, we shall use the theory of reproducing kernel Hilbert spaces of polynomials and the theory of orthogonal polynomials on the unit circle.

\subsubsection{Polynomial approximation} The general interpolation formulas allow the solution of the analogous problem for one-sided approximation of $x\mapsto e^{-\lambda|x|}$ (and a class of even functions by integrating the parameter $\lambda$) by polynomials on a finite symmetric interval. We do not include this here since this problem was solved already by R.\ Bojanic and R.\ DeVore \cite{BV} with different methods.


\section{Interpolation at zeros of Laguerre-P\'{o}lya functions}\label{polya}

\subsection{Laguerre-P\'{o}lya class and Laplace transforms}
An entire function $F:\C \to \C$ is said to be of {\it P\'{o}lya class} if:
\begin{enumerate}
\item[(i)] $F$ has no zeros in open upper half plane $\U$;
\smallskip
\item[(ii)] $|F^*(z)|\le |F(z)|$ for all $z \in \U$;
\smallskip
\item[(iii)] $y\mapsto |F(x+iy)|$ is a nondecreasing function of $y>0$ for each fixed $x\in \R$.
\smallskip
\end{enumerate}
Such entire functions can be characterized by their Hadamard factorization as in \cite[Theorem 7]{B}. Here we will be mainly interested in the subclass of real entire functions of P\'{o}lya class, called the {\it Laguerre-P\'{o}lya class}. The functions in the Laguerre-P\'{o}lya class have only real zeros and their Hadamard factorization takes the form
\begin{equation}\label{hadamardproduct}
F(z)=\frac{F^{(r)}(0)}{r!}\,z^r\, e^{-az^2+bz}\,\prod_{j=1}^\infty\Big(1-\frac{z}{x_j}\Big)e^{z/x_j},
\end{equation}
where $r\in\Z^{+}$, $a, b, x_j \in \R$, with $a \geq 0$, $x_j \neq 0$ and $\sum_{j=1}^\infty x_j^{-2}<\infty$. An obvious adjustment of notation includes the case when there is only a finite number of roots. These functions were first considered by E. Laguerre \cite{La} and later by G. P\'{o}lya \cite{P1}.

\smallskip

We say that a Laguerre-P\'{o}lya function $F$ has degree $\mc{N}$, with $0 \leq \mc{N} < \infty$, if $a=0$ in \eqref{hadamardproduct} and $F$ has exactly $\mc{N}$ zeros counted with multiplicity. Otherwise we set the degree of $F$ to be $\mc{N} = \infty$. Throughout this paper we shall denote the degree of $F$ by $\mc{N}(F)$ or simply by $\mc{N}$ when there is no ambiguity.

\smallskip

Let $F$ be a Laguerre-P\'{o}lya function of degree $\mc{N}\geq 2$ (here and later on such a condition will include the case $\mc{N}= \infty$). For $c \in \R$ with $F(c)\neq 0$ we define
\begin{equation}\label{gc-def}
g_c(t) = \frac{1}{2\pi i }\int_{c-i\infty}^{c+i\infty} \frac{e^{st}}{F(s)} \,\ds,
\end{equation}
where integration is understood to be a complex line integral along the vertical line $c+iy$ with $y\in\R$. Observe that the condition $\mc{N}(F) \geq 2$ implies that $1/|F(c +iy)| = O(|y|^{-2})$ as $|y| \to \infty$, and the integral above is in fact absolutely convergent. Moreover, if $c \in (\tau_1, \tau_2)$, where $(\tau_1,\tau_2) \subset \R$ is the largest open interval containing no zeros of $F$ (we allow $\tau_1,\tau_2\in \{\pm\infty\}$), using this basic growth estimate for $F$ and the residue theorem we see that $g_c = g_d$ for all $d \in  (\tau_1,\tau_2)$. A Fourier inversion shows that
\begin{equation}\label{gc-trafo}
\frac{1}{F(z)} = \int_{-\infty}^\infty g_c(t)\, e^{-zt}\, \,\dt
\end{equation}
in the strip $\tau_1<\re(z)<\tau_2$ (see Lemma \ref{growth in LP} below or \cite[Chapter V, Theorem 2.1]{HW} to justify the absolute convergence). It is precisely this expression for the reciprocal of $F$ as a Laplace transform that will play an important role in this work. 

\smallskip

If $\mc{N}(F) =1$ we also have functions $g_c$ that satisfy \eqref{gc-trafo} in the appropriate half planes, which can be verified directly. Let $\tau$ be the zero of $F$, written in the form \eqref{hadamardproduct}. If $\tau = 0$ then \eqref{gc-trafo} holds with
\begin{equation}\label{Def_g_c_N_1_0}
g_c(t) = \left\{
\begin{array}{ll}
\vspace{0.2cm}
F'(0)^{-1}\, \chi_{(b,\infty)}(t), & \ {\rm for}\ c>0;\\
-F'(0)^{-1}\, \chi_{(-\infty,b)}(t),& \ {\rm for}\ c<0.\\
\end{array}
\right.
\end{equation}
If $\tau \neq 0$ then \eqref{gc-trafo} holds with
\begin{equation}\label{Def_g_c_N_1_1}
g_c(t) = \left\{
\begin{array}{ll}
\vspace{0.2cm}
-\tau \, F(0)^{-1}\, e^{\tau(t - b) -1}\chi_{(b + \tau^{-1},\infty)}(t), & \ {\rm for}\ c>\tau;\\
\tau \, F(0)^{-1}\, e^{\tau(t - b) -1}\chi_{(-\infty,b + \tau^{-1})}(t),& \ {\rm for}\ c<\tau.\\
\end{array}
\right.
\end{equation}
If $\mc{N}(F) = 0$ then \eqref{gc-trafo} holds with 
\begin{equation}\label{Def_g_c_N_0}
g_c(t) = F(0)^{-1}\,\delta(t-b),
\end{equation}
for any $c \in \R$, where $\delta$ is the Dirac delta distribution.

\smallskip

In the next two lemmas we state the relevant properties of the frequency functions $g_c$ needed for this work. The first one describes the sign changes of $g_c$ and its derivatives, while the second one describes the asymptotic behavior of $g_c$ and its derivatives. These and many other interesting properties of the functions $g_c$ are well detailed in \cite[Chapters II to V]{HW}.

\begin{lemma}\label{g-sign-changes} 
Let $F$ be a Laguerre-P\'olya function of degree $\mc{N}\ge 2$ and let $g_c$ be defined by \eqref{gc-def}, where $c \in \R$ and $F(c) \neq 0$. The following propositions hold:
\begin{enumerate} 
\item [(i)] The function $g_c \in C^{\mc{N}-2}(\R)$ and is real valued.
 
 \smallskip
 
\item[(ii)] The function $g_c$ is of one sign, and its sign equals the sign of $F(c)$.

\smallskip

\item[(iii)] If $c=0$, the function $g_0^{(n)}$ has exactly $n$ sign changes for $n = 0,1,\ldots, \mc{N}-2$. Moreover, the zeros associated with the sign changes of $g_0^{(n)}$ $(n=1,2,\ldots,\mc{N}-3)$ are all simple zeros.
\end{enumerate}
\end{lemma}

\begin{proof} If $\mc{N}$ is finite we observe that $1/|F(c +iy)| = O\big(|y|^{-\mc{N}}\big)$ as $|y| \to \infty$. If $\mc{N} = \infty$, then $1/|F(c +iy)| = O(|y|^{-k})$ as $|y| \to \infty$, for all $k \in \N$ (for details see \cite[Chapter II, Theorem 6.3 and Chapter III, Theorem 5.3]{HW}). This allows us to differentiate $\mc{N}-2$ times inside the integral \eqref{gc-def}. Moreover, the fact that $F$ is real entire implies that $g_c$ is real valued. This establishes (i). 

\smallskip

The statement (iii) is \cite[Chapter IV, Theorem 5.1 and Theorem 5.3]{HW}. 

\smallskip

To show (ii) we let $c \in (\tau_1, \tau_2)$, where $(\tau_1,\tau_2) \subset \R$ is the largest open interval containing no zeros of $F$. If $0 \in (\tau_1,\tau_2)$ then $g_0 = g_c$ and (ii) follows from (iii) and \eqref{gc-trafo}. If $0 \notin (\tau_1,\tau_2)$ we consider the translation $F_c(z) := F(z+c)$, which is still a Laguerre-P\'{o}lya function. With this change we have
\begin{equation}\label{SecLP_shift}
g_c(t) = \frac{e^{ct}}{2\pi i } \int_{-i\infty}^{+i\infty} \frac{e^{st}}{F_c(s)} \,\ds\,,
\end{equation}
and again (ii) follows from (iii) since $e^{ct}$ is always positive.
\end{proof}

\begin{lemma}\label{growth in LP}
Let $F$ be a Laguerre-P\'olya function of degree $\mc{N}\ge 2$ and let $g_c$ be defined by \eqref{gc-def}, where $c \in \R$ and $F(c) \neq 0$. If $(\tau_1,\tau_2) \subset \R$ is the largest open interval containing no zeros of $F$ such that $c \in (\tau_1,\tau_2)$, and $0\leq n\le \mc{N}-2$, then there exist polynomials $p_n$ and $q_n$ such that
\begin{align}\label{gc-growth}
\big|g_c^{(n)}(t)\big| \le \begin{cases}
p_n(t) \,e^{\tau_1 t} &\text{ as }t\to +\infty,\\
q_n(t) \, e^{\tau_2 t} & \text{ as }t\to -\infty.
\end{cases} 
\end{align}
Note: When $\tau_2 = +\infty$ the meaning of \eqref{gc-growth} is that, for each $c < k < \infty$, there is actually a constant $C_{n,k}$ such that 
$$\big|g_c^{(n)}(t)\big| \leq C_{n,k}\, e^{k t}$$
as $t\to -\infty$. An analogous statement holds if $\tau_1 = -\infty$.
\end{lemma}

\begin{proof} If we can show the result for $c=0$, the general case follows by considering the shift given by \eqref{SecLP_shift}. Let us then focus on the case $c=0$.

\smallskip

If $\mc{N}$ is finite, then from Lemma \ref{g-sign-changes} we know that $g_0^{(n)}$ exists for $n\le \mc{N}-2$. The growth estimate follows from the explicit representations given by \cite[Chapter II, Theorem 8.2]{HW}, which come from the Laplace expansions of the partial fraction decomposition of $F$. 
\smallskip

If $\mc{N}=\infty$, \eqref{gc-growth} follows from the statement of \cite[Chapter V, Theorem 2.1]{HW}.
\end{proof}

\subsection{Interpolation formulas}

In the next series of propositions we present suitable entire functions that interpolate $\mc{F}_{\lambda}(x) = e^{-\lambda|x|}$ at the zeros of a given Laguerre-P\'{o}lya function. We shall exploit the representation of the inverse of a Laguerre-P\'{o}lya function as a Laplace transform given by \eqref{gc-trafo}, where we have a good knowledge of the qualitative properties of the frequency functions $g_c$ provided by Lemmas \ref{g-sign-changes} and \ref{growth in LP}. The material in this subsection generalizes \cite[Section 3]{GV}.

\smallskip

Given a Laguerre-P\'olya function $F$, we henceforth denote by $\alpha_F$ the smallest positive zero of $F$ (if no such zero exists we set $\alpha_F = +\infty$). Similarly, we denote by  $\beta_F$ the largest nonpositive zero of $F$ (if no such zero exists we set $\beta_F = -\infty$). 

\begin{proposition} Let $F$ be a Laguerre-P\'olya function. Let $g=g_{\alpha_F/2}$ and assume that $F(\alpha_F/2)>0$ $($in case $\alpha_F= +\infty$, let $g = g_1$ and assume $F(1) >0$$)$. Define, for $\lambda >0$,
\begin{align}
\mc{A}_1(F,\lambda,z)  &= F(z) \int_{-\infty}^0 g(w-\lambda)\, e^{-zw}\, \dw\ \ \text{\rm for }\,\re(z)<\alpha_F,\label{al-def-1}\\
\mc{A}_2(F,\lambda,z) &=  e^{-\lambda z} - F(z) \int_{0}^\infty g(w-\lambda)\, e^{-zw} \,\dw\ \ \text{ \rm for }\, \re(z)>\beta_F.\label{al-def-2}
\end{align}
Then $z\mapsto \mc{A}_1(F,\lambda,z)$ is analytic in $\re(z)<\alpha_F$, $z\mapsto \mc{A}_2(F,\lambda,z)$ is analytic in $\re(z)>\beta_F$, and these functions are restrictions of an entire function, which we will denote by $\mc{A}(F,\lambda,z)$. Moreover, there exists $c>0$ so that 
\begin{align}\label{al-growth}
|\mc{A}(F,\lambda,z)| \le c\,\big(1+|F(z)|\big)
\end{align}
for all $z\in\C$.
\end{proposition}

\begin{proof} If $\mc{N}(F) \geq 2$ we obtain from Lemma \ref{growth in LP} that the two integrals converge absolutely and define analytic functions in the stated half planes.  If $\mc{N}(F) = 1$ or $0$ the same holds from \eqref{Def_g_c_N_1_0}, \eqref{Def_g_c_N_1_1} and \eqref{Def_g_c_N_0}. Consider now $z \in \C$ with $\beta_F<\re(z)<\alpha_F$. From \eqref{gc-trafo} we have
\begin{align*}
e^{-\lambda z} = F(z) \int_{-\infty}^\infty g(w-\lambda)\, e^{-zw} \, \dw\,, 
\end{align*}
which implies that $\mc{A}_2(F, \lambda,z) = \mc{A}_1(F, \lambda,z)$ in the strip $\beta_F<\re(z)<\alpha_F$. Hence $z\mapsto \mc{A}_1(F, \lambda,z)$ and $z\mapsto \mc{A}_2(F, \lambda,z)$ are analytic continuations of each other and this defines the entire function $z \mapsto \mc{A}(F, \lambda, z)$. The integral representations for $\mc{A}$ imply, for $\re(z)\le \alpha_F/2$, that
\begin{align}\label{Sec_ILP_eq1}
\begin{split}
|\mc{A}(F, \lambda,z)|& \le |F(z)| \int_{-\infty}^0 g(w-\lambda) \,e^{-w\re(z) } \,\dw\\
&  \le |F(z)| \int_{-\infty}^0 g(w-\lambda) \,e^{-w \alpha_F/2} \,\dw\,,
\end{split}
\end{align}
while for $\re(z)\ge \alpha_F/2$ we have
\begin{equation}\label{Sec_ILP_eq2}
|\mc{A}(F, \lambda,z)|\le e^{-\lambda\alpha_F/2} +|F(z)| \int_0^\infty g(w-\lambda) \,e^{-w \alpha_F/2} \,\dw.
\end{equation}
Estimates \eqref{Sec_ILP_eq1} and \eqref{Sec_ILP_eq2} plainly verify \eqref{al-growth}. 
\end{proof}

\begin{proposition} Let $\lambda>0$. Let $F$ be an even Laguerre-P\'olya function such that $F(0)> 0$. Then the entire function $z\mapsto L(F,\lambda,z)$ defined by
\begin{align*}
 L(F,\lambda,z) &= \mc{A}(F,\lambda,z) + \mc{A}(F,\lambda,-z)
\end{align*}
satisfies
\begin{align}\label{minorant-ineq}
F(x)\left\{e^{-\lambda|x|}- L(F,\lambda,x)\right\}\ge 0
\end{align}
for all $x \in \R$ and
\begin{align}\label{minorant-interpolation}
L(F,\lambda,\xi) = e^{-\lambda|\xi|} 
\end{align}
for all $\xi \in \R$ with $F(\xi)=0$. 
\end{proposition}

\begin{proof} From the assumptions we have $g_{\alpha_F/2} = g_0$ and we call this function $g$. We first treat the case $\mc{N}(F) \geq 4$. Since $F$ is even, we obtain $F(iy)=F(-iy)$, and this implies that $g = g_0$ is even from \eqref{gc-def}. For $x<\alpha_F$, we use \eqref{al-def-1} for $\mc{A}(F,\lambda,z)$ and \eqref{al-def-2} for $\mc{A}(F,\lambda,-z)$ to obtain the integral representation
\begin{align}\label{ml-el-rep}
\begin{split}
L(F, \lambda,x) - e^{\lambda x} &= F(x)\!\int_{-\infty}^0 g(w-\lambda) \,e^{-xw} \,\dw -F(x) \!\int_0^\infty \!g(w-\lambda) \,e^{x w}\,\dw \\
&= F(x) \int_{-\infty}^0 \big\{g(w-\lambda) - g(w+\lambda)\big\} \,e^{-xw} \,\dw\,.
\end{split}
\end{align}
 
\noindent Since $\mc{N}(F) \geq 4$, $g'(t)$ exists for all real $t$. Since $F(0)>0$, the function $g=g_0$ decays exponentially as $t\to\pm\infty$ by Lemma \ref{growth in LP}. If $g$ had more than one local maximum, then $g'$ would have more than one sign change, which is not possible by Lemma \ref{g-sign-changes}. Since $g$ is even, it follows that $g$ is nondecreasing on $(-\infty,0)$ and nonincreasing on $(0,\infty)$. Hence, for $w <0$, 
\begin{align}\label{g-min-ineq}
g(w-\lambda) = g(-|w|-\lambda) \le g(-|w|+\lambda) = g(w+\lambda).
\end{align}
By inserting \eqref{g-min-ineq} into \eqref{ml-el-rep} we obtain \eqref{minorant-ineq} for $x\leq0$. By symmetry, \eqref{minorant-ineq} also follows for $x\geq0$. Identity \eqref{minorant-interpolation} follows from \eqref{ml-el-rep} for $\xi <0$ and, by symmetry, for $\xi >0$.

\smallskip

If $\mc{N}(F) = 2$ we must have
\begin{equation}\label{Sec2_eq1_NF2}
F(z) = C\left(1 - \frac{z^2}{\alpha^2}\right).
\end{equation} 
We proceed as above by observing that 
\begin{equation}\label{Sec2_eq1_g_t}
g(t) = \frac{\alpha}{2C}\,e^{-\alpha |t|}.
\end{equation}

\smallskip
If $\mc{N}(F) =0$ then $F$ is a constant and $L$ is identically zero.
\end{proof}

\begin{proposition} \label{Sec_ILP_Prop5}
Let $\lambda>0$. Let $F$ be an even Laguerre-P\'olya function that has a double zero at the origin. Let $g= g_{\alpha_F/2}$ and assume that $F(\alpha_F/2)>0$ $($in case $\alpha_F = +\infty$, let $g = g_1$ and assume $F(1) >0$$)$. Then the entire function $z\mapsto M(F,\lambda,z)$ defined by
\begin{align*}
M(F,\lambda,z) = \mc{A}(F,\lambda,z) + \mc{A}(F,\lambda,-z) +2g'(0) \frac{F(z)}{z^2}
\end{align*}
satisfies
\begin{align}\label{majorant-ineq}
F(x)\left\{M(F,\lambda,x) - e^{-\lambda|x|}\right\}\ge 0
\end{align}
for all $x \in \R$ and
\begin{align}\label{majorant-interpolation}
M(F,\lambda,\xi) = e^{-\lambda|\xi|}
\end{align}
for all $\xi \in \R$ with $F(\xi)=0$. 

\smallskip

\noindent Note: When $\mc{N}(F)= 2$, the function $g$ is not differentiable at the origin and we set $g'(0) := \tfrac12\big\{g'(0^-) + g'(0^+)\big\}$.
\end{proposition}

\begin{proof} Let us first deal with the case $\mc{N}(F)\geq 4$. For $x<0$, equations \eqref{al-def-1} and \eqref{al-def-2} imply
\begin{align}\label{Ml-el-rep}
\begin{split}\!
\!\!M(F,\lambda,x) -e^{-\lambda |x|} &= F(x) \int_{-\infty}^0 g(w-\lambda) \,e^{-xw}\, \dw \\
&\qquad \ \ \ - F(x)\int_0^\infty g(w-\lambda) \,e^{xw}\, \dw \, +\, 2g'(0)\frac{F(x)}{x^2}\\
&=F(x) \int_0^\infty \big\{g(-w-\lambda) -g(w-\lambda) + 2g'(0)w\big\}\,e^{xw}\, \dw.
\end{split}
\end{align}
The assumptions on $F$ imply that $z\mapsto z^{-2} F(z)$ is an even Laguerre-P\'olya function that is positive in an interval containing the origin. By multiplying both sides of \eqref{gc-trafo} by $z^2$ and integrating by parts twice using Lemma \ref{growth in LP}, it follows that $g''$ is the corresponding function whose Laplace transform represents $z^2 /F(z)$ in a strip containing the origin. Hence $g''$ is even, nonnegative and has exponential decay as $t\to\pm\infty$. Moreover, $g''$ is nondecreasing on $(-\infty,0)$ and nonincreasing on $(0,\infty)$ (if $\mc{N}\big(F(z)/z^2\big) \geq 4$ we can invoke Lemma \ref{g-sign-changes} to see this and if $\mc{N}\big(F(z)/z^2\big) =2$ we can verify directly by \eqref{Sec2_eq1_g_t}). Therefore $u\mapsto \{g'(0)  - g'(u)\}$ is odd and nonincreasing, and hence, for $w >0$, we have
\begin{align*}
2g'(0)w -\big\{g(w-\lambda) - g(-w-\lambda)\big\} &= \int_{-\lambda - w}^{-\lambda+w} \big\{g'(0) - g'(u)\big\} \,\du\\
&\ge \int_{-w}^w \big\{g'(0) - g'(u)\big\} \,\du \\
& =0.
\end{align*}
By inserting this into \eqref{Ml-el-rep} we obtain \eqref{majorant-ineq} for $x<0$. Therefore \eqref{majorant-ineq} also holds for $x>0$ by symmetry and for $x=0$ by continuity. Identity \eqref{majorant-interpolation} follows 
from \eqref{Ml-el-rep} for $\xi <0$ and, by symmetry, for $\xi >0$. 

\smallskip

To check it for $\xi=0$ define
\begin{align*}
h(\lambda,w) = g(-w-\lambda) - g(w-\lambda) + 2g'(0) w.
\end{align*}
We will denote differentiation of $h$ with respect to $w$ by $h'(\lambda,w)$. For $x<0$, we use integration by parts twice on the right-hand side of \eqref{Ml-el-rep}, noting that $h(\lambda,0)=0$ and $h'(\lambda,0) = 2g'(0) - 2g'(-\lambda)$, to obtain
\begin{align}\label{Ml-interpol-at0}
\begin{split}
M(F, \lambda,x) -e^{-\lambda |x|} &= F(x)\int_0^\infty h(\lambda,w) \,e^{x w}\, \dw\\
&=F(x)\left\{\frac{2g'(0) - 2g'(-\lambda)}{x^2}\right. \\
&\ \ \ \ \qquad \ \left. +\frac{1}{x^2} \int_0^\infty \!\! \big\{g''(-w-\lambda) -g''(w-\lambda)\big\}\, e^{x w}\,\dw\right\}.
\end{split}
\end{align}
We have already remarked that $g''(w)$ decays exponentially as $w\to \pm \infty$. From Lemma \ref{growth in LP} we also know that $g'(w)$ decays exponentially as $w\to -\infty$. We then obtain
\begin{align*}
\int_0^\infty  & \big\{   g''(-w -\lambda)   - g''(w-\lambda)\big\}\,\dw  \\
& = \int_{-\infty}^0 g''(w-\lambda)\, \dw - \int_0^\infty g''(w-\lambda)\, \dw \\
&= 2 \int_{-\infty}^0 g''(w-\lambda)\, \dw - \int_{-\infty}^\infty g''(w-\lambda)\, \dw \\
&= 2g'(-\lambda) -\int_{-\infty}^\infty g''(w-\lambda)\, \dw.
\end{align*}
Since $g''$ is nonnegative and even, the final integral equals $2g'(0)$. Hence, letting $x\to 0^-$ in \eqref{Ml-interpol-at0}, we obtain $M(F, \lambda,0) - 1 =0$, which finishes the proof of \eqref{majorant-interpolation}.

\smallskip

In the case $\mc{N}(F) =2$ we must have $F(z) = Cz^{2}$,
where $C>0$. Then $g(t) = C^{-1}\,t\,\chi_{(0, \infty)}(t)$ and 
$$g'(0) := \frac12\big\{g'(0^-) + g'(0^+)\big\} = \frac{1}{2C}.$$
The computations now can be done directly and we obtain
\begin{align}\label{Sec2_lastcase_M}
M(F,\lambda,x) &= 1
\end{align}
for all $x \in \R$. The result easily follows.
\end{proof}

The purpose of the next proposition is two-fold. Firstly, we want to establish that the functions $x \mapsto L(F,\lambda,x) - e^{-\lambda|x|}$ and $x \mapsto M(F,\lambda,x) - e^{-\lambda |x|}$ belong to $L^1(\R, |F(x)|^{-1}\,\dx)$ (this will be used in Lemma \ref{L1-estimate} below). Secondly, we want to display the dependence of such integrals on the parameter $\lambda$, specially as $\lambda \to 0$ (this will be used to show \eqref{Intro_open_asymp}).

\begin{proposition}\label{Prop10}
Let $\lambda >0$. Let $F$ be an even Laguerre-P\'olya function of degree $\mc{N}(F) \geq 4$.
\smallskip  
\begin{enumerate}
\item[(i)] If $F(0)\neq 0$, then there exists a constant $c = c(F) >0$ such that
\begin{align}\label{mf-integral-small-l}
\Big|L(F,\lambda,x) - e^{-\lambda|x|}\Big| \le c\, \lambda (1 + \lambda) \, \frac{|F(x)|}{1+x^2}
\end{align}
for all $x\in \R$. 

\smallskip

\item[(ii)] If $F$ has a double zero at the origin, then there exists a constant $c = c(F) >0$ such that
\begin{align}\label{Mf-integral-small-l}
\Big|M(F,\lambda,x) - e^{-\lambda |x|}\Big| \le c\, \lambda (1 + \lambda)\,\frac{|F(x)|}{x^2}
\end{align} 
for all $x \in \R$.
\end{enumerate}
\end{proposition}

\noindent{\it Remark}. Both \eqref{mf-integral-small-l} and \eqref{Mf-integral-small-l} fail when $\mc{N}(F) = 0$ or $2$. Computations can be done directly.

\begin{proof} {\it Part} (i). For $x<0$ and $g = g_0$ (which is an even function), we use \eqref{ml-el-rep} and integration by parts (using Lemma \ref{growth in LP} to cancel the boundary terms) to get
\begin{align}
\frac{L(F,\lambda,x)-e^{-\lambda |x|}}{F(x)} & = \int_{-\infty}^0 \big\{g(w-\lambda) -g(w+\lambda)\big\} \,e^{-xw}\, \dw \label{proof_lambda_small_eq1}\\
&= \frac{1}{x} \int_{-\infty}^0  \big\{g'(w-\lambda) -g'(w+\lambda)\big\} \,e^{-xw}\,\dw. \label{proof_lambda_small_eq2}
\end{align}
If we write $c_1 = \sup\big\{ |g''(u)|; u\in \RR \big\}$ (which is finite from Lemma \ref{growth in LP}), the mean value theorem gives us
\begin{equation*}
\big|g'(w-\lambda) -g'(w+\lambda)\big|\le 2\, \lambda \,c_1\,,
\end{equation*}
and from \eqref{proof_lambda_small_eq2} we find
\begin{equation}\label{proof_lambda_small_eq5}
\left|\frac{L(F,\lambda,x)-e^{-\lambda |x|}}{F(x)} \right| \leq \frac{2\, \lambda\, c_1}{x^2}.
\end{equation}
Observe also that
\begin{equation}\label{proof_lambda_small_eq3}
\big|g(w-\lambda) -g(w+\lambda)\big|\le 2\, \lambda \, \sup\big\{ |g'(u)| \,;\, u\in [w-\lambda,w+\lambda]\big\}.
\end{equation}
Now fix $0<\rho<\alpha_F$. From Lemma \ref{growth in LP} we can find a constant $c_2$ such that 
\begin{equation*}
|g'(t)| \leq c_2\, e^{-\rho |t|}
\end{equation*} 
for all $t \in \R$. In particular, if $w <0 $ and $u \in [w-\lambda,w+\lambda]$, we have 
\begin{equation}\label{proof_lambda_small_eq4}
|g'(u)| \leq 
\left\{
\begin{array}{cl}
\vspace{0.2cm}
c_2 & {\rm if} \ -\lambda<w<0;\\
c_2 \, e^{-\rho|w + \lambda|} & {\rm if}\ w \leq -\lambda.
\end{array}
\right.
\end{equation}
By inserting \eqref{proof_lambda_small_eq4} and \eqref{proof_lambda_small_eq3} into \eqref{proof_lambda_small_eq1} we obtain
\begin{align}\label{proof_lambda_small_eq6}
\begin{split}
\left|\frac{L(F,\lambda,x)-e^{-\lambda |x|}}{F(x)} \right| & \leq 2\lambda c_2 \left[\int_{-\infty}^{-\lambda} e^{-\rho |w + \lambda|} \,e^{-xw}\, \dw + \int_{-\lambda}^{0} e^{-xw}\, \dw\right]\\
&  \leq 2\lambda c_2 \left[\int_{-\infty}^{-\lambda} e^{-\rho |w + \lambda|} \, \dw + \lambda\right]\\[0.2em]
& =   2\lambda c_2 \left( \rho^{-1} + \lambda\right).
\end{split}
\end{align}
Now clearly \eqref{proof_lambda_small_eq5} and \eqref{proof_lambda_small_eq6} imply \eqref{mf-integral-small-l} for all $x<0$. The result follows for $x>0$ by symmetry since all the functions are even, and also for $x=0$ by continuity.

\smallskip

\noindent{\it Part} (ii). We now let $g = g_{\alpha_F/2}$. From \eqref{Ml-interpol-at0} we obtain, for $x<0$,
\begin{align}\label{Ml-rep-ext}
\begin{split}
\Big|M(F,\lambda,x) & -e^{-\lambda |x|}\Big|\\
&\le \frac{2\,|F(x)|}{x^2}\, \big|g'(0) - g'(-\lambda)\big| \\
& \ \ \  \ \ \ \ \ \ \ \ \ +\frac{|F(x)|}{x^2} \!\int_0^\infty \big|g''(-w-\lambda) -g''(w-\lambda)\big| \,e^{x w}\, \dw.
\end{split}
\end{align}
Recall that $g''$ satisfies \eqref{gc-trafo} with $F$ replaced by the function $z\mapsto z^{-2}F(z)$, which is even and nonzero at the origin (this is accomplished by integrating by parts twice). It follows that $g''$ is even and has exponential decay as $t\to\pm\infty$ (given by Lemma  \ref{growth in LP}). Writing again $c_1 = \sup\big\{ |g''(u)|; u\in \RR \big\}$ we control the first term on the right-hand side of \eqref{Ml-rep-ext},
\begin{equation*}
\frac{2\,|F(x)|}{x^2}\, \big|g'(0) - g'(-\lambda)\big| \leq \frac{2\,\lambda \,c_1\,|F(x)|}{x^2}.
\end{equation*}
To deal with the second term, if $\mc{N}(F) \geq 6$, we may use 
\begin{align*}
\big|g''(-w-\lambda) -g''(w-\lambda)\big| & = \big|g''(w+\lambda) -g''(w-\lambda)\big| \\
& \leq 2 \,\lambda \, \sup\big\{ |g'''(u)| \,;\, u\in [w-\lambda,w+\lambda]\big\}
\end{align*}
and follow the method employed in \eqref{proof_lambda_small_eq4} and \eqref{proof_lambda_small_eq6} in the proof of part (i).

\smallskip

If $\mc{N}(F) = 4$, then the function $F(z)z^{-2}$ is given by \eqref{Sec2_eq1_NF2} and $g''$ is given by \eqref{Sec2_eq1_g_t}. We then compute directly, for $x <0$, 
\begin{align*}
\int_0^{\infty}  \left\{e^{-\alpha|w-\lambda|} - e^{-\alpha|-w-\lambda|}\right\}e^{xw}\,\dw  = \frac{2\alpha \left(e^{\lambda x} - e^{-\lambda \alpha}\right)}{(\alpha^2 - x^2)} \leq \frac{2\alpha \lambda}{(\alpha - x)} \leq 2\lambda\,,
\end{align*}
by the mean value theorem. This completes the proof.

\end{proof}



\section{De Branges spaces}\label{debranges-section}

The main objective of this section is to prove Theorem \ref{Intro_Thm1}.

\subsection{Preliminaries} We start with a series of lemmas that will help us establish the connection between the hypotheses (P1)\,-\,(P4) of Theorem  \ref{Intro_Thm1} and the interpolation theory developed in Section \ref{polya}. We keep our notation close to \cite{B,HV} to facilitate the references.

\smallskip

First we review the connection between functions of exponential type and functions that have bounded type in the upper and lower half planes. Throughout this section we write $\log^+|x| =\max\{0, \log |x|\}$.

\begin{lemma}\label{theorem-krein} 
Let $F:\C \to \C$ be an entire function. The following conditions are equivalent:
\begin{enumerate}
\item[(i)] $F$ and $F^*$ have bounded type in the open upper half plane $\U$.

\smallskip

\item[(ii)] $F$ has exponential type and 
\begin{equation*}
\int_{-\infty}^\infty \frac{\log^+|F(x)|}{1+x^2} \, \dx<\infty.
\end{equation*}
\end{enumerate}
If either and therefore both of these conditions hold then $\tau(F) = \max\{v(F), v(F^*)\}$.
\end{lemma}
\begin{proof} This is a theorem of M. G.\ Krein \cite{K}.
\end{proof} 
\noindent In what follows we let $\B$ denote the set of entire functions $F$ that satisfy one and therefore both of the conditions (i) or (ii) in Lemma \ref{theorem-krein}.

\smallskip

For a Hermite-Biehler function $E(z)$ we recall the decomposition $E(z) = A(z) - iB(z)$, with $A$ and $B$ real entire functions given by  
\begin{equation*}
A(z) := \frac12 \big\{E(z) + E^*(z)\big\} \ \ \ {\rm and}  \ \ \ B(z) := \frac{i}{2}\big\{E(z) - E^*(z)\big\}.
\end{equation*}

\begin{lemma}\label{HBI-RE} 
Let $E$ be a Hermite-Biehler function with no real zeros.  The following conditions are equivalent:
\begin{enumerate}
\item[(i)]  $z\mapsto E(iz)$ is real entire.

\smallskip

\item[(ii)] The functions $z\mapsto A^2(iz)$ and $z\mapsto B^2(iz)$ are even functions and $B$ has a simple zero at the origin.
\end{enumerate}
\end{lemma}

\begin{proof} To prove that (i) implies (ii), note that (i) implies $ E^*(-iz) = E(iz)$, hence $z\mapsto A(iz)$ is even and $z\mapsto B(iz)$ is odd, and $B(0)=0$. If the zero were not simple, from \eqref{Intro_Def_K} we would have $K(0,0)=0$, and from \cite[Lemma 11]{HV} this would imply $E(0)=0$, a contradiction.

\smallskip

For the other direction, $B(0)=0$ implies $A(0)\neq 0$ since $E$ has no real zeros. This implies that $z\mapsto A(iz)$ is even. If $z\mapsto B(iz)$ were even, then $z\mapsto E(iz)$ would be even, which violates  \eqref{HB-condition}. Thus $z\mapsto B(iz)$ is odd and (i) follows.
\end{proof}

\begin{lemma}\label{Sec_dBS_E_is_in_B} 
Assume that $E$ is a Hermite-Biehler function of bounded type in $\U$. Then the functions $A$ and $B$ are in the Laguerre-P\'{o}lya class.
\end{lemma}

\begin{proof} Since $E^*/E$ is bounded in $\U$, it is of bounded type. Hence $E^*$ is also of bounded type in $\U$, which shows that $E \in \B$. From \eqref{HB-condition} and \eqref{Intro_def_A_B} we see that $A$ and $B$ have only real zeros. Since $A=A^*$ and $A$ has bounded type in the upper half plane, $A$ is in the P\'{o}lya class by \cite[Problem 34]{B} (alternatively, see \cite{KW}). The same argument shows that $B$ is in the P\'{o}lya class. Since $A$ and $B$ are real entire functions, they are in the Laguerre-P\'{o}lya class.
\end{proof}

\subsection{The $UU^*$ decomposition} The next three lemmas are key steps connecting the $L^1$ and $L^2$ theories. They will allow us to use the Hilbert space structure to prove the optimality part of Theorem \ref{Intro_Thm1}.

\smallskip

If $E$ is a Hermite-Biehler function of bounded type in $\U$, we have seen in the proof of Lemma \ref{Sec_dBS_E_is_in_B} that $E \in\mc{B}$. From \eqref{HB-condition} we know that $v(E^*) \leq v(E)$ and thus, from Krein's theorem (Lemma \ref{theorem-krein}), we obtain that $E$ has exponential type $\tau(E) = v(E)$.

\begin{lemma}\label{nonnegl1-tohe} 
Let $E$ be a Hermite-Biehler function of bounded type in $\U$ with exponential  type $\tau(E)$. Let $F:\C \to \C$ be a real entire function of exponential type at most $2\tau(E)$ that satisfies
\begin{align*}
F(x)\ge 0
\end{align*}
for all $x \in \R$ and
\begin{align}\label{integral-F-finite}
\int_{-\infty}^\infty F(x)\, |E(x)|^{-2} \,\dx <\infty.
\end{align} 
Then there exists $U\in\H(E)$ such that
\begin{align*}
F(z) = U(z) U^*(z)
\end{align*}
for all $z \in \C$.
\end{lemma}

\begin{proof} This argument is essentially used in the proof of \cite[Theorem 15]{HV} and we include it here for completeness. 

\smallskip

Using Jensen's inequality and \eqref{integral-F-finite} we obtain
\begin{align*}
\int_{-\infty}^\infty  & \frac{\log^+|F(x)/E(x)^2|}{1+x^2} \,\dx \\
& \qquad \le \int_{-\infty}^\infty \frac{\log\big(1+|F(x)/E(x)^2|\big)}{1+x^2}\,\dx\\
& \qquad \leq \pi \log\left\{ 1 + \pi^{-1} \int_{-\infty}^{\infty} \frac{|F(x)/E(x)^2|}{1+x^2}\,\dx\right\} < \infty.
\end{align*}
Applying the elementary inequality 
$$\log^+|ab|\le \log^+|a|+\log^+|b|$$
with $a=|F(x)/E(x)^{2}|$ and $b = |E(x)|^2$, and using that $E \in \mc{B}$, we conclude that
\begin{equation*}
\int_{-\infty}^\infty \frac{\log^+|F(x)|}{1+x^2}\, \dx<\infty.
\end{equation*}
By Lemma \ref{theorem-krein} it follows that $F\in\Bc$. 

\smallskip

Since $F$ is real entire and has bounded type in $\U$, if $z_1, z_2, \ldots., z_n,\ldots$ are the zeros of $F$ in $\U$ listed with appropriate multiplicity, then by \cite[Theorem 8]{B} we have
\begin{equation*}
\sum_{n=1}^{\infty} \frac{y_n}{x_n^2 +y_n^2} < \infty\,,
\end{equation*}
where $z_n = x_n + iy_n$ and $y_n >0$. The Blaschke product 
\begin{equation*}
D(z) = \lim_{N\to \infty} \prod_{n=1}^{N} \left(1 - \frac{z}{z_n}\right)\left(1 - \frac{z}{\overline{z_n}}\right)^{-1}
\end{equation*}
defines a meromorphic function on $\C$, which is analytic in $\U$ and continuous on the closure of $\U$. Moreover, $D^{-1} F$ is entire, all of its zeros have even multiplicity and none lies in $\U$ (note that the zeros of $F$ on $\R$ have even multiplicity since $F$ is nonnegative on $\R$). Hence there exists an entire function $U$, with no zeros in $\U$, such that
\begin{align}\label{F=DU2}
F = DU^2.
\end{align}
Since $DD^*=1$ and $F=F^*$, we have $FD = (U^*)^2$. It follows that $F^2 = (UU^*)^2$. Since $F$ is nonnegative on $\R$, we obtain the representation
\begin{equation*}
F(z) = U(z) U^*(z)
\end{equation*}
for all $z \in \C$. 

\smallskip

From \eqref{F=DU2} and \cite[Theorem 9]{B} we find that $U$ has bounded type in $\U$. Since $UU^* = F =  D^*(U^*)^2$ it follows that $U^*  = DU$. Since the Blaschke product $D$ has bounded type in $\U$, the function $U^*$ has bounded type in $\U$ and therefore $U\in \Bc$. Moreover, since $1/E$ has bounded type in $\U$, we conclude that $U/E$ and $U^*/E$ have bounded type in $\U$.  Since $v(D) = 0$ and $U^* = DU$ we have $v(U) = v(U^*)$. Also, since $F=DU^2$, we have
$$2 v(E) = 2\tau(E) \geq \tau (F) = v(F) = 2 v(U),$$ 
and from this it follows that $U/E$ and $U^*/E$ have nonpositive mean type in $\U$. Finally, since $F = UU^*$, it follows from \eqref{integral-F-finite} that
\[
\int_{-\infty}^\infty |U(x)|^2\,|E(x)|^{-2}\, \dx <\infty,
\]
hence $U\in\H(E)$.
\end{proof}

\begin{lemma}\label{L1-estimate} 
Let $E$ be a Hermite-Biehler function of bounded type in $\U$, with no real zeros and such that $z\mapsto E(iz)$ is real entire. Then
$$x\mapsto L(A^2,\lambda,x)-e^{-\lambda|x|}$$ 
and 
$$x\mapsto M(B^2,\lambda,x) -e^{-\lambda|x|}$$ 
belong to $L^1\big(\R,|E(x)|^{-2}\,\dx\big)$.
\end{lemma}

\begin{proof} Observe first, from Lemmas \ref{HBI-RE} and \ref{Sec_dBS_E_is_in_B}, that the functions $z\mapsto A^2(z)$ and $z \mapsto B^2(z)$ are even Laguerre-P\'{o}lya functions with $A^2(0) \neq 0$ and $B^2$ having a double zero at the origin. We are then able to apply the results of Section \ref{polya}.

\smallskip

\noindent{\it Step 1 - Proof for $L$}. The case $\mc{N}(A^2) \geq 4$ follows from \eqref{mf-integral-small-l} combined with the fact that $E$ has no real zeros and $A^2/E^2$ is bounded on $\R$ (by \eqref{Intro_def_A_B}).

\smallskip

In the case $\mc{N}(A^2) = 2$, we must have $A^2$ given by \eqref{Sec2_eq1_NF2} and $g=g_0$ given by \eqref{Sec2_eq1_g_t}. For $x<0$ we can compute directly from \eqref{ml-el-rep} that 
\begin{equation}\label{Sec3_eq0_L_A}
\left|L(A^2, \lambda,x) - e^{-\lambda |x|} \right|= \left| A(x)^2\,\frac{\alpha^2}{A(0)^2} \frac{\big(e^{- \lambda \alpha}- e^{\lambda x}\big)}{\alpha^2 - x^2}\right|  \ll \frac{A^2(x)}{x^2}
\end{equation}
as $x \to -\infty$. The same bound holds when $x \to \infty$ by symmetry and the result follows.

\smallskip

In the case $\mc{N}(A^2) = 0$, the function $A^2$ is constant and $L(A^2, \lambda,z)$ is identically zero. The inequality \eqref{Sec3_eq0_L_A} follows trivially.

\smallskip

\noindent{\it Step 2 - Proof for $M$}. The case $\mc{N}(B^2)\geq 4$ follows from \eqref{Mf-integral-small-l} combined with the fact that $E$ has no real zeros and $B^2/E^2$ is bounded on $\R$.

\smallskip

If $\mc{N}(B^2) = 2$, we must have $B(z)^2 = Cz^{2}$ for some $C>0$. The result follows from \eqref{Sec2_lastcase_M}.
\end{proof}

\begin{lemma} \label{Sec3_lem16}
Let $\lambda>0$. Let $E$ be a Hermite-Biehler function of bounded type in $\U$ with exponential  type $\tau(E)$. Assume that $E$ has no real zeros and $z\mapsto E(iz)$ is real entire. If
\begin{align}\label{fk-in-L1}
\int_{-\infty}^\infty  e^{-\lambda|x|}\, |E(x)|^{-2}\,\dx<\infty,
\end{align}
then there exist $U,V\in \H(E)$ so that
\begin{align}
M(B^2,\lambda,z) &= U(z) U^*(z)\label{decompose-maj}\\
L(A^2,\lambda,z)&= U(z)U^*(z) - V(z) V^*(z)\label{decompose-min}
\end{align}
for all $z \in \C$.
\end{lemma}

\begin{proof} We note first that $A$ and $B$ have exponential type $\tau(E)$. In fact, in the proof of Lemma \ref{Sec_dBS_E_is_in_B} we have already seen that $A, B \in \B$. Since
$$\frac{A(z)}{E(z)} = \frac{1}{2}\left(1 - \frac{E^*(z)}{E(z)}\right),$$
the real part of $A/E$ on $\U$ is positive. From \cite[Problem 30]{B} we have $v(A/E) = 0$ and therefore $\tau(A) = v(A) = v(E) = \tau(E)$. The same applies to $B$. 

\smallskip

From \eqref{al-growth} we obtain that $z \mapsto L(A^2,\lambda,z)$ and $z\mapsto M(B^2,\lambda,z)$ have exponential type at most $2\tau(E)$.  Lemma \ref{L1-estimate} and \eqref{fk-in-L1} imply that $x \mapsto M(B^2,\lambda,x) \in L^1\big(\R, |E(x)|^{-2}\,\dx\big)$. From Lemma \ref{nonnegl1-tohe} there exists $U\in \H(E)$ such that 
$$M(B^2,\lambda,z) = U(z)U^*(z).$$ 
Note that 
$$M(B^2,\lambda,x) - L(A^2,\lambda,x) \ge 0$$
for all $x \in\R$, and Lemma \ref{L1-estimate} implies that this difference also belongs to $L^1\big(\R$, $|E(x)|^{-2}\,\dx\big)$. It follows from another application of Lemma \ref{nonnegl1-tohe} that 
$$M(B^2,\lambda,z) - L(A^2,\lambda,z) = V(z) V^*(z),$$ 
for some $V\in\H(E)$, and hence 
$$L(A^2,\lambda,z) = U(z)U^*(z) - V(z)V^*(z).$$
\end{proof}

\subsection{Proof of Theorem \ref{Intro_Thm1}} We have now gathered all the necessary elements to prove our first theorem.

\subsubsection{Existence} From the proof of Lemma \ref{Sec3_lem16} we have seen that $z \mapsto L(A^2,\lambda,z)$ and $z\mapsto M(B^2,\lambda,z)$ have exponential type at most $2\tau(E)$. Moreover, from \eqref{minorant-ineq} and  \eqref{majorant-ineq} we find that
\begin{equation*}
L(A^2,\lambda,x) \leq e^{-\lambda |x|} \leq M(B^2,\lambda,x)
\end{equation*}
for all $x\in \R$. Let $U,V \in \mc{H}(E)$ be given by \eqref{decompose-maj} and \eqref{decompose-min}. We now use the fact that $B \notin \mc{H}(E)$, given by hypothesis (P4), to invoke \cite[Theorem 22, case $\alpha=0$]{B}. This guarantees that $\{z \mapsto K(\xi,z);\ B(\xi) = 0\}$ is a complete orthogonal set in $\mc{H}(E)$. We then have
\begin{align*}
\int_{-\infty}^\infty  M(B^2,\lambda,x)\,|E(x)|^{-2} \,\dx & = \int_{-\infty}^\infty |U(x)|^2 \,|E(x)|^{-2} \dx\\
& = \sum_{B(\xi)=0} \frac{|U(\xi)|^2}{K(\xi,\xi)}\\
& = \sum_{B(\xi)=0} \frac{e^{-\lambda|\xi|}}{K(\xi,\xi)}\,,
\end{align*}
where the last equality comes from the fact that $|U(\xi)|^2 = e^{-\lambda |\xi|}$ for all $\xi \in \R$ with $B(\xi)=0$ by \eqref{decompose-maj} and \eqref{majorant-interpolation}.

\smallskip

From the fact that $A \notin \mc{H}(E)$, an application of \cite[Theorem 22, case $\alpha=\pi/2$]{B} gives us that $\{z \mapsto K(\xi,z);\ A(\xi) = 0\}$ is also a complete orthogonal set in $\mc{H}(E)$. Therefore
\begin{align*}
\int_{-\infty}^\infty  L(A^2,\lambda,x)\,|E(x)|^{-2} \,\dx & = \int_{-\infty}^\infty \left\{|U(x)|^2 - |V(x)|^2\right\} |E(x)|^{-2} \dx\\
& = \sum_{A(\xi)=0} \frac{|U(\xi)|^2 - |V(\xi)|^2}{K(\xi,\xi)}\\
& = \sum_{A(\xi)=0} \frac{e^{-\lambda|\xi|}}{K(\xi,\xi)}\,,
\end{align*}
where the last equality comes from the fact that $|U(\xi)|^2 - |V(\xi)|^2 = e^{-\lambda |\xi|}$ for all $\xi \in \R$ with $A(\xi)=0$ by \eqref{decompose-min} and \eqref{minorant-interpolation}.

\subsubsection{Optimality} Now let $M:\C \to \C$ be an entire function of exponential type at most $2 \tau(E)$ such that $M(x) \geq e^{-\lambda |x|}$ for all $x \in \R$. If 
\begin{equation*}
\int_{-\infty}^{\infty} M(x) \,|E(x)|^{-2}\,\dx = \infty
\end{equation*}
there is nothing to prove, hence we assume that this integral is finite. Lemma \ref{nonnegl1-tohe} implies that $M(z)=W(z)W^*(z)$ for some $W\in \H(E)$ and then 
 \begin{align*}
\int_{-\infty}^{\infty} M(x) \,|E(x)|^{-2}\,\dx & = \int_{-\infty}^{\infty} |W(x)|^2 \,|E(x)|^{-2}\,\dx \\
& = \sum_{B(\xi)=0} \frac{|W(\xi)|^2}{K(\xi,\xi)}\\ 
& \ge \sum_{B(\xi)=0} \frac{e^{-\lambda|\xi|}}{K(\xi,\xi)}\,,
\end{align*}
as claimed. 

\smallskip

Let $L:\C \to \C$ be an entire function of exponential type at most $2 \tau(E)$ such that $L(x) \leq e^{-\lambda |x|}$ for all $x \in \R$. If 
\begin{equation*}
\int_{-\infty}^{\infty} L(x) \,|E(x)|^{-2}\,\dx = - \infty
\end{equation*}
there is nothing to prove, hence we assume that this integral is finite. We have already noticed the existence of a majorant $z \mapsto M(B^2, \lambda,z)$ of exponential type at most $2\tau(E)$ for $e^{-\lambda |x|}$. In particular, the function $z \mapsto M(B^2, \lambda,z) - L(z)$ has exponential type at most $2\tau(E)$, is nonnegative on $\R$ and belongs to $L^1(\R, |E(x)|^{-2}\,\dx)$. A new application of Lemma \ref{nonnegl1-tohe} gives us that $M(B^2, \lambda,z) - L(z) = T(z)T^*(z)$ for some $T \in \mc{H}(E)$. From \eqref{decompose-maj} we then have $L(z) = U(z)U^*(z) - T(z)T^*(z)$, and finally
 \begin{align*}
\int_{-\infty}^{\infty} L(x) \,|E(x)|^{-2}\,\dx & = \int_{-\infty}^{\infty} \left\{|U(x)|^2 - |T(x)|^2\right\} \,|E(x)|^{-2}\,\dx \\
& = \sum_{A(\xi)=0} \frac{|U(\xi)|^2 - |T(\xi)|^2}{K(\xi,\xi)}\\ 
& \le \sum_{A(\xi)=0} \frac{e^{-\lambda|\xi|}}{K(\xi,\xi)}.
\end{align*}
This completes the proof of Theorem \ref{Intro_Thm1}.


\section{Homogeneous spaces and radial symmetrization}\label{Hom_spaces}

In this section we shall prove Theorem \ref{Thm1}. An important ingredient of the proof is the specialization of Theorem \ref{Intro_Thm1} to a suitable class of homogeneous de Branges spaces. We start by briefly reviewing these spaces and their relevant properties.

\subsection{The homogeneous spaces $\mc{H}(E_{\nu})$} Let $\nu > -1$. A space $\mc{H}(E)$ is said to be homogeneous of order $\nu$ if, for all $0 < a < 1$ and all $F \in \mc{H}(E)$, the function $z \mapsto a^{\nu +1}F(az)$ belongs to $\mc{H}(E)$ and has the same norm as $F$.  Such spaces were characterized by L. de Branges in \cite{B2} (see also \cite[Section 50]{B}). 

\smallskip

We consider here the special family of homogeneous spaces $\mc{H}(E_{\nu})$ where
\begin{equation*}
E_{\nu}(z) = A_{\nu}(z) - iB_{\nu}(z),
\end{equation*}
and the real entire functions $A_{\nu}$ and $B_{\nu}$ are given by \eqref{Intro_A_nu} and \eqref{Intro_B_nu}. From \cite[Section 50, Problems 227 and 228]{B} we see that in fact $E_{\nu}$ is a Hermite-Biehler function with no real zeros and $\mc{H}(E_{\nu})$ is homogeneous of order $\nu$. Moreover, $E_{\nu}$ has bounded type in $\mc{U}$ with mean type $1$ and, by Krein's theorem (Lemma \ref{theorem-krein}), it has exponential type $1$ as well. Since $A_{\nu}$ is even and $B_{\nu}$ is odd, the function $z \mapsto E_{\nu}(iz)$ is real entire. This accounts for hypotheses (P1)\,-\,(P3) of Theorem \ref{Intro_Thm1}.

\smallskip

From the proof of Lemma \ref{Sec3_lem16} we see that $\tau(A_{\nu}) = \tau(B_{\nu}) = \tau(E_{\nu}) =1$. We gather other relevant facts about the spaces $\mc{H}(E_{\nu})$ in the next lemma. In particular, the item (ii) below is the key identity that connects the theory of de Branges spaces to the multidimensional Euclidean problems.

\begin{lemma}\label{Sec5_rel_facts}
Let $\nu > -1$. The following properties hold:
\begin{enumerate}
\item[(i)] There exist positive constants $a_\nu,b_\nu$ such that 
\begin{align}\label{Lem17_i}
a_\nu |x|^{2\nu+1} \le |E_{\nu}(x)|^{-2} \le b_\nu |x|^{2\nu+1}
\end{align}
for all $x \in \R$ with $|x|\geq1$.
\smallskip

\item[(ii)] For $F\in\H(E_\nu)$ we have the identity 
\begin{align} \label{Lem17_ii}
\int_{-\infty}^\infty |F(x)|^{2}\,|E_{\nu}(x)|^{-2}\, \dx = c_\nu \int_{-\infty}^\infty |F(x)|^2 \,|x|^{2\nu+1} \,\dx\,,
\end{align}
with $c_\nu = \pi \,2^{-2\nu-1}\, \Gamma(\nu+1)^{-2}$.

\smallskip

\item[(iii)] An entire function $F$ belongs to $\H(E_\nu)$ if and only if $F$ has exponential type at most $1$ and
\begin{equation}\label{Sec4_eq0_int_cond}
\int_{-\infty}^\infty |F(x)|^2 \,|x|^{2\nu+1}\, \dx <\infty.
\end{equation}
\end{enumerate}
\end{lemma}

\begin{proof} These facts are all contained in \cite[Equations (5.1), (5.2) and Lemma 16]{HV}.
\end{proof}

From the asymptotics \eqref{Asymptotic_Bessel_functions}, together with \eqref{Intro_Bessel1} and \eqref{Intro_Bessel2}, we find that $A_{\nu}$ and $B_{\nu}$ do not satisfy the integrability condition \eqref{Sec4_eq0_int_cond}. Therefore $A_{\nu}, B_{\nu} \notin \mc{H}(E_{\nu})$ and the final hypothesis (P4) is verified. The general framework of Theorem \ref{Intro_Thm1} is then available for the spaces $\mc{H}(E_{\nu})$.

\subsection{Radial symmetrization} We now briefly recall a couple of results from \cite[Section 6]{HV} that will be useful to establish a connection between the one-dimensional theory and the multidimensional theory. 
\begin{lemma}\label{type-to-ntype} 
Let $F:\C\to \C$ be an even entire function with power series representation
\begin{align*}
F(z) = \sum_{k=0}^\infty c_k z^{2k}
\end{align*}
and let $\psi_N(F):\C^N \to \C$ be the entire function
\begin{align*}
\psi_N(F)({\bf z}) = \sum_{k=0}^\infty c_k (z_1^2 + \ldots + z_n^2)^k.
\end{align*}
The following properties hold:
\smallskip
\begin{itemize}
\item[(i)] $F$ has exponential type if and only if $\psi_N(F)$ has exponential type, and $\tau(F) = \tau(\psi_N(F))$. 

\smallskip

\item[(ii)] We have
\begin{align*}
\tfrac12\, \omega_{N-1} \int_{-\infty}^\infty F(x)\,|x|^{2\nu+1} \,\dx = \int_{\R^N} \psi_N(F)({\bf x}) \,|{\bf x}|^{2\nu+2-N}\,\d{\bf x},
\end{align*}
where $\nu >1$ and $\omega_{N-1} = 2 \pi^{N/2} \,\Gamma(N/2)^{-1}$ is the surface area of the unit sphere in $\R^N$, provided that both integrals are absolutely convergent.
\end{itemize}
\end{lemma}

\begin{proof} This is \cite[Lemma 18]{HV}.
\end{proof}

For $N\geq 2$ let $SO(N)$ denote the compact topological group of real orthogonal $N \times N$ matrices $M$ with $\det M =1$. Let $\sigma$ be its left-invariant (and also right-invariant since $SO(N)$ is compact) Haar measure, normalized so that $\sigma (SO(N)) =1$. For an entire function $\F: \C^N \to \C$ we define its radial symmetrization $\widetilde{\F}:\C^N \to \C$ by 
\begin{equation*}
\widetilde{\F}(\z) = \int_{SO(N)}\F(M\z)\,\d \sigma(M).
\end{equation*}
If $N=1$ we define $\widetilde{\F}(z) := \tfrac12\{\F(z) + \F(-z)\}$. The following lemma lists the relevant properties of $\widetilde{\F}$.

\begin{lemma}\label{Lemma19_HV}
Let $\F: \C^N \to \C$ be an entire function. Then $\widetilde{\F}:\C^N \to \C$ is an entire function that satisfies the following properties:
\begin{itemize}
\item[(i)] $\widetilde{\F}$ has a power series expansion of the form
\begin{equation*}
\widetilde{\F}({\bf z}) = \sum_{k=0}^\infty c_k (z_1^2 + \ldots + z_n^2)^k.
\end{equation*}

\smallskip

\item[(ii)] If $\F$ has exponential type then $\widetilde{\F}$ has exponential type and $\tau\big(\widetilde{\F}\big) \leq \tau(\F)$.

\smallskip

\end{itemize}
\end{lemma}
\begin{proof} This is \cite[Lemma 19]{HV}.
\end{proof}

\subsection{Proof of Theorem \ref{Thm1}} We are now in position to prove our second theorem. Recall that we write $\F_{\lambda}(\x) = e^{-\lambda |\x|}$. 

\subsubsection{Proof of part {\rm (i)}} For $\kappa >0$ observe that $\L \in \E_{\delta}^{N-} (\mc{F}_{\lambda})$ if and only if $\x \mapsto \L(\kappa \x) \in \E_{\kappa\delta}^{N-} (\mc{F}_{\kappa\lambda})$. An analogous property holds for the majorants. This is enough to conclude part (i).

\subsubsection{Proof of part {\rm (ii)}} Let $L \in \E_{\delta}^{1-} (\mc{F}_{\lambda})$ be such that 
\begin{equation}\label{Sec4_added_eq}
\int_{-\infty}^\infty \left\{ e^{-\lambda|x|} - L(x)\right\} \,|x|^{2\nu+1} \,\dx < \infty.
\end{equation}
By considering $z \mapsto \tfrac{1}{2} \{L(z) + L(-z)\}$ we may assume that $L$ is even (note that this does not change the value of the integral in \eqref{Sec4_added_eq}). By Lemma \ref{type-to-ntype} we have that $\psi_{N}(L) \in \E_{\delta}^{N-} (\mc{F}_{\lambda})$, and a change to radial variables gives us
\begin{align*}
\int_{\R^N} \Big\{e^{-\lambda|\x|}  - \psi_N(L)&(\x)\Big\}\, |\x|^{2\nu+2 - N}\,\dxx \\
& = \tfrac12\, \omega_{N-1} \int_{-\infty}^\infty \Big\{ e^{-\lambda|x|} - L(x)\Big\} \,|x|^{2\nu+1} \,\dx.
\end{align*}
Hence
\begin{equation*}\label{Sec5_eq3_pI}
U_{\nu}^{N-}(\delta, \lambda) \leq  \tfrac12\, \omega_{N-1} \,U_{\nu}^{1-}(\delta, \lambda).
\end{equation*}

\smallskip

On the other hand, let $\L \in \E_{\delta}^{N-} (\mc{F}_{\lambda})$ be such that 
\begin{align*}
\int_{\R^N} \Big\{e^{-\lambda|\x|} - \L(\x)\Big\}\, |\x|^{2\nu+2 - N}\,\dxx < \infty.
\end{align*}
By Lemma \ref{Lemma19_HV} we have that $\widetilde{\L} \in \E_{\delta}^{N-} (\mc{F}_{\lambda})$ and it has a power series expansion 
\begin{equation*}
\widetilde{\L}({\bf z}) = \sum_{k=0}^\infty c_k (z_1^2 + \ldots + z_n^2)^k.
\end{equation*}
Define the entire function $L:\C \to \C$ by
\begin{equation*}
L(z) = \widetilde{\L}(z,0,0,\ldots,0) = \sum_{k=0}^\infty c_k z^{2k}.
\end{equation*}
By Lemma \ref{type-to-ntype} we know that $L \in \E_{\delta}^{1-} (\F_{\lambda})$. An application of Fubini's theorem now gives us
\begin{align*}
\tfrac12\, \omega_{N-1} \int_{-\infty}^\infty & \Big\{ e^{-\lambda|x|}  - L(x)\Big\} \,|x|^{2\nu+1} \,\dx\\
& = \int_{\R^N} \Big\{e^{-\lambda|\x|} - \widetilde{\L}(\x)\Big\}\, |\x|^{2\nu+2 - N}\,\dxx\\
& = \int_{\R^N} \int_{SO(N)}  \Big\{e^{-\lambda|M\x|} - \L(M\x)\Big\}\, |M\x|^{2\nu+2 - N}\,\d \sigma(M)\,\dxx\\
& =  \int_{SO(N)}   \int_{\R^N} \Big\{e^{-\lambda|M\x|} - \L(M\x)\Big\}\, |M\x|^{2\nu+2 - N}\,\dxx\,\d \sigma(M)\\
& =  \int_{\R^N} \Big\{e^{-\lambda|\x|} - \L(\x)\Big\}\, |\x|^{2\nu+2 - N}\,\dxx.
\end{align*}
Hence
\begin{equation*}\label{Sec5_eq3_pII}
 \tfrac12\, \omega_{N-1} \,U_{\nu}^{1-}(\delta, \lambda) \leq U_{\nu}^{N-}(\delta, \lambda),
\end{equation*}
and this concludes the proof for the minorant case. The majorant case is treated analogously.

\subsubsection{Proof of part {\rm (iii)}} We are now interested in computing $U_{\nu}^{1\pm}(2, \lambda)$. Let us start with the majorant case. First observe that 
\begin{equation*}
\int_{-\infty}^{\infty} e^{-\lambda|x|}\,|x|^{2\nu +1} \,\dx = \frac{2 \,\Gamma(2\nu +2)}{\lambda^{2\nu +2}}.
\end{equation*}
Let $M:\C \to \C$ be an entire function of exponential type at most $2$ such that $M(x) \geq e^{-\lambda|x|}$ for all $x \in \R$ and 
\begin{equation*}
\int_{-\infty}^{\infty}M(x)\,|x|^{2\nu +1}\,\dx <\infty.
\end{equation*}
From \eqref{Lem17_i} this implies that
\begin{equation*}
\int_{-\infty}^{\infty}M(x)\,|E_{\nu}(x)|^{-2}\,\dx <\infty
\end{equation*}
and thus, by Lemma \ref{nonnegl1-tohe}, we know that $M(z) = W(z)W^*(z)$ for some $W \in \mc{H}(E_{\nu})$. From Theorem \ref{Intro_Thm1} and the key identity \eqref{Lem17_ii} we have
\begin{align}\label{Sec5_eq3_majorant_B}
\begin{split}
\sum_{B_{\nu}(\xi)=0} \frac{e^{-\lambda|\xi|}}{K_{\nu}(\xi,\xi)} & \leq \int_{-\infty}^{\infty} M(x) \,|E_{\nu}(x)|^{-2}\,\dx\\
& =  \int_{-\infty}^{\infty} |W(x)|^2 \,|E_{\nu}(x)|^{-2}\,\dx\\
& =  c_{\nu}\int_{-\infty}^{\infty} |W(x)|^2 \,|x|^{2\nu +1}\,\dx\\
& =  c_{\nu}\int_{-\infty}^{\infty}M(x)\,|x|^{2\nu +1}\,\dx.
\end{split}
\end{align}
Moreover, we have seen that the entire function $z \mapsto M(B_{\nu}^2, \lambda, z)$ of exponential type at most $2$ verifies the equality in \eqref{Sec5_eq3_majorant_B}. Therefore
\begin{equation*}
U_{\nu}^{1+}(2, \lambda) = \sum_{B_{\nu}(\xi)=0} \frac{e^{-\lambda|\xi|}}{c_{\nu}K_{\nu}(\xi,\xi)} - \frac{2 \,\Gamma(2\nu +2)}{\lambda^{2\nu +2}}.
\end{equation*} 

\smallskip

The minorant part follows in a similar way. Indeed, if $L:\C \to \C$ is an entire function of exponential type at most $2$ such that $L(x) \leq e^{-\lambda|x|}$ for all $x \in \R$ and
\begin{equation*}
\int_{-\infty}^{\infty}L(x)\,|x|^{2\nu +1}\,\dx > -\infty,
\end{equation*}
we use the existence of a majorant $z \mapsto M(B_{\nu}^2, \lambda, z)$ to apply Lemma \ref{nonnegl1-tohe} to the nonnegative function $x \mapsto M(B_{\nu}^2, \lambda, x) - L(x)$ and obtain the representation $L(z) = U(z)U^*(z) - T(z)T^*(z)$ with $U,T \in \mc{H}(E_{\nu})$. The rest follows as in \eqref{Sec5_eq3_majorant_B} using the minorant part of Theorem \ref{Intro_Thm1} and the key identity \eqref{Lem17_ii}. The entire function $z \mapsto L(A_{\nu}^2, \lambda, z)$ of exponential type at most $2$ verifies the equality.

\subsubsection{Proof of part {\rm (iv)}} For $N \geq 1$ and $\delta = 2$ we define
\begin{equation}\label{Sec4_eq0_def_L_nu}
\L_{\nu}(2, \lambda, \z) := \psi_N \big(L(A_{\nu}^2, \lambda, \cdot)\big)(\z)
\end{equation}
and
\begin{equation}\label{Sec4_eq0_def_M_nu}
\M_{\nu}(2, \lambda, \z) := \psi_N \big(M(B_{\nu}^2, \lambda, \cdot)\big)(\z).
\end{equation}
These functions have exponential type at most $2$ (by Lemma \ref{type-to-ntype}) and satisfy
$$\L_{\nu}(2, \lambda, \x)  \leq e^{-\lambda|\x|} \leq \M_{\nu}(2, \lambda, \x)$$
for all $\x \in \R^N$. Moreover,
\begin{align*}
\int_{\R^N} \Big\{ \M_{\nu}&(2, \lambda, \x) - e^{-\lambda|\x|}\Big\}\,|\x|^{2\nu +2 - N}\,\dxx \\
& = \tfrac12\, \omega_{N-1}\int_{-\infty}^{\infty} \Big\{ M(B_{\nu}^2,\lambda, x) - e^{-\lambda|x|}\Big\}\,|x|^{2\nu +1}\,\dx\\[0.5em]
& =  \tfrac12\, \omega_{N-1} \,U_{\nu}^{1+}(2,\lambda) \\[0.5em]
& = U_{\nu}^{N+}(2,\lambda).
\end{align*}
The computation for the minorant is analogous. The existence of extremal functions for general $\delta >0$ follows by the change of variables given by part (i). This completes the proof.


\section{Difference of squares in $\mc{H}(E_{\nu})$}\label{Diff_Squares}

When dealing with the sort of extremal problems presented here, associated to a de Branges space $\mc{H}(E)$ with $E$ of bounded type in $\U$, an important feature for optimality considerations is the ability to write a real entire function of exponential type at most $2 \tau(E)$ whose restriction to $\R$ belongs to $L^1\big(\R,|E(x)|^{-2}\,\dx\big)$ as a difference of squares
$$F(z) = U(z)U^*(z) - V(z)V^*(z)$$
where $U, V \in \mc{H}(E)$. This was accomplished in Lemma \ref{nonnegl1-tohe} for a nonnegative $F$ and, more generally, in Lemma \ref{Sec3_lem16} and in the proofs of Theorems \ref{Intro_Thm1} and \ref{Thm1}, for an $F$ that admits a nonnegative majorant  of exponential type at most $2 \tau(E)$ in $L^1\big(\R,|E(x)|^{-2}\,\dx\big)$. 

\smallskip

In this section we work with the homogeneous spaces $\mc{H}(E_{\nu})$ and extend this construction to {\it any} real entire function $F:\C \to \C$ of exponential type at most $2$ (recall that $\tau(E_{\nu})=1$) such that $F \in L^1\big(\R,|E_{\nu}(x)|^{-2}\,\dx\big)$. 

\subsection{Revisiting a result of Plancherel and P\'{o}lya} In \cite[\S 32]{PP} Plancherel and P\'{o}lya observed that if an entire function $F:\C \to \C$ has exponential type at most $\delta$ and belongs to $L^p(\R)$, for some $1\leq p \leq \infty$, then $F'$ has the same properties. The case $p=\infty$ is originally due to Bernstein \cite{Ber}. We now extend this result to the homogeneous spaces $\mc{H}(E_{\nu})$.

\begin{theorem}\label{PP-in-HE} 
Let $\nu>-1$ and $\delta >0$. Let $F:\C \to \C$ be an entire function of exponential type at most $\delta$ such that $F\in L^p\big(\R, |E_\nu(x)|^{-2}\, \dx\big)$, for some $1 \leq p \leq \infty$. Then $F'$ has exponential type at most $\delta$ and belongs to $L^p\big(\R, |E_\nu(x)|^{-2} \,\dx\big)$. 
\end{theorem}

\begin{proof} The case $p=\infty$ is just a restatement of the original result, so let us focus on the case $1\leq p <\infty$. Let us first prove that $F'$ has exponential type at most $\delta$.

\smallskip

If $2\nu +1 \geq 0$, by \eqref{Lem17_i} we find that $F \in L^p(\R)$, and then it follows that $F'$ has exponential type at most $\delta$. We now consider the case $-1 < 2 \nu +1 <0$. If $F$ has no zeros, then $F(z) = e^{az+b}$ with $|a| \leq \delta$, and $F'$ clearly has exponential type at most $\delta$ as well. If $F$ has zeros, let $\beta$ be a zero of $F$ and consider the function $Q(z) = F(z)/(z-\beta)$. From \eqref{Lem17_i} we have that $Q \in L^p(\R)$. Since $Q$ has exponential type at most $\delta$, we find that
$$Q'(z) = \frac{F'(z)(z-\beta) - F(z)}{(z-\beta)^2}$$
also has exponential type at most $\delta$, and therefore so does $F'$.

\smallskip

Now choose $\theta >-1$ such that $(2\nu +1)/p = 2\theta +1$. Let $\alpha = \lfloor \theta \rfloor -\theta+1$ and $k =  \lfloor \theta \rfloor + 2$. Define the entire function $G:\C \to \C$ by
$$G(z) = z^{2k} E_\alpha(z) E^*_\alpha(z).$$
From \eqref{Lem17_i} we have that
\begin{align}\label{asy-G}
c_\alpha |x|^{2\nu+1} \le G(x)^p\le d_\alpha |x|^{2\nu+1}
\end{align} 
for $x \in \R$ with $|x|\ge 1$ and constants $c_\alpha,d_\alpha>0$. From the differential equations \eqref{Intro_Diff_Eqs} we get
\begin{align*}
E_\alpha'(z) &= -i E_\alpha(z) + i (2\alpha+1) z^{-1}B_\alpha(z),\\
(E_\alpha^*)'(z) &= i E_\alpha^*(z) - i(2\alpha+1)z^{-1} B_\alpha(z).
\end{align*}
It follows that
\begin{align*}
E_\alpha(z) (E_\alpha^*)'(z) + E'_\alpha(z) E_\alpha^*(z)& = E_\alpha(z) \big\{ i E_\alpha^*(z) - i(2\alpha+1)z^{-1} B_\alpha(z)\big\} \\
& \qquad +E_\alpha^*(z)  \big\{-i E_\alpha(z) + i (2\alpha+1) z^{-1}B_\alpha(z)\big\} \\
& = -i(2\alpha+1)z^{-1} B_\alpha(z) \big\{E_\alpha(z)  - E_\alpha^*(z)\big\}\\
&= -(4\alpha+2)z^{-1} B_\alpha^2(z),
\end{align*}
and hence 
\begin{equation*}
G'(z) = 2k \,z^{2k-1} E_\alpha(z) E_\alpha^*(z) - (4\alpha+2)z^{2k-1} B_\alpha^2(z).
\smallskip
\end{equation*}
Since $|B_\alpha(x)|^2\le |E_\alpha(x)|^2$ we obtain 
\begin{align}\label{asy-gp}
|G'(x)|\le c\big(1+ |x|\big)^{2\theta}
\end{align}
for some $c>0$ and all $x \in \R$.

\smallskip

By hypothesis and \eqref{asy-G} we have that $FG\in L^p(\R)$. Since $FG$ is an entire function of exponential type, the theorem of Plancherel and P\'olya implies that $(FG)'\in L^p(\R)$. Using that
\begin{align*}
\big|F'(x)G(x)\big|^p & = \big|(FG)'(x) - F(x)G'(x)\big|^p \\
& \leq  2^p \left\{\big|(FG)'(x)|^p + \big|F(x)G'(x)\big|^p\right\},
\end{align*}
together with \eqref{asy-gp} we find that $F' G\in L^p(\R)$. By \eqref{Lem17_i} and  \eqref{asy-G} this is equivalent to
$$ F' \in L^p\big(\R, |E_{\nu}(x)|^{-2}\, \dx\big),$$
and the proof is complete.
\end{proof}

\subsection{Majorizing functions of bounded variation}

Given an arbitrary real number $\xi$, we define the real entire function $k_\xi:\C \to \C$ by 
\begin{equation}\label{Sec6_eq0_k}
k_\xi(z) = \frac{K_\nu(\xi,z)}{K_\nu(\xi,\xi)}.
\end{equation}
From \cite[Lemma 12]{HV} we see that $k_\xi$ has exponential type and $\tau(k_\xi) = \tau(E_{\nu}) = 1$. In \cite[Theorem 15]{HV} Holt and Vaaler constructed a real entire function $\ell_\xi:\C \to \C$ of exponential type at most $2$ such that 
\begin{align}\label{k-ell-ineq-new}
|\sgn(x-\xi) - \ell_\xi(x)| \le   k_\xi^2(x)
\end{align}
for all $x \in \R$. The following lemma briefly describes the asymptotic behavior of  $k_\xi(x)$ in both real variables $\xi$ and $x$.

\begin{lemma}
Let $\nu > -1$. For $\xi, x \in \R$, let $k_\xi(x)$ be defined as in \eqref{Sec6_eq0_k}. The following properties hold:
\begin{enumerate}
\item[(i)] For a fixed $\xi \in \R$ we have
\begin{equation}\label{Lem21_i}
\big|k_{\xi}^2(x)\big| \ll_{\xi, \nu} |x|^{-2\nu -3}
\end{equation}
as $|x| \to \infty$.

\smallskip

\item[(ii)] For a fixed $x \in \R$ we have
\begin{equation}\label{Lem21_ii}
\big|k_{\xi}^2(x)\big| \ll_{x,\nu} |\xi|^{2\nu -1}
\end{equation}
as $|\xi| \to \infty$. 
\end{enumerate}
\end{lemma}
\begin{proof}
Both asymptotics follow from \eqref{Intro_Def_K_0},  \eqref{Intro_asymp1}, \eqref{Intro_asymp2} and \eqref{Lem17_i}.
\end{proof}

\smallskip

Let $f:\R \to \R$ be a right continuous function of locally bounded variation (i.e. $f$ has bounded variation on any compact interval $[a,b]\subset \R$). We denote by $V(f)|_{(a,b]}$ the total variation of $f$ on the interval $(a,b]$ and define the total variation function $V_f:\R \to \R$ by
\begin{align}\label{BV-def}
V_f(x) = \begin{cases}
V(f)|_{(0,x]}&\text{ if }x>0,\\
0&\text{ if }x=0,\\
-V(f)|_{(x,0]}&\text{ if }x<0.
\end{cases}
\end{align}
Observe that $V_f$ is a nondecreasing and right continuous function that is locally of bounded variation, and $f(x) = V_f(x)  - (V_f(x) - f(x))$ is a decomposition of $f$ as a difference of two nondecreasing right continuous functions. Associated to $V_f$ we have a unique Borel measure $\sigma_{V_f}$ on $\R$ such that $\sigma_{V_f}((a,b]) = V_f(b) - V_f(a)$ for all $a,b\in\R$, and associated to $f$ we have a signed Borel measure $\sigma_{f} = \sigma_{V_f} - \sigma_{(V_f - f)}$ on the compact subsets of $\R$ such that $\sigma_{f}((a,b]) = f(b) - f(a)$ (note that, in principle, $\sigma_f$ is not a signed Borel measure on the whole $\R$, for we may not be able to properly define $\sigma_f((-\infty, +\infty))$, for instance). The corresponding Lebesgue-Stieltjes integrals against these measures will be denoted by $\d V_f$ and $\d f$ below, and we slightly abuse the notation to write
\begin{equation*}
\int_{-\infty}^\infty g(x) \, \d f(x) := \lim_{M\to \infty} \int_{-M}^{M} g(x) \, \d f(x),
\end{equation*}
provided this limit actually exists. Our next result is inspired by \cite[Theorem 11 and Corollary 12]{V}.

\begin{theorem}\label{general-majorant} 
Let $\nu>-1$. Let $f:\RR\to\RR$ be locally of bounded variation and right continuous. Assume that its total variation function $V_f:\R \to \R$ satisfies
\begin{align}\label{variation-integrable}
\int_{-\infty}^\infty |E_\nu(\xi)|^{-2} \,\d V_f(\xi)<\infty.
\end{align}
Define, for $x \in \R$,
\begin{align}\label{defL}
L_1(x) := f(x) -\frac12 \int_{-\infty}^\infty \big\{\sgn(x-\xi) - \ell_\xi(x)\big\} \,\d f(\xi) -\frac12\int_{-\infty}^\infty k_\xi^2(x) \,\d V_f(\xi)
\end{align}
and
\begin{align}\label{defM}
M_1(x) := f(x) -\frac12 \int_{-\infty}^\infty \big\{\sgn(x-\xi) - \ell_\xi(x)\big\} \,\d f(\xi) + \frac12\int_{-\infty}^\infty k_\xi^2(x) \,\d V_f(\xi).
\end{align}
The following properties hold:
\begin{enumerate}
\item[(i)] The integrals \eqref{defL} and \eqref{defM} are absolutely convergent for all $x \in \R$.

\smallskip

\item[(ii)] There exists a pair of entire functions $L:\C \to \C$ and $M:\C \to \C$ of exponential type at most $2$ such that $L(x) = L_1(x)$ and $M(x) = M_1(x)$ for almost all $x \in \R$. 

\smallskip

\item[(iii)] The restrictions of $L$ and $M$ to $\R$ satisfy
\begin{equation*}
L(x) \le f(x) \leq M(x)
\end{equation*}
for all $x \in \R$ and
\begin{equation*}
\int_{-\infty}^\infty \big\{M(x) - L(x)\big\}\,|E_\nu(x)|^{-2}\,  \dx \le \int_{-\infty}^\infty \frac{1}{K_\nu(\xi,\xi)}\, \d V_f(\xi) <\infty.
\end{equation*}
\end{enumerate}
\end{theorem}

\begin{proof} From \eqref{Lem17_i}, \eqref{Lem21_ii} and \eqref{variation-integrable} we see that  the second integral on the right-hand side of  \eqref{defL} (and \eqref{defM}) is well-defined for all $x \in \R$. The absolute convergence of the first integral on the right-hand side of \eqref{defL} (and \eqref{defM}) is then guaranteed by \eqref{k-ell-ineq-new}. This proves (i).

\smallskip

By construction and \eqref{k-ell-ineq-new} we have that 
\begin{equation*}
L_1(x) \leq f(x) \leq M_1(x)
\end{equation*}
for all $x \in \R$. If we show that $L_1$ and $M_1$ are almost everywhere equal to restrictions of entire functions $L$ and $M$ to $\R$, we will have
\begin{equation*}
L(x) \leq f(x) \leq M(x)
\end{equation*}
for all $x \in \R$, since $f$ is right continuous. Moreover, an application of Fubini's theorem gives us
\begin{align}\label{Sec5_comp_fub}
\begin{split}
\int_{-\infty}^\infty \big\{M(x) - L(x)\big\}\,|E_\nu(x)|^{-2}\,  \dx & = \int_{-\infty}^\infty \big\{M_1(x) - L_1(x)\big\}\,|E_\nu(x)|^{-2}\,  \dx \\
& = \int_{-\infty}^\infty \int_{-\infty}^\infty k_\xi^2(x) |E_{\nu}(x)|^{-2}  \dx \,\d V_f(\xi)\\
& = \int_{-\infty}^\infty \frac{1}{K_{\nu}(\xi,\xi)} \,\d V_f(\xi) <\infty,
\end{split}
\end{align}
where the last inequality follows from \eqref{Intro_asymp1}, \eqref{Intro_asymp2} and \eqref{Lem17_i}. This proves (iii).

\smallskip

We show next that $M_1$ is almost everywhere equal to the restriction of an entire function of exponential type $2$ to $\R$. The proof for $L_1$ will be analogous. From \eqref{k-ell-ineq-new} and \eqref{Lem21_i} we find that $\ell_\xi(x)$ defines a tempered distribution via 
\begin{equation*}
\phi \mapsto \int_\R \phi(x)  \,\ell_\xi(x) \,\dx
\end{equation*} 
for all Schwartz functions $\phi$. Let 
\begin{equation*}
g_{\xi}(z) = \frac{\ell_\xi(z) - \ell_\xi(0)}{z}.
\end{equation*}
 From \eqref{k-ell-ineq-new} and \eqref{Lem21_i} we see that $x \mapsto g_{\xi}(x) \in L^2(\R)$. Since $g_{\xi}$ has exponential type at most $2$, the Paley-Winer theorem implies that the Fourier transform $\widehat{g_{\xi}}$ is supported on $\big[-\tfrac{1}{\pi},\tfrac{1}{\pi} \big]$ and 
 \begin{equation*}
g_{\xi}(z) = \int_{-\tfrac{1}{\pi}}^{\tfrac{1}{\pi}} \widehat{g_{\xi}}(t)\,e^{2\pi i z t }\,\dt,
\end{equation*}
from which we get the bound
\begin{equation*}
|g_{\xi}(z)| \leq C_{\xi}\,e^{2|\im(z)|},
\end{equation*}
and consequently
\begin{equation}\label{Sec5_bound_PWTD}
|\ell_{\xi}(z)| \leq C'_{\xi}\,(1 + |z|)\,e^{2|\im(z)|},
\end{equation}
for all $z \in \C$, where $C_{\xi}$ and $C'_{\xi}$ are positive constants that may depend on $\xi$. With the bound \eqref{Sec5_bound_PWTD} we can use the Paley-Wiener theorem for distributions \cite[Theorem 1.7.7]{Hor} to conclude that $\widehat{\ell_{\xi}}$ is a tempered distribution supported on $\big[-\tfrac{1}{\pi},\tfrac{1}{\pi} \big]$. Therefore, if $\varepsilon>0$ and $\theta$ is a Schwartz function which is equal to zero on $\big[-\tfrac{1}{\pi}-\varepsilon,\tfrac{1}{\pi} + \varepsilon\big]$, we have 
\begin{align}\label{ell-zero-int}
\int_{-\infty}^\infty \ell_\xi (x) \,\widehat{\theta}(x) \,\dx =0
\end{align}
for all $\xi \in \R$. Let $\varphi(t) := -(2\pi it )^{-1} \theta(t)$. Then $\varphi$ is also a Schwartz function and it satisfies $\big(\widehat{\varphi}\big)' =\widehat{\theta}$. Define $D_\xi:\R \to \R$ by
\begin{equation*}
D_\xi(x) = \tfrac{1}{2}\, \big\{\sgn(x-\xi) - \ell_\xi(x)\big\}.
\end{equation*}
Since $x\mapsto \tfrac12 \sgn(x)$ has distributional Fourier transform $t\mapsto (2\pi i t)^{-1}$ for nonzero $t$, from \eqref{ell-zero-int} we obtain
\begin{align*}
-\widehat{\varphi}(\xi) = \int_{-\infty}^\infty \frac{e^{-2\pi i t \xi}}{2\pi i t}\, \theta(t)\, \dt = \frac{1}{2} \int_{-\infty}^\infty \sgn(x-\xi)\, \widehat{\theta}(x) \,\dx = \int_{-\infty}^\infty  D_\xi(x) \,\widehat{\theta}(x)\, \dx.
\end{align*}
We now claim that $f$ has at most linear growth and thus defines a tempered distribution. In fact, this plainly follows from \eqref{Lem17_i} and \eqref{variation-integrable} since
\begin{equation*}
|f(x) - f(1)| \leq \int_{1-}^{x+} \d V_f(\xi) = \int_{1-}^{x+} |\xi|^{-2\nu -1}\,|\xi|^{2\nu +1} \d V_f(\xi) \leq  \int_{1-}^{x+} |\xi|\,|\xi|^{2\nu +1} \d V_f(\xi) \ll |x|,
\end{equation*}
for $x >1$. The same bound holds for $x<-1$. An integration by parts then gives
\begin{align}\label{Sec5_TD_1}
\begin{split}
\int_{-\infty}^\infty  \widehat{\theta}(x) f(x)  \,\dx &= -\int_{-\infty}^\infty \widehat{\varphi}(\xi) \,\d f(\xi)\\
&=\int_{-\infty}^\infty \int_{-\infty}^\infty  D_\xi(x)\, \widehat{\theta}(x) \,\dx \, \d f(\xi)\\
&= \int_{-\infty}^\infty \widehat{\theta}(x) \int_{-\infty}^\infty D_\xi(x)\, \d f(\xi)\, \dx,
\end{split}
\end{align}
where the interchange of the integrals is justified by the absolute convergence given by \eqref{k-ell-ineq-new} and \eqref{Sec5_comp_fub}. Note that $x \mapsto \int_{-\infty}^\infty D_\xi(x)\, \d f(\xi)$ belongs to $L^1(\R, |E_{\nu}(x)|^{-2}\,\dx)$  by \eqref{k-ell-ineq-new} and \eqref{Sec5_comp_fub}, and thus it defines a tempered distribution. From \eqref{Sec5_TD_1} it follows that the distributional Fourier transform of
\[
x\mapsto f(x) -   \int_{-\infty}^\infty D_\xi(x)\, \d f(\xi)
\]
is supported on $\big[-\tfrac{1}{\pi},\tfrac{1}{\pi}\big]$. The converse of the Paley-Wiener theorem for distributions \cite[Theorem 1.7.7]{Hor} implies that this difference is almost everywhere equal to the restriction of an entire function of exponential type $2$ to the real line. 

\smallskip

The same argument can be used to show that the second integral in \eqref{defM}, when viewed as a function of $x$, has Fourier transform supported on $\big[-\tfrac{1}{\pi},\tfrac{1}{\pi}\big]$. In fact, from \eqref{Lem21_i} we already have that $k_{\xi}^2 \in L^2(\R)$. The remaining steps are analogous.
\end{proof}

We now come to the result that originally motivated this section.

\begin{theorem}\label{Cor22}
Let $\nu>-1$ and $G:\C \to \C$ be a real entire function of exponential type at most $2$ such that $G\in L^1(\R, |E_\nu(x)|^{-2}\, \dx)$. Then there exist $U,V\in\H(E_{\nu})$ such that 
\begin{align*}
G(z) = U(z)U^*(z) - V(z)V^*(z)
\end{align*}
for all $z \in \C$.
\end{theorem}

\begin{proof} We define a continuous function $f:\R\to\R$ by
\begin{equation*}
f(x) = \max\{G(x),0\}.
\end{equation*}
Theorem \ref{PP-in-HE} implies that $G' \in L^1(\R, |E_{\nu}(x)|^{-2}\, \dx)$, hence $f$ is locally of bounded variation and the total variation function $V_f$ defined in \eqref{BV-def} satisfies $|E_\nu|^{-2}\in L^1(\R, \d V_f(x))$. Since $f\in L^1(\R,|E_\nu(x)|^{-2} \,\dx)$, by Theorem \ref{general-majorant} there exists an entire function $M$ of exponential type at most $2$ such that $M(x)\ge f(x)$ for all real $x$ and $M\in L^1(\R, |E_{\nu}(x)|^{-2}\, \dx)$. Since $f$ is nonnegative, so is $M$, and it follows from Lemma \ref{nonnegl1-tohe} that there exists $U\in \H(E_\nu)$ such that $M = UU^*$. By construction, $M-G \ge M-f\ge 0$, and another application of Lemma \ref{nonnegl1-tohe} gives $M-G = VV^*$ with $V\in \H(E_\nu)$. It follows that $G = UU^* - VV^*$ as claimed.
\end{proof}


\section{The extremal problem for a class of radial functions}\label{Sec_Radial_Functions}

In this section we shall prove Theorem \ref{Thm2}.

\subsection{Preliminaries} We start by showing that the asymptotic estimate \eqref{Intro_open_asymp} holds, a point left open in the introduction. In fact, this plainly follows from Proposition \ref{Prop10} specialized to the spaces $\H(E_{\nu})$, together with \eqref{Lem17_i}, \eqref{Sec4_eq0_def_L_nu} and \eqref{Sec4_eq0_def_M_nu}.

\smallskip

If the nonnegative Borel measure $\mu$ on $(0,\infty)$ satisfies \eqref{Intro_mu_1} we define
\begin{equation}\label{Sec6_eq1_int1}
D_{\nu}^-(\mu,{\bf x}) = \int_0^{\infty} \Big\{e^{-\lambda |{\bf x}|} - \mc{L}_{\nu}(2,\lambda, \x)\Big\} \, \dmu,
\end{equation}
and if $\mu$ satisfies the more restrictive condition \eqref{Intro_mu_2} we define
\begin{equation}\label{Sec6_eq1_int2}
D_{\nu}^+(\mu,{\bf x}) = \int_0^{\infty} \Big\{\mc{M}_{\nu}(2, \lambda, \x) - e^{-\lambda |{\bf x}|}\Big\} \, \dmu,
\end{equation}
where $\z \mapsto  \mc{L}_{\nu}(2, \lambda, \z)$ and $\z \mapsto  \mc{M}_{\nu}(2, \lambda, \z)$ are the extremal functions of exponential type at most $2$ defined in Theorem \ref{Thm1} (iv). We know a priori that the nonnegative functions $\x \mapsto D_{\nu}^{\pm}(\mu,{\bf x})$ are radial and belong to $L^1(\R^N, |\x|^{2\nu +2 - N}\,\dxx)$, hence the integrals \eqref{Sec6_eq1_int1} and \eqref{Sec6_eq1_int2} converge almost everywhere.

\begin{lemma}\label{Sec6_Lem23}
The following properties hold:
\smallskip
\begin{enumerate}
\item[(i)] If $\mu$ satisfies \eqref{Intro_mu_1} then the nonnegative function $\x \mapsto D_{\nu}^{-}(\mu,{\bf x})$ is continuous for $\x \neq 0$. Moreover, $D_{\nu}^{-}(\mu,{\bf x}) =0 $ for $|{\bf x}| = \xi$ with $A_{\nu}(\xi)=0$.
\smallskip
\item[(ii)] If $\mu$ satisfies \eqref{Intro_mu_2} then the nonnegative function $\x \mapsto D_{\nu}^{+}(\mu,{\bf x})$ is continuous. Moreover, $D_{\nu}^{+}(\mu,{\bf x}) =0 $ for $|{\bf x}| = \xi$ with $B_{\nu}(\xi)=0$.
\end{enumerate}
\end{lemma}

\begin{proof} 
\noindent{\it Part} (i). Recall that $A_{\nu}(0)\neq 0$. Let $g = g_0$ in \eqref{gc-def} with $F = A_{\nu}^2$. Since $A_{\nu}^2$ is even we have that $g$ is even. From \eqref{ml-el-rep} we have
\begin{align*}
e^{-\lambda|x|} - L(A_{\nu}^2, \lambda,x) =  A_{\nu}^2(x) \int_{-\infty}^0 \big\{g(w+\lambda) - g(w-\lambda)\big\} \,e^{-xw} \,\dw,
\end{align*}
for $x\leq0$. Recall that for $w<0$,
\begin{equation*}
g(w+\lambda) - g(w-\lambda)\ge 0.
\end{equation*}
It follows that $H_{\lambda}^{-}$ defined by
\begin{equation*}
H_{\lambda}^{-}(x) =\int_{-\infty}^0 \big\{g(w+\lambda) - g(w-\lambda)\big\} \,e^{-xw} \,\dw
\end{equation*}
is nondecreasing on $(-\infty,0]$ for every $\lambda>0$. For $x\leq0$ we define
\begin{equation*}
d_\nu^-(\mu,x) := A_{\nu}^2(x) \int_0^\infty H_{\lambda}^{-}(x)\, \dmu,
\end{equation*}
and note that this function is continuous on $(-\infty, 0)$ by dominated convergence (at $x=0$ it may blow up). It follows that $d_\nu^-(\mu,\xi)=0$ for every $\xi<0$ with $A_{\nu}(\xi)=0$. Since $D_\nu^-(\mu,\x) = d_\nu^-(\mu,-|\x|)$, our claim follows.

\smallskip

\noindent{\it Part} (ii). Recall that $B_{\nu}^2$ has a double zero at the origin. From \eqref{Ml-interpol-at0} we have the representation
\begin{align*}
& M(B_{\nu}^2,\lambda,x) - e^{-\lambda|x|}\\[0.2em]
& = \frac{B_{\nu}^2(x)}{x^2}\left\{ 2\big(g'(0) - g'(-\lambda)\big) + \int_0^\infty \big\{g''(-w-\lambda) - g''(w-\lambda)\big\}\,e^{xw}\, \dw\right\},
\end{align*}
valid for $x\leq0$, where $g = g_{\alpha_{F}/2}$ in \eqref{gc-def} with $F = B_{\nu}^2$. Since $B_{\nu}^2$ is even it follows that $g''$ is even. Hence, for $w >0$,
\begin{equation*}
g''(-w-\lambda) = g''(w+\lambda) \le g''(w-\lambda)\,,
\end{equation*}
as we have seen in the proof of Proposition \ref{Sec_ILP_Prop5}. This implies that $H_\lambda^+$ defined by
\begin{equation*}
H_\lambda^+(x) = 2\big(g'(0) - g'(-\lambda)\big) + \int_0^\infty \big\{g''(-w-\lambda) - g''(w-\lambda)\big\}\,e^{xw}\, \dw
\end{equation*}
is nonincreasing on $(-\infty,0]$ and satisfies
\begin{equation*}
0\le H_\lambda^+(x) \le 2\big(g'(0) - g'(-\lambda)\big).
\end{equation*}
For $x\leq0$ we define
\begin{equation*}
d_\nu^+(\mu,x) := \frac{B_{\nu}^2(x)}{x^2} \int_0^\infty H_\lambda^+(x)\, \dmu,
\end{equation*}
and note that this function is continuous on $(-\infty, 0]$ by dominated convergence. It follows that $d_\nu^+(\mu,\xi)=0$ for every $\xi <0$ with $B_{\nu}(\xi)=0$. Since $H_\lambda^+(0) =0$ by the proof of \eqref{majorant-interpolation}, we also find that $d_\nu^+(\mu,0)=0$. Since $D_\nu^+(\mu,\x) = d_\nu^+(\mu,-|\x|)$, our claim follows.
\end{proof}

\subsection{Proof of Theorem \ref{Thm2}} 

\subsubsection{Proof of part {\rm (i)}} For $\kappa >0$, recall the definition of the measure $\mu_{\kappa}$ given in \eqref{def_mu_kappa}. We observe that
\begin{align*}
\mc{G}_{\mu}(\kappa\x) & = \int_0^{\infty}\left\{ e^{-\lambda \kappa|\x|} - e^{-\lambda}\right\}\dmu\\[0.2em]
& = \int_0^{\infty}\left\{ e^{-\lambda|\x|} - e^{-\lambda \kappa^{-1}}\right\}\d \mu_{\kappa^{-1}}(\lambda)\\[0.2em]
& =  \int_0^{\infty}\left\{ e^{-\lambda|\x|} - e^{-\lambda}\right\}\d \mu_{\kappa^{-1}}(\lambda) -  \int_0^{\infty}\left\{ e^{-\lambda \kappa^{-1}} - e^{-\lambda}\right\}\d \mu_{\kappa^{-1}}(\lambda)\\[0.2em]
& = \mc{G}_{\mu_{\kappa^{-1}}}(\x) - \mc{G}_{\mu_{\kappa^{-1}}}\big(\kappa^{-1}\,{\bf e_1}\big),
\end{align*}
where ${\bf e_1} = (1,0,0,\ldots,0) \in \R^N$. We conclude that $\z \mapsto  \mc{L}(\z) \in \E_{\delta}^{N-} (\mc{G}_{\mu})$ if and only if $\z \mapsto  \big\{\mc{L}(\kappa\z) + \mc{G}_{\mu_{\kappa^{-1}}}\big(\kappa^{-1}\,{\bf e_1}\big)\big\}  \in \E_{\kappa\delta}^{N-} \left(\mc{G}_{\mu_{\kappa^{-1}}}\right)$. An analogous property holds for the majorants. A simple change of variables then establishes (i).

\subsubsection{Proof of parts {\rm (ii) and (iii)}} We shall prove these two parts together. By part (i) it suffices to consider the case $\delta = 2$.

\subsubsection*{Existence} Given $\nu >-1$ we shall prove the existence of suitable minorants that interpolate $\mc{G}_{\mu}$ at the ``correct" radii. The construction for the majorants will be analogous. 

\smallskip

The strategy to prove this part is as follows. We would like to argue that $\x \mapsto \G_{\mu}(\x)$ and $\x \mapsto D_{\nu}^-(\mu,{\bf x})$ are tempered distributions whose Fourier transforms satisfy
\begin{equation*}
\widehat{\G_{\mu}}(\t) = \widehat{D_{\nu}^-}(\mu,{\bf t})
\end{equation*}
for $|\t| \geq \pi^{-1}$. We would then invoke the Paley-Wiener theorem for distributions to see that $\x \mapsto \big(\G_{\mu}(\x) - D_{\nu}^-(\mu,{\bf x})\big)$ is equal almost everywhere to the restriction of an entire function of exponential type at most $2$ to $\R^N$. This would be our desired minorant.

\smallskip

However, given the parameter $\nu >-1$, it is not clear {\it a priori} that  $\x \mapsto \G_{\mu}(\x)$ and $\x \mapsto D_{\nu}^-(\mu,{\bf x})$ define tempered distributions on $\R^N$ for any $N \geq 1$. We will overcome this technical difficulty by treating first the convenient case $N = \lfloor 2\nu +4\rfloor$ to then extend the construction to general $N$. We recall that the nonnegative Borel measure $\mu$ on $(0,\infty)$ satisfies \eqref{Intro_mu_1} for the minorant problem and \eqref{Intro_mu_2} for the majorant problem.

\smallskip

\noindent{\it Step 1}. The case $N = \lfloor 2\nu +4\rfloor$. 

\smallskip

We claim that our function $\G_{\mu}$ defines a tempered distribution on the class of Schwartz functions on $\R^N$. In fact, for $|\x| \geq 1$, we can use the mean value theorem to get
$$\left|e^{-\lambda |\x|} - e^{-\lambda} \right| \leq \lambda e^{-\lambda} (|\x| -1),$$ 
and then
\begin{align}\label{Sec6_pf_thm3_eq0}
\begin{split}
\G_{\mu}(\x) & = \int_0^{\infty}\left\{ e^{-\lambda |\x|} - e^{-\lambda}\right\}\dmu\\
& \leq (|\x| -1)  \int_0^{\infty}  \lambda e^{-\lambda} \dmu \\
& \ll_{\mu} (|\x| -1). 
\end{split}
\end{align}
\smallskip
Let $B(r) \subset \R^N$ denote the open ball of radius $r$ centered at the origin. We show that $\G_{\mu}$ is integrable in $B(1)$. For $r>0$ we use the two elementary inequalities
$$\left|e^{-\lambda r} -e^{-\lambda}\right| \leq \lambda |1-r|$$
and
$$e^{-\lambda r}\,r^{N-1} \ll_N\ \frac{1}{\lambda^{N-1}}$$
in the following computation
\begin{align}\label{Sec6_pf_thm3_eq1}
\begin{split}
\int_{B(1)} \G_{\mu}(\x) \,\dxx &= \int_{B(1)} \int_0^{\infty}\left\{ e^{-\lambda |\x|} - e^{-\lambda}\right\} \dmu\,\dxx \\
& = \omega_{N-1} \int_0^{\infty} \int_{0}^1 \Big\{e^{-\lambda r} - e^{-\lambda}\Big\}\,r^{N-1}\,\dr\, \dmu\\
& \ll_N  \ \omega_{N-1} \left[ \int_0^{1} \lambda \int_{0}^1 (1-r)\,r^{N-1}\,\dr\, \dmu\right.\\
&  \ \ \ \ \ \ \ \ \  \ \ \ \ \ \ \ \ \ \ \ \ \left. + \int_1^{\infty} \int_{0}^1 \left(\frac{1}{\lambda^{N-1}} + e^{-\lambda} r^{N-1}\right)\dr\,\dmu\right]\\
& \ll_{N,\mu}\  1.
\end{split}
\end{align}
To estimate the last integral we used the fact that $N \geq 2 \nu +3$. The bounds \eqref{Sec6_pf_thm3_eq0} and \eqref{Sec6_pf_thm3_eq1} imply our claim. 

\smallskip

Next we recall that the function $\x \mapsto D_{\nu}^-(\mu,{\bf x})$ defined in \eqref{Sec6_eq1_int1} belongs to $L^1(\R^N, |\x|^{2\nu +2 - N}\,\dxx)$. In particular this function is integrable in $B(1)$ and also defines a tempered distribution.

\smallskip

From \eqref{Sec4_eq0_def_L_nu} we know that 
\begin{equation*}
\L_{\nu}(2,\lambda, \z) := \psi_N \big(L(A_{\nu}^2, \lambda, \cdot)\big)(\z)
\end{equation*}
and the even function $z \mapsto L(A_{\nu}^2, \lambda, z)$ has exponential type at most $2$ and its restriction to $\R$ belongs to $L^1(\R, |x|^{2\nu +1}\,\dx)$. We claim that $z \mapsto L(A_{\nu}^2, \lambda, z)$ must have a zero. Otherwise, it would have to be of the form $e^{az +b}$ with $|a| \leq 2$, and since it is an even function, we would have $a=0$ and the function would be constant. However, this contradicts the fact that its restriction to $\R$ belongs to $L^1(\R, |x|^{2\nu +1}\,\dx)$. If $a$ is a zero of $z \mapsto L(A_{\nu}^2, \lambda, z)$, so is $-a$ and the function
$$J(z) = \frac{L(A_{\nu}^2, \lambda, z)}{(z^2-a^2)}$$
is even and real entire (if $a=0$ it must be a zero of even multiplicity). Therefore the function of $N$ complex variables
$$\mc{J}(\z) := \psi_N(J)(\z) = \frac{\L_{\nu}(2,\lambda, \z) }{(z_1^2 + \ldots + z_N^2-a^2)}$$
is entire. Since $\x \mapsto \L_{\nu}(2, \lambda, \x) \in L^1(\R^N, |\x|^{2\nu +2 - N}\,\dxx)$ and $N = \lfloor 2\nu +4\rfloor$, we find that $\x \mapsto \mc{J}(\x) \in L^1(\R^N)$. Moreover $\mc{J}$ still has exponential type at most $2$. By a theorem of Plancherel and P\'{o}lya \cite[\S47, Theorem III]{PP} we know that $\mc{J}$ is bounded, and therefore it belongs to $L^2(\R^N)$. From the Paley-Wiener theorem in several variables \cite[Chapter III, Theorem 4.9]{SW} we find that $\widehat{\mc{J}}$ is a continuous function supported on $\overline{B}(\pi^{-1})$, and thus
\begin{equation}\label{Sec6_PW1}
\mc{J}(\z) = \int_{\overline{B}(\pi^{-1})} \widehat{\mc{J}}(\t)\,e^{2\pi i \t \cdot \z}\,\d \t.
\end{equation}
If we write $\z = \x + i\y$, from \eqref{Sec6_PW1} we find that 
\begin{equation*}
|\mc{J}(\z)| \leq C \,e^{2|\y|},
\end{equation*}
which implies that 
\begin{equation}\label{Sec6_PWTD_bound}
\left|\L_{\nu}(2, \lambda, \z)\right|  \leq C \left|(z_1^2 + \ldots + z_N^2-a^2)\right|\,e^{2|\y|}.
\end{equation}
The bound \eqref{Sec6_PWTD_bound} allows us to invoke the Paley-Wiener theorem for distributions \cite[Theorem 1.7.7]{Hor} to conclude that $\x \mapsto \L_{\nu}(2, \lambda, \x)$ defines a tempered distribution with compact support contained in the closed ball $\overline{B}(\pi^{-1})$. 

\smallskip

We now have all the ingredients to compute the Fourier transforms of $\x \mapsto \G_{\mu}(\x)$ and $\x \mapsto D_{\nu}^-(\mu,{\bf x})$ outside $\overline{B}(\pi^{-1})$. Writing $\mc{F}_{\lambda} (\x) = e^{-\lambda |\x|}$, its Fourier transform $\widehat{\mc{F}}_{\lambda}$ is given by \eqref{Intro_FTF}. Let $\varepsilon >0$ and $\varphi$ be a Schwartz function on $\R^N$ that vanishes in the closed ball $\overline{B}(\pi^{-1} + \varepsilon)$. We then have
\begin{align}\label{Sec6_Dist_eq1}
\begin{split}
\int_{\R^N} \G_{\mu}(\x) \, \widehat{\varphi}(\x)\, \dxx & =  \int_{\R^N} \int_0^{\infty} \left\{ e^{-\lambda |\x|} - e^{-\lambda} \right\} \widehat{\varphi}(\x)\, \dmu\,\dxx\\
& = \int_0^{\infty} \int_{\R^N} \left\{ e^{-\lambda |\x|} - e^{-\lambda} \right\} \widehat{\varphi}(\x)\,\dxx\, \dmu \\
& = \int_0^{\infty} \int_{\R^N}  \widehat{\mc{F}}_{\lambda}(\y) \,\varphi(\y)\,\d \y\, \dmu\\
& = \int_{\R^N} \left(\int_0^{\infty} \widehat{\mc{F}}_{\lambda}(\y) \,\dmu\right)\varphi(\y)\,\d \y.
\end{split}
\end{align}
We also have
\begin{align}\label{Sec6_Dist_eq2}
\begin{split}
\int_{\R^N} D_{\nu}^-(\mu,\x) \, \widehat{\varphi}(\x)\, \dxx & =  \int_{\R^N} \int_0^{\infty} \Big\{e^{-\lambda |{\bf x}|} - \mc{L}_{\nu}(2,\lambda, \x)\Big\} \widehat{\varphi}(\x)\, \dmu\,\dxx\\
& = \int_0^{\infty} \int_{\R^N}\Big\{e^{-\lambda |{\bf x}|} - \mc{L}_{\nu}(2,\lambda, \x)\Big\}\widehat{\varphi}(\x)\,\dxx\, \dmu \\
& = \int_0^{\infty} \int_{\R^N}  \widehat{\mc{F}}_{\lambda}(\y) \,\varphi(\y)\,\d \y\, \dmu\\
& = \int_{\R^N} \left(\int_0^{\infty} \widehat{\mc{F}}_{\lambda}(\y) \,\dmu\right)\varphi(\y)\,\d \y.
\end{split}
\end{align}
From \eqref{Sec6_Dist_eq1} and \eqref{Sec6_Dist_eq2} we conclude that the tempered distribution defined by the function 
$$\x \mapsto \left(\G_{\mu}(\x) - D_{\nu}^-(\mu,\x) \right)$$ 
has Fourier transform with compact support contained in the closed ball $\overline{B}(\pi^{-1})$. By the converse of the Paley-Wiener theorem for distributions, this function is equal almost everywhere to the restriction to $\R^N$ of an entire function of exponential type at most $2$ that we shall call $\z\mapsto \mc{L}_{\nu}(2, \mu, \z)$. By Lemma \ref{Sec6_Lem23} we find that 
\begin{equation}\label{Sec6_Min_eq1}
\mc{L}_{\nu}(2,\mu,\x) = \G_{\mu}(\x) - D_{\nu}^-(\mu,\x) 
\end{equation}
for all $\x \neq 0$, and that 
\begin{equation}\label{Sec6_Min_eq2}
\mc{L}_{\nu}(2,\mu, \x) = \G_{\mu}(\x)
\end{equation} 
for $|{\bf x}| = \xi$ with $A_{\nu}(\xi)=0$. Since $D_{\nu}^-(\mu,\x)$ is nonnegative we also find that
\begin{equation}\label{Sec6_Min_eq3}
\mc{L}_{\nu}(2, \mu, \x) \leq \G_{\mu}(\x)
\end{equation}
for all $\x \in \R^N$. From \eqref{Sec6_Min_eq1} we obtain
\begin{align}\label{Sec6_Min_eq4}
\begin{split}
 \int_{\R^N} \big\{\mc{G}_{\mu}(\x) & - \mc{L}_{\nu}(2,\mu, \x) \big\}\,|\x|^{2\nu+2-N}\,\dxx =  \int_{\R^N} D_{\nu}^-(\mu,\x)\,|\x|^{2\nu+2-N}\,\dxx\\
& = \int_{\R^N} \int_0^{\infty} \Big\{e^{-\lambda |{\bf x}|} - \mc{L}_{\nu}(2, \lambda, \x)\Big\} |\x|^{2\nu+2-N} \, \dmu \,\dxx\\
& =  \int_0^{\infty} U^{N-}_{\nu}(2,\lambda)\, \dmu.
\end{split}
\end{align}

\smallskip

\noindent{\it Step 2}. The case $N = 1$. 

\smallskip 

The entire minorant $\z\mapsto \mc{L}_{\nu}(2, \mu, \z)$, constructed above for the case $N = \lfloor 2\nu +4\rfloor$, is a radial function when restricted to $\R^N$ by \eqref{Sec6_Min_eq1}. Therefore, for any $M \in SO(N)$ we find that $\mc{L}_{\nu}(2, \mu, M\x) = \mc{L}_{\nu}(2, \mu, \x)$ and thus  
$$\widetilde{\mc{L}_{\nu}}(2, \mu, \x) = \mc{L}_{\nu}(2, \mu, \x)$$
for all $\x \in \R^N$. This implies that 
$$\widetilde{\mc{L}_{\nu}}(2, \mu, \z) = \mc{L}_{\nu}(2, \mu, \z)$$
for all $\z \in \C^N$, and by Lemma \ref{Lemma19_HV} we have the power series expansion
\begin{equation*}
\mc{L}_{\nu}(2, \mu, \z) = \sum_{k=0}^{\infty} c_k (z_1^2 + \ldots + z_N^2)^k.
\end{equation*}
Define the function $L_{\nu}(2,\mu,\cdot):\C \to \C$ by
\begin{equation}\label{Sec6_const_min}
L_{\nu}(2, \mu, z) = \sum_{k=0}^{\infty} c_k z^k.
\end{equation}
By Lemma \ref{type-to-ntype} we know that $z \mapsto L_{\nu}(2, \mu,z)$ is an entire function of exponential type at most $2$. It is easy to see that properties \eqref{Sec6_Min_eq1}, \eqref{Sec6_Min_eq2}, \eqref{Sec6_Min_eq3} and \eqref{Sec6_Min_eq4} continue to hold in the case $N=1$ with the minorant $z \mapsto L_{\nu}(2, \mu, z)$.

\smallskip

\noindent{\it Step 3}. The case $N \geq 1$. 

\smallskip

For general $N \geq 1$ we now define
$$\mc{L}_{\nu}(2, \mu, \z)  = \psi_N\big(L_{\nu}(2, \mu,\cdot)\big)(\z).$$
This function has exponential type at most $2$ by Lemma \ref{type-to-ntype} and verifies properties \eqref{Sec6_Min_eq1}, \eqref{Sec6_Min_eq2}, \eqref{Sec6_Min_eq3} and \eqref{Sec6_Min_eq4}.

\smallskip

For the majorants, the construction follows along the same lines, using $D_{\nu}^+(\mu,{\bf x})$ defined in \eqref{Sec6_eq1_int2}.

\subsubsection*{Optimality} We now prove that the minorants and majorants constructed above are in fact the best possible. We start by observing that the exact same argument as in the proof of Theorem \ref{Thm1} (ii) gives
$$U^{N\pm}_{\nu}(\delta,\mu) = \tfrac12 \,\omega_{N-1}\, U^{1\pm}_{\nu}(\delta,\mu).$$
Therefore it suffices to prove optimality in the case $N=1$ and $\delta = 2$. Let us do the minorant case, and the majorant case will be analogous.

\smallskip

Let $L :\C \to \C$ be a real entire function of exponential type at most $2$ such that 
\begin{equation}\label{Sec6_cond_min}
L(x) \leq \G_{\mu}(x)
\end{equation} 
for all $x \in \R$ and 
\begin{equation*}
 \int_{-\infty}^{\infty} \big\{\mc{G}_{\mu}(x) - L(x)\big\}\,|x|^{2\nu+1}\,\dx < \infty.
\end{equation*}
Let $L_{\nu}(2, \mu,\cdot):\C \to \C$ be the minorant of $\mc{G}_{\mu}$ constructed in \eqref{Sec6_const_min}. Then the function
$$F(z) = L_{\nu}(2,\mu,z) - L(z)$$
is real entire of exponential type at most $2$ and belongs to $L^1\big(\R, |x|^{2\nu +1}\,\dx\big)$. By \eqref{Lem17_i} we then have $F \in L^1\big(\R, |E_{\nu}(x)|^{-2}\,\dx\big)$. By Theorem \ref{Cor22} there exist $U,V \in \mc{H}(E_{\nu})$ such that 
$$F(z) = U(z)U^*(z) - V(z)V^*(z).$$
Using now the key identity \eqref{Lem17_ii} we find
\begin{align}\label{Sec6_Id_min_final}
\begin{split}
\int_{-\infty}^{\infty} F(x)\,|x|^{2\nu+1}\,\dx & = \int_{-\infty}^{\infty} \big\{|U(x)|^2 - |V(x)|^2\big\}\,|x|^{2\nu+1}\,\dx\\
& = c_{\nu}^{-1} \int_{-\infty}^{\infty} \big\{|U(x)|^2 - |V(x)|^2\big\}\,|E_{\nu}(x)|^{-2}\,\dx\\
& = c_{\nu}^{-1}\sum_{A_{\nu}(\xi)=0} \frac{\big\{|U(\xi)|^2 - |V(\xi)|^2\big\}}{K_{\nu}(\xi,\xi)}\\
& = c_{\nu}^{-1}\sum_{A_{\nu}(\xi)=0} \frac{F(\xi)}{K_{\nu}(\xi,\xi)}\\
& \geq 0,
\end{split}
\end{align}
since, by \eqref{Sec6_Min_eq2} and \eqref{Sec6_cond_min},
\begin{align*}
F(\xi) = L_{\nu}(2, \mu,\xi) - L(\xi) = \G_{\mu}(\xi) - L(\xi) \geq 0
\end{align*}
for all $\xi \in \R$ with $A_{\nu}(\xi)=0$. Inequality \eqref{Sec6_Id_min_final} is plainly equivalent to 
\begin{equation*}
 \int_{-\infty}^{\infty} \big\{\mc{G}_{\mu}(x) - L(x)\big\}\,|x|^{2\nu+1}\,\dx \geq  \int_{-\infty}^{\infty} \big\{\mc{G}_{\mu}(x) - L_{\nu}(2, \mu,x)\big\}\,|x|^{2\nu+1}\,\dx,
\end{equation*}
which concludes the proof.


\section{Hilbert-type inequalities}
\subsection{Proof of Theorem \ref{Thm4}} Recall that we write $\mc{F}_{\lambda}(\x) =e^{-\lambda |\x|}$. Let $2\nu +2 - N = 2r$, where $r$ is a nonnegative integer. In Theorem \ref{Thm1} (iv) we have produced a minorant  $\z \mapsto \mc{L}_{\nu}(2\pi\delta, \lambda, \z) \in \E_{2\pi\delta}^{N-} (\mc{F}_{\lambda})$ such that 
\begin{equation}\label{Sec7_eq1}
\x \mapsto \mc{L}_{\nu}(2\pi\delta,\lambda, \x)  \in L^1(\R^N, |\x|^{2r}\,\dxx).
\end{equation} 
In particular $\x \mapsto \mc{L}_{\nu}(2\pi\delta,\lambda, \x)  \in L^1(\R^N)$. By a theorem of Plancherel and P\'{o}lya \cite[\S47, Theorem III]{PP} we know that $\x \mapsto \mc{L}_{\nu}(2\pi\delta,\lambda, \x)$ is bounded, and therefore it belongs to $L^2(\R^N)$. By the Paley-Wiener theorem \cite[Chapter III, Theorem 4.9]{SW}, the Fourier transform
\begin{equation*}
\widehat{\mc{L}_{\nu}}(2\pi\delta,\lambda, \t) = \int_{\R^N}  \mc{L}_{\nu}(2\pi\delta, \lambda, \x) \, e^{-2\pi i \x \cdot \t}\,\dxx
\end{equation*}
is a continuous function supported on the closed ball $\overline{B}(\delta) = \{\t \in \R^N; |\t| \leq \delta\}$. From \eqref{Sec7_eq1} we know that $\t \mapsto \widehat{\mc{L}_{\nu}}(2\pi\delta,\lambda, \t)$ has continuous partial derivatives of all orders less than or equal to $2r$. In particular, 
\begin{equation*}
(4\pi^2)^{-r}\,(-\Delta)^r \,\widehat{\mc{L}_{\nu}}(2\pi\delta,\lambda, \t) = \int_{\R^N}  \mc{L}_{\nu}(2\pi\delta,\lambda, \x) \,|\x|^{2r}\, e^{-2\pi i \x \cdot \t}\,\dxx
\end{equation*}
is also supported on the closed ball $\overline{B}(\delta)$. We then have
\begin{align*}
0 & \leq \int_{\R^N} \big\{\mc{F}_{\lambda}(\x) -  \mc{L}_{\nu}(2\pi\delta,\lambda, \x)\big\}\, |\x|^{2r}\,\left|\sum_{j=1}^M a_j\, e^{-2\pi i \x\cdot \y_j}\right|^2\,\dxx\\
& = \sum_{j,l=1}^M a_j\, \overline{a_l}\,\int_{\R^N} \big\{\mc{F}_{\lambda}(\x) -  \mc{L}_{\nu}(2\pi\delta, \lambda, \x)\big\}\, |\x|^{2r}\, e^{-2\pi i \x\cdot (\y_j- \y_l)}\,\dxx\\
& =U_{\nu}^{N-}(2\pi\delta, \lambda) \ \sum_{j=1}^M |a_j|^2 + \sum_{\stackrel{j,l=1}{j\neq l}}^{M}a_j\, \overline{a_l}\,(4\pi^2)^{-r}\,(-\Delta)^r \,\widehat{\mc{F}_{\lambda}}(\y_j - \y_l).
\end{align*}
Therefore we arrive at
\begin{equation}\label{Sec7_eq2}
- U_{\nu}^{N-}(2\pi\delta, \lambda) \ \sum_{j=1}^M |a_j|^2 \leq \sum_{\stackrel{j,l=1}{j\neq l}}^{M}a_j\, \overline{a_l}\, (4\pi^2)^{-r}\,(-\Delta)^r \,\widehat{\mc{F}_{\lambda}}(\y_j - \y_l).
\end{equation}
If we integrate both sides of \eqref{Sec7_eq2} with respect to $\dmu$ we arrive at
\begin{equation*}
- U_{\nu}^{N-}(2\pi\delta, \mu) \ \sum_{j=1}^M |a_j|^2 \leq \sum_{\stackrel{j,l=1}{j\neq l}}^{M}a_j\, \overline{a_l}\,\,\mc{Q}_{\mu,r}(\y_j - \y_l),
\end{equation*}
and this proves (i). The proof of (ii) follows along the same lines using the extremal majorants $\z \mapsto \mc{M}_{\nu}(2\pi\delta, \lambda, \z)$.


\section{Periodic analogues}\label{periodic_analogues}

The general interpolation theory developed in Section \ref{polya} can also be used to treat the problem of optimal one-sided approximation of even periodic functions by trigonometric polynomials. In this section we provide a brief account of these methods.

\subsection{Reproducing kernel Hilbert spaces of polynomials}

We start by recalling some facts about reproducing kernel Hilbert spaces of polynomials as in \cite{Li, LV} (see also \cite{G} for an approach based on the argument principle). In what follows we denote by $\ud \subset \C$ the open unit disc and by $\uc$ the unit circle. Let $n\in\Z^{+}$ and define $\mathcal{P}_n$ to be the set of polynomials of degree at most $n$ with complex coefficients. If $P \in \mc{P}_n$ we define the polynomial $P^{*,n}$ by
\begin{equation}\label{antiunitary_map}
P^{*,n}(z) = z^n \, \overline{P\big(\bar{z}\,^{-1}\big)}.
\end{equation}
When $P$ is a polynomial of exact degree $n$, we shall usually omit the superscript $n$ and write only $P^*$ for simplicity. Part of the notation below is inspired in the notation already used for the de Branges spaces presented in the introduction (specially the conjugation $*$ and the basic functions $K$, $A$ and $B$) but the difference of context between reproducing kernel Hilbert spaces of entire functions and reproducing kernel Hilbert spaces of polynomials should be clear.
\smallskip

Let $P$ be a polynomial of exact degree $n+1$ such that
\begin{align}\label{P-ineq}
|P^*(z)|<|P(z)|
\end{align}
for all $z\in\ud$. For our purposes we assume furthermore that $P$ has no zeros on the unit circle. We consider the Hilbert space $\dbpn$ consisting of the elements in $\mathcal{P}_n$ with scalar product
$$\langle Q,R \rangle_{\dbpn} = \int_{-\hh}^{\hh} Q(z) \,\overline{R(z)} \, |P(z)|^{-2}\,\d\theta,$$
where $z = e^{2\pi i \theta}$. The function 
$$K(w,z) = \frac{P(z)\overline{P(w)} - P^*(z)\overline{P^*(w)}}{1-\bar{w}z}$$
 is the reproducing kernel of  the space $\dbpn$ as it verifies the identity
$$\langle Q, K(w,\cdot)\rangle_{\dbpn} = Q(w)$$
for all $w \in \C$. If we define $A$ and $B$ by
\begin{align*}
A(z) &= \frac12 \,\big\{P(z)+P^*(z)\big\},\\
B(z) &= \frac{i}{2}\,\big\{ P(z) - P^*(z)\big\},
\end{align*}
then $A(z)=A^*(z)$, $B(z) = B^*(z)$ and $P(z) = A(z) - i B(z)$. The inequality \eqref{P-ineq} implies that $A$ and $B$ have zeros only on the unit circle. The reproducing kernel has the alternative representation
\begin{equation}\label{Sec8_Rep2}
K(w,z) = \frac2i \,\frac{B(z) \overline{A(w)} - A(z)\overline{B(w)} }{ 1-\bar{w} z}.
\end{equation}
We note that the coefficient of $z^0$ and the coefficient of $z^{n+1}$ of $P$ cannot have equal absolute values, for this would contradict \eqref{P-ineq} at $z=0$. Hence $\tau P -\bar{\tau} P^*$ has degree $n+1$ for every $\tau\in\uc$. In particular, $A$ and $B$ have degree $n+1$. Since
$$K(w,w) = \langle K(w,\cdot),K(w,\cdot)\rangle_{\dbpn},$$
$K(w,w)=0$ would imply that $K(w,z) =0$ for all $z$, hence $Q(w)=0$ for every $Q\in\mathcal{P}_{n}$ by the reproducing kernel identity, which is not possible. Therefore $K(w,w)>0$ for all $w\in\CC$. From the representation \eqref{Sec8_Rep2} it follows that $A$ and $B$ have only simple zeros and their zeros never agree. 

\smallskip

From \eqref{Sec8_Rep2} we see that the $n+1$ polynomials $\{z \mapsto K(\zeta,z);\  B(\zeta) =0\}$ form an orthogonal basis for $\dbpn$ and, in particular, we arrive at Parseval's formula (see \cite[Theorem 2]{Li})
\begin{align}\label{poly-parseval}
||Q||^2_{\dbpn}= \sum_{B(\zeta) =0} \frac{|Q(\zeta)|^2}{K(\zeta,\zeta)}.
\end{align}
An analogous expression holds if we consider the zeros of $A$.


\subsection{Orthogonal polynomials on the unit circle} We now recall a few facts about the classical theory of orthogonal polynomials on the unit circle. In doing so, we follow the notation of \cite{S} to facilitate some of the references. Related material on the topics presented here can be found in \cite{LV, Sz}. The quadrature formula of this section can also be found in \cite{G, JNT}.

\smallskip

Let $\vartheta$ be a nontrivial probability measure on $\uc$ (we say that $\vartheta$ is trivial if it has finite support). We define the {\it monic orthogonal polynomials} $\Phi_n(z) = \Phi_n(z;\d\vartheta)$ by the conditions 
$$\Phi_n(z) = z^n + \text{lower order terms}\,; \qquad \int_{\uc} \Phi_n(z)\, \bar{z}^j\, \d\vartheta(z)=0\qquad (0\le j<n);$$
and we define the {\it orthonormal polynomials} $\varphi_n := \Phi_n/||\Phi_n||_2$. The scalar product on $L^2(\uc, \d\vartheta)$ is given by
$$\langle f,g\rangle_{L^2(\uc, \d\vartheta)} = \int_{\uc} f(z)\,\overline{g(z)}\, \d\vartheta(z),$$
and we drop the subscript if the $L^2$-space is clear. Observe that 
\begin{equation}\label{antiunitary_2}
\langle Q^{*,n},R^{*,n}\rangle = \langle R,Q\rangle
\end{equation} 
for all polynomials $Q,R \in \mc{P}_n$, where the map $*$ was defined in \eqref{antiunitary_map}. The next lemma collects the relevant facts for our purposes.

\begin{lemma} \label{Sec8_Lem25} 
Let $\vartheta$ be a nontrivial probability measure on $\uc$. 
\begin{enumerate}
\item[(i)]  If $\vartheta(A) = \vartheta(\overline{A})$ for every Borel set  $A \subset \uc$, then $\Phi_n$ has real coefficients. 
\smallskip
\item[(ii)] $\varphi_n$ has all its zeros in $\ud$ and $\varphi_n^*$ has all its zeros in $\CC\backslash\overline{\ud}$. 
\smallskip
\item[(iii)] Define a new measure $\vartheta_n$ on $\uc$ by
$$\d\vartheta_n (z)= \frac{\d\theta}{\big|\varphi_n(e^{2\pi i\theta}; \d\vartheta)\big|^2},$$
where $z = e^{2\pi i \theta}$, $-\hh \leq \theta < \hh$. 
Then $\vartheta_n$ is a probability measure on $\uc$, $\varphi_j(z;\d\vartheta) =  \varphi_j(z;\d\vartheta_n)$ for $j=0,1,\ldots,n$ and for all $Q,R \in \mathcal{P}_{n}$ we have
\begin{equation}\label{Sec8_equ_measures}
\langle Q, R \rangle_{L^2(\uc,\d\vartheta)} = \langle Q, R \rangle_{L^2(\uc,\d\vartheta_n)}.\end{equation}
\end{enumerate}
\end{lemma}

\begin{proof}
(i) Let $Q \in \mathcal{P}_{n-1}$. We have
\begin{align*}
0 = \langle Q(z),\Phi_n(z) \rangle = \langle Q(1/z), \Phi_n(1/z) \rangle = \langle z^n Q(1/z), z^n\Phi_n(1/z)\rangle.
\end{align*}
Since $z^nQ(1/z)$ can be any polynomial of the form $\sum_{j=0}^{n-1} b_j z^{j+1}$, we obtain that $z^n\Phi_n(1/z)$ is orthogonal to $z^j$ for $j=1,\ldots,n$. From \eqref{antiunitary_2} there exists $c \in \C$ so that
\begin{equation}\label{Sec8_eq1_comparison}
\Phi_n^*(z) = cz^n \Phi_n(1/z).
\end{equation}
Since $\Phi_n$ is monic, when we compare the coefficient of $z^0$ on both sides of \eqref{Sec8_eq1_comparison} we find that $c=1$. This implies the result.

\smallskip

\noindent (ii) This is \cite[Theorem 4.1]{S}.

\smallskip

\noindent (iii) This follows from \cite[Theorem 2.4, Proposition 4.2 and Theorem 4.3]{S}.
\end{proof}

By Lemma \ref{Sec8_Lem25} (ii) and the maximum principle we have
\begin{align*}\label{HB-analogue}
|\varphi_{n+1}(z)| < |\varphi_{n+1}^*(z)|
\end{align*}
for all $z\in \ud$. By Lemma \ref{Sec8_Lem25} (iii) we note (Christoffel-Darboux formula) that $\mathcal{P}_n$ with the scalar product $\langle \cdot,\cdot\rangle_{L^2(\uc,\d\vartheta)}$ is a reproducing kernel Hilbert space with reproducing kernel
\begin{equation}\label{Sec8_defKn}
K_{n}(w,z) =\frac{\varphi^*_{n+1}(z)\, \overline{\varphi_{n+1}^*(w)} - \varphi_{n+1}(z)\,\overline{\varphi_{n+1}(w)}}{1-\bar{w}z}.
\end{equation}
Observe that $\varphi^*_{n+1}$ plays the role of $P$ in the previous subsection. As before we define the two polynomials
\begin{align}
A_{n+1}(z) &= \frac{1}{2}\big\{\varphi_{n+1}^*(z) + \varphi_{n+1}(z)\big\},\label{Sec8_DefAn}\\
B_{n+1}(z) &= \frac{i}{2}\big\{ \varphi_{n+1}^*(z) - \varphi_{n+1}(z)\big\},\label{Sec8_DefBn}
\end{align}
and we note that \eqref{poly-parseval} holds.

\smallskip

Define by $\Gamma_{n}$ the space of rational functions $W$ of the form
$$W(z) = \sum_{k=-n}^n a_k z^k\,,$$
where $a_k\in\CC$. 

\begin{corollary}\label{Sec8_Cor26}
Let $\vartheta$ be a nontrivial probability measure on $\uc$ and $W\in\Gamma_{n}$. Then with $z=e^{2\pi i \theta}$ we have
$$\int_{\uc} W(z)\,\d\vartheta(z) = \int_{-\hh}^{\hh} \frac{W(z)}{|\varphi_{n+1}(z;\d\vartheta)|^2}\,\d\theta = \sum_{B_{n+1}(\zeta)=0} \frac{W(\zeta)}{K_{n}(\zeta,\zeta)}.$$ 
A similar formula holds if we consider the zeros of $A_{n+1}$.
\end{corollary}

\begin{proof} Assume first that $W$ is real valued on $\uc$, i.e. $a_k = \overline{a_{-k}}$. Let
$$\tau = \min_{z\in\uc} W(z).$$
Then $W(z) - \tau\ge 0$ and $z \mapsto W(z) -\tau\in \Gamma_{n}$. The Riesz-F\'{e}jer theorem implies that there exists a polynomial $Q \in \mc{P}_n$ such that
$$W(z) = |Q(z)|^2 + \tau$$
on the unit circle. Writing $\tau= |\tau_1|^2 - |\tau_2|^2$, $z=e^{2\pi i \theta}$, and using \eqref{poly-parseval} and \eqref{Sec8_equ_measures}, we obtain 
\begin{align*}
\int_{\uc} W(z)\,\d\vartheta(z) & = \int_{-\hh}^{\hh} \frac{W(z)}{|\varphi_{n+1}(z;\d\vartheta)|^2}\,\d\theta \\
& = \sum_{B_{n+1}(\zeta) =0} \frac{|Q(\zeta)|^2+|\tau_1|^2 - |\tau_2|^2}{K_n(\zeta,\zeta)}\\
& = \sum_{B_{n+1}(\zeta) =0} \frac{W(\zeta)}{K_n(\zeta,\zeta)}.
\end{align*}
The general statement follows by writing $W(z) = W_1(z)- iW_2(z)$, with $W_1(z) =  \sum_{k=-n}^n b_k z^k$ and $W_2(z) =  \sum_{k=-n}^n c_k z^k$, where $b_k = \frac12(a_k + \overline{a_{-k}})$ and $c_k =  \frac i 2 (a_k - \overline{a_{-k}})$. 
\end{proof}


\subsection{Extremal trigonometric polynomials I} 
\subsubsection{Main statement} The map $\theta \mapsto e^{2\pi i \theta}$ allows us to identify measures on $\R/\Z$ with measures on the unit circle $\uc$. Throughout the rest of this section we let $\vartheta$ be a nontrivial {\it even} probability measure on $\R/\Z$, and thus the corresponding nontrivial probability measure on the unit circle $\uc$ (that we keep calling $\vartheta$) is even with respect to angle zero, i.e. for a Borel set $A\subset\uc$ we have
$$\vartheta(A) = \vartheta(\overline{A}),$$
where $\overline{A} = \{\overline{z}:z\in A\}$. 

\smallskip

For $\lambda >0$, consider the $1$-periodic function $f_{\lambda}:\R \to \R$ defined by 
\begin{align*}
f_\lambda(\theta) = \frac{\cosh(\lambda(\theta-\lfloor \theta \rfloor-\frac12))}{\sinh(\frac{\lambda}{2})} = \sum_{j\in\ZZ} e^{-\lambda|\theta+j|}.
\end{align*}
This is the periodization of the exponential function $\mc{F}_{\lambda}(x) = e^{-\lambda |x|}$. We now address the problem of majorizing and minorizing  $f_{\lambda}$ by trigonometric polynomials of a given degree $n$, in a way to minimize the $L^1(\R/\Z,\d\vartheta)$-error. This problem with respect to the Lebesgue measure was treated in \cite[Section 6]{CV2}. Here, of course, a trigonometric polynomial $m(\theta)$ of degree at most $n$ is a $1$-periodic function of the form 
$$m(\theta) = \sum_{k=-n}^{n} a_k \, \!e^{2 \pi i k \theta},$$
where $a_k \in \C$. We say the $m(\theta)$ is a real trigonometric polynomial if it is real for real $\theta$. Let us denote the space of trigonometric polynomials of degree at most $n$ by $\Lambda_n$. 

\begin{theorem}\label{Sec8_Thm27}
Let $n \in \Z^+$ and $\vartheta$ be a nontrivial {\it even} probability measure on $\R/\Z$. Let $\varphi_{n+1}(z) = \varphi_{n+1}(z;\d\vartheta)$ be the $(n+1)$-th orthonormal polynomial on the unit circle with respect to this measure and consider $K_n, A_{n+1}, B_{n+1}$ as defined in \eqref{Sec8_defKn}, \eqref{Sec8_DefAn} and \eqref{Sec8_DefBn}. Let $\mc{A}_{n+1} =\big\{\xi\in \R/\Z: A_{n+1}(e^{2\pi i \xi}) =0\big\}$ and $\mc{B}_{n+1} = \big\{\xi\in\R/\Z: B_{n+1}(e^{2\pi i \xi}) =0\big\}$.

\smallskip

\begin{itemize}
\item[(i)] If $\ell \in \Lambda_n$ satisfies 
\begin{equation}\label{Sec8_ineq_min1}
\ell(\theta) \leq f_\lambda(\theta)
\end{equation} 
for all $\theta \in \R$ then
\begin{equation}\label{Sec8_ineq_min2}
\int_{\R/\Z} \ell(\theta) \, \d\vartheta(\theta) \leq \sum_{\xi \in \mc{A}_{n+1}} \frac{f_{\lambda}(\xi)}{K_{n}(e^{2\pi i \xi},e^{2\pi i \xi})}.
\end{equation}
Moreover, there exists a unique trigonometric polynomial $\theta \mapsto \ell_{\vartheta}(n, \lambda, \theta) \in \Lambda_n$ satisfying \eqref{Sec8_ineq_min1} for which the equality in \eqref{Sec8_ineq_min2} holds.

\smallskip

\item[(ii)] If $m \in \Lambda_n$ satisfies 
\begin{equation}\label{Sec8_ineq_maj1}
m(\theta) \geq f_\lambda(\theta)
\end{equation} 
for all $\theta \in \R$ then
\begin{equation}\label{Sec8_ineq_maj2}
\int_{\R/\Z} m(\theta) \, \d\vartheta(\theta) \geq \sum_{\xi \in \mc{B}_{n+1}} \frac{f_{\lambda}(\xi)}{K_{n}(e^{2\pi i \xi},e^{2\pi i \xi})}.
\end{equation}
Moreover, there exists a unique trigonometric polynomial $\theta \mapsto m_{\vartheta}(n, \lambda, \theta) \in \Lambda_n$ satisfying \eqref{Sec8_ineq_maj1} for which the equality in \eqref{Sec8_ineq_maj2} holds.
\end{itemize}
\end{theorem}

\subsubsection{Interpolation lemma} Before we move on to the proof of Theorem \ref{Sec8_Thm27} we present a lemma that connects the interpolation theory developed in Section \ref{polya} to the optimal approximations by trigonometric polynomials. In the lemma below and its proof we keep the notation already used in Section \ref{polya}. 

\begin{lemma}\label{Sec8_Lem28}
Let $F$ be an even and $1$-periodic Laguerre-P\'{o}lya function of exponential type $\tau(F)$. Assume that $F$ is non-constant.
\begin{itemize}
\item[(i)] If $F(0)>0$, define for $\theta \in \R$
$$\ell(F,\lambda,
\theta) = \sum_{j\in\ZZ} L(F,\lambda,\theta+j).$$
Then $\theta \mapsto \,\ell(F,\lambda, \theta)$ is a trigonometric polynomial of degree less than $\tau(F)/2\pi$ satisfying
\begin{align}\label{trig-ineq-ell}
F(\theta)\,\big\{f_\lambda(\theta) - \ell(F,\lambda,\theta)\big\} \ge 0
\end{align}
for all $\theta \in \R$ and
\begin{equation}\label{trig-ineq-ell-eq}
 \ell(F,\lambda,\xi) = f_\lambda(\xi)
 \end{equation}
for all $\xi \in \R$ with $F(\xi)=0$. 

\smallskip

\item[(ii)] If $F$ has a double zero at the origin and $F(\alpha_F/2) >0$, define for $\theta \in \R$
$$m(F,\lambda,
\theta) = \sum_{j\in\ZZ} M(F,\lambda,\theta+j).$$
Then $\theta \mapsto m(F,\lambda, \theta)$ is a trigonometric polynomial of degree less than $\tau(F)/2\pi$ satisfying
\begin{align}\label{trig-ineq-m}
F(\theta)\,\big\{m(F,\lambda,\theta) - f_\lambda(\theta) \big\} \ge 0
\end{align}
for all $\theta \in \R$ and
\begin{equation}\label{trig-ineq-m-eq}
m(F,\lambda,\xi) = f_\lambda(\xi)
\end{equation}
for all $\xi \in \R$ with $F(\xi)=0$. 
\end{itemize}
\end{lemma}

\begin{proof} As is well known, the assumptions on $F$ imply that $F$ is a trigonometric polynomial of degree at most $\lfloor \tau(F)/(2\pi) \rfloor$. 

\smallskip

 {\it Part} (i). From \eqref{al-growth} and \eqref{mf-integral-small-l} (note that $\mc{N}(F) = \infty$) we observe that $z \mapsto L(F,\lambda,z)$ has exponential type at most $\tau(F)$ and belongs to $L^1(\R)$. The classical result of Plancherel and P\'{o}lya \cite{PP} gives us
\begin{equation}\label{Sec8_PP1}
\sum_{j=-\infty}^{\infty} |L(F,\lambda,\alpha_j)| \leq C_1(\varepsilon, \tau(F)) \int_{-\infty}^{\infty} |L(F,\lambda,x)|\,\dx
\end{equation}
for any increasing sequence $\{\alpha_j\}$ of real numbers such that $\alpha_{j+1} - \alpha_j \geq \varepsilon>0$, and 
\begin{equation}\label{Sec8_PP2}
\int_{-\infty}^{\infty} |L'(F,\lambda,x)|\,\dx \leq C_2(\tau(F)) \int_{-\infty}^{\infty} |L(F,\lambda,x)|\,\dx.
\end{equation}
Estimate \eqref{Sec8_PP1} shows that $L(F,\lambda, \cdot)$ is bounded on $\R$ and thus belongs to $L^2(\R)$. The Paley-Wiener theorem then implies that $\widehat{L}(F,\lambda, \cdot)$ is a continuous function supported on the interval $[-\frac{\tau(F)}{2\pi},\frac{\tau(F)}{2\pi}]$. The bound \eqref{Sec8_PP2} implies that $L(F,\lambda, \cdot)$ has bounded variation on $\R$ and thus the Poisson summation formula (see \cite[Volume 1, Chapter II, Section 13]{Z}) holds as a pointwise identity
\begin{equation}\label{Sec8_PP3} 
\sum_{j\in\ZZ} L(F,\lambda,\theta+j) = \sum_{k\in\ZZ} \widehat{L}(F,\lambda, k)\,e^{2\pi i k \theta}.
\end{equation}
From \eqref{Sec8_PP1}, it follows that the sum on the left of \eqref{Sec8_PP3} is absolutely convergent, and this shows that $\ell(F,\lambda,
\theta)$ is a trigonometric polynomial of degree less than $\tau(F)/2\pi$. From \eqref{minorant-ineq} and \eqref{minorant-interpolation} we obtain \eqref{trig-ineq-ell} and \eqref{trig-ineq-ell-eq}.

\smallskip

\noindent{\it Part} (ii). The majorant part is analogous.

\end{proof}

\subsubsection{Proof of Theorem \ref{Sec8_Thm27}}

\subsubsection*{Existence} Define
\begin{align*}
\frak{A}_{n+1}(\theta) &= A_{n+1}(e^{2\pi i \theta}) \,A_{n+1}(e^{-2\pi i \theta}),\\
\frak{B}_{n+1}(\theta) &= - B_{n+1}(e^{2\pi i \theta})\, B_{n+1}(e^{-2\pi i \theta}).
\end{align*}
Observe that $\frak{A}_{n+1}$ and $\frak{B}_{n+1}$ are even trigonometric polynomials of degree $n+1$. Since $\varphi_{n+1}$ has real coefficients, $A_{n+1}$ has real coefficients and $B_{n+1}$ has purely imaginary coefficients. Thus 
\begin{align*}
A_{n+1}(e^{-2\pi i \theta})& = \overline {A_{n+1}(e^{2\pi i \theta})},\\
B_{n+1}(e^{-2\pi i \theta})& = -\overline {B_{n+1}(e^{2\pi i \theta})},
\end{align*}
for $\theta \in \R$, and we see that $\frak{A}_{n+1}$ and $\frak{B}_{n+1}$ are nonnegative on the real axis. Moreover, since $A_{n+1}$ and $B_{n+1}$ have only simple zeros on the unit circle, $\frak{A}_{n+1}$ and $\frak{B}_{n+1}$ have only double zeros on the real line (and thus they are entire functions in the Laguerre-P\'{o}lya class). From the fact that $A_{n+1}(1)\ne0$ and $B_{n+1}(1)=0$ we see that $\frak{A}_{n+1}(0) \neq 0$ and $\frak{B}_{n+1}$ has a double zero at the origin.

\smallskip

Using the construction of Lemma \ref{Sec8_Lem28} we define 
\begin{align*}
\ell_{\vartheta}(n,\lambda, \theta)&:= \ell(\frak{A}_{n+1},\lambda, \theta),\\
m_{\vartheta}(n,\lambda, \theta)&:= m(\frak{B}_{n+1},\lambda, \theta).
\end{align*}
It follows from \eqref{trig-ineq-ell} and \eqref{trig-ineq-m} that
$$\ell_{\vartheta}(n,\lambda, \theta) \leq f_{\lambda}(\theta) \leq m_{\vartheta}(n,\lambda, \theta)$$
for all $\theta \in \R$. The equalities in \eqref{Sec8_ineq_min2} and \eqref{Sec8_ineq_maj2} follow from \eqref{trig-ineq-ell-eq}, \eqref{trig-ineq-m-eq} and Corollary \ref{Sec8_Cor26}.

\subsubsection*{Optimality and uniqueness} If $\ell \in \Lambda_n$ satisfies $\ell(\theta) \leq f_\lambda(\theta)$ for all $\theta \in \R$, then by Corollary \ref{Sec8_Cor26} we have
\begin{align*}
\int_{\R/\Z} \ell(\theta) \, \d\vartheta(\theta)   =   \sum_{\xi \in \mc{A}_{n+1}} \frac{\ell(\xi)}{K_{n}(e^{2\pi i \xi},e^{2\pi i \xi})}  \leq \sum_{\xi \in \mc{A}_{n+1}} \frac{f_{\lambda}(\xi)}{K_{n}(e^{2\pi i \xi},e^{2\pi i \xi})}.
\end{align*}
If equality happens we must have $\ell(\xi) = f_{\lambda}(\xi)$ and $\ell'(\xi) = f_{\lambda}'(\xi)$ for all $\xi \in \mc{A}_{n+1}$. This gives us $2n+2$ conditions that completely determine a trigonometric polynomial of degree at most $n$. The majorant case is analogous, with the detail that $0 \in \mc{B}_{n+1}$ and $f_{\lambda}$ is not differentiable at $\xi =0$. In this case, these $2n+1$ conditions are still sufficient to determine a trigonometric polynomial of degree at most $n$. This concludes the proof.


\subsection{Extremal trigonometric polynomials II - Integrating the parameter}
We consider here two classes of nonnegative Borel measures $\varsigma$  on $(0,\infty)$. For the minorant problem we require that
\begin{equation}\label{Sec8_meas1}
\int_0^{\infty} \lambda\,e^{-a\lambda}\,\d\varsigma(\lambda) < \infty
\end{equation}
for any $a >0$, whereas for the majorant problem we require the more restrictive condition
\begin{equation}\label{Sec8_meas2}
\int_0^{\infty} \frac{\lambda}{1+\lambda} \,\d\varsigma(\lambda) < \infty.
\end{equation}
In this subsection we address the problem of majorizing and minorizing the periodic function
\begin{equation}\label{def_h_varsigma}
h_{\varsigma}(\theta) := \int_0^{\infty} \big\{f_{\lambda}(\theta) - f_{\lambda}(\hh)\big\}\,\d\varsigma(\lambda)
\end{equation}
by trigonometric polynomials of a given degree $n$, minimizing the $L^1(\R/\Z, \d\vartheta)$-error. Note the convenient subtraction of the term $f_{\lambda}(\hh)$ to generate a better decay rate in $\lambda$ as $\lambda \to 0$. If $\varsigma$ satisfies \eqref{Sec8_meas1}, $h_{\varsigma}$ is well-defined for all $\theta \notin \Z$ (it might blow up at $\theta \in \Z$), and if $\varsigma$ satisfies \eqref{Sec8_meas2}, $h_{\varsigma}$ is well-defined for all $\theta \in \R$. We shall prove the following result.

\begin{theorem}\label{Sec8_Thm29}
Let $n \in \Z^{+}$ and $\vartheta$ be a nontrivial {\it even} probability measure on $\R/\Z$. Let $\varphi_{n+1}(z) = \varphi_{n+1}(z;\d\vartheta)$ be the $(n+1)$-th orthonormal polynomial on the unit circle with respect to this measure and consider $K_n, A_{n+1}, B_{n+1}$ as defined in \eqref{Sec8_defKn}, \eqref{Sec8_DefAn} and \eqref{Sec8_DefBn}. Let $\mc{A}_{n+1} = \big\{\xi\in \R/\Z: A_{n+1}(e^{2\pi i \xi}) =0\big\}$ and $\mc{B}_{n+1} = \big\{\xi\in\R/\Z: B_{n+1}(e^{2\pi i \xi}) =0\big\}$.
\smallskip
\begin{itemize}
\item[(i)] Let $\varsigma$ satisfy \eqref{Sec8_meas1}. If $\ell \in \Lambda_n$ satisfies 
\begin{equation}\label{Sec8_ineq_min1-thm29}
\ell(\theta) \leq h_\varsigma(\theta)
\end{equation} 
for all $\theta \in \R$ then
\begin{equation}\label{Sec8_ineq_min2-thm29}
\int_{\R/\Z} \ell(\theta) \, \d\vartheta(\theta) \leq \sum_{\xi \in \mc{A}_{n+1}} \frac{h_{\varsigma}(\xi)}{K_{n}(e^{2\pi i \xi},e^{2\pi i \xi})}.
\end{equation}
Moreover, there exists a unique trigonometric polynomial $\theta \mapsto \ell_{\vartheta}(n, \varsigma, \theta) \in \Lambda_n$ satisfying \eqref{Sec8_ineq_min1-thm29} for which the equality in \eqref{Sec8_ineq_min2-thm29} holds.

\smallskip

\item[(ii)] Let $\varsigma$ satisfy \eqref{Sec8_meas2}. If $m \in \Lambda_n$ satisfies 
\begin{equation}\label{Sec8_ineq_maj1-thm29}
m(\theta) \geq h_{\varsigma}(\theta)
\end{equation} 
for all $\theta \in \R$ then
\begin{equation}\label{Sec8_ineq_maj2-thm29}
\int_{\R/\Z} m(\theta) \, \d\vartheta(\theta) \geq \sum_{\xi \in \mc{B}_{n+1}} \frac{h_{\varsigma}(\xi)}{K_{n}(e^{2\pi i \xi},e^{2\pi i \xi})}.
\end{equation}
Moreover, there exists a unique trigonometric polynomial $\theta \mapsto m_{\vartheta}(n, \varsigma, \theta) \in \Lambda_n$ satisfying \eqref{Sec8_ineq_maj1-thm29} for which the equality in \eqref{Sec8_ineq_maj2-thm29} holds.
\end{itemize}
\end{theorem}

\begin{proof} {\it Part} (i). Write 
$$h_{\lambda}(\theta) = f_{\lambda}(\theta) -f_{\lambda}(\hh).$$
From Theorem \ref{Sec8_Thm27}, the optimal trigonometric polynomial of degree at most $n$ that minorizes $h_{\lambda}$ is 
$$\widetilde{\ell}_{\vartheta}(n,\lambda,\theta) := \ell_{\vartheta}(n,\lambda,\theta) - f_{\lambda}(\hh).$$
Let us write
$$\widetilde{\ell}_{\vartheta}(n,\lambda,\theta) = \sum_{k=-n}^{n} a_k(n,\lambda)\, e^{2\pi i k \theta},$$
where $a_k = a_{k,\vartheta}$. From the interpolation properties we have, for all $\lambda >0$, 
\begin{align*}
\widetilde{\ell}_{\vartheta}(n,\lambda,\xi) &=  \sum_{k=-n}^{n} a_k(n,\lambda)\, e^{2\pi i k \xi}= h_{\lambda}(\xi)\\
\widetilde{\ell}_{\vartheta}\,'(n,\lambda,\xi) & = \sum_{k=-n}^{n} 2\pi i k \,a_k(n,\lambda)\, e^{2\pi i k \xi} = h'_{\lambda}(\xi)
\end{align*}
for all $\xi \in \mc{A}_{n+1}$. Since we have $2n+1$ coefficients $a_k$'s and $2n+2$ equations, this is an overdetermined system. We also know that this interpolation problem has a unique solution, so we can drop the last equation and invert the coefficient matrix to obtain the $a_k$'s as a function of the $h_{\lambda}(\xi)$ and $h'_{\lambda}(\xi)$. Since 
$$h_{\lambda}(\xi), h'_{\lambda}(\xi) \ll_{\xi}  \min\big\{\lambda, e^{-\lambda |\xi|}\big\}$$
for any $\hh \leq \xi< \hh$, the $a_k$'s will satisfy the same growth conditions. Since $0 \notin \mc{A}_{n+1}$ we conclude that each coefficient $a_k(n,\lambda)$ is absolutely integrable with respect to $\d\varsigma(\lambda)$ and we define
\begin{equation*}
\ell_{\vartheta}(n,\varsigma, \theta) := \int_0^{\infty} \widetilde{\ell}_{\vartheta}(n,\lambda, \theta)\,\d\varsigma(\lambda).
\end{equation*}
It is clear that 
$$\ell_{\vartheta}(n,\varsigma, \theta) \leq h_{\varsigma}(\theta)$$
for all $\theta \in \R$ and that 
$$\ell_{\vartheta}(n,\varsigma, \xi) =  h_{\varsigma}(\xi)$$
for all $\xi \in \mc{A}_{n+1}$. The optimality (and uniqueness) of this minorant follows from Corollary \ref{Sec8_Cor26} as in the proof of Theorem \ref{Sec8_Thm27}.

\smallskip

\noindent{\it Part} (ii). The majorant part is essentially analogous, just observing that $0 \in \mc{B}_{n+1}$, which justifies the more restrictive condition \eqref{Sec8_meas2} on the measure $\varsigma$.
\end{proof}

We remark that the particular choice $\d\varsigma(\lambda) = \lambda^{-1}\,\d\lambda$ yields
\begin{align*}
h_{\varsigma}(\theta) &=  \int_0^{\infty} \big\{f_{\lambda}(\theta) - f_{\lambda}(\hh)\big\}\, \lambda^{-1}\,\d\lambda\\
& =  \int_0^{\infty} \big\{f_{\lambda}(\theta) - \tfrac{2}{\lambda}\big\}\, \lambda^{-1}\,\d\lambda - \int_0^{\infty} \big\{f_{\lambda}(\hh) - \tfrac{2}{\lambda}\big\}\, \lambda^{-1}\,\d\lambda\\
&  = -\log|2 \sin \pi \theta| + \log 2.
\end{align*}
The function $\theta \mapsto -\log|2 \sin \pi \theta|$ is the harmonic conjugate of the sawtooth function treated in \cite{LV}. Theorem \ref{Sec8_Thm29} provides an extremal minorant for this function with respect to any even nontrivial probability measure $\vartheta$ on $\R/\Z$, generalizing the work done in \cite[Section 6]{CV2}. 

\subsection{Polynomial majorants on the sphere} 
\subsubsection {Main statement} Let $N\geq 2$. Let $\v \in \SS^{N-1} \subset \R^N$ be a unit vector and $h_{\varsigma}$ be defined by \eqref{def_h_varsigma}. For $\x \in \SS^{N-1}$ we consider here the problem of majorizing and minorizing the function (symmetric with respect to the $\v$-axis)
$$\x \mapsto h_{\varsigma}\left(\tfrac{1}{2\pi} \arccos( \x\cdot \v)\right)$$
by polynomials of degree at most $n$ with real coefficients (in $N$ variables) in a way to minimize the $L^1(\SS^{N-1}, w( \x\cdot \v)\,\d\sigma_N(\x))$-error. Here $\d\sigma_N$ denotes the surface measure on the sphere $\SS^{N-1}$, normalized so that $\SS^{N-1}$ has measure one, and $\x \mapsto w( \x\cdot \v)$ is an arbitrary nonnegative weight in $L^1(\SS^{N-1}, \d\sigma_N)$ with norm one (i.e. $w( \x\cdot \v)\,\d\sigma_N(\x)$ is still a probability measure on $\SS^{N-1}$). The corresponding problem for characteristic functions of spherical caps with respect to Jacobi measures was treated in \cite{LV}.

\smallskip

For a function $\x \mapsto F \left(\tfrac{1}{2\pi} \arccos( \x\cdot \v)\right)$, where $F:\R/\Z \to \C$ is a even function, a change of variables gives us 
\begin{align}
\begin{split}\label{Sec8_change_var}
\int_{\SS^{N-1}} F&  \left(\tfrac{1}{2\pi} \arccos( \x\cdot \v)\right)w( \x\cdot \v)\, \d\sigma_N(\x)\\
& \ \ \ =C_N \int_{-\hh}^{\hh} F(\theta)\,w(\cos 2\pi \theta)\,|\sin 2\pi \theta|^{N-2}\,\d\theta,
\end{split}
\end{align}
where $C_N =  \sqrt{\pi}\,\Gamma\left(\frac{N}{2}\right) \Gamma\left(\frac{N-1}{2}\right)^{-1}$, and we notice the connection with the one-dimensional problem. To state the next result consider the even nontrivial probability measure $\vartheta_{N,w}$ on $\R/\Z$ given by 
$$\d\vartheta_{N,w}(\theta) = C_N \,w(\cos 2\pi \theta)\,|\sin 2\pi \theta|^{N-2}\,\d\theta$$
and denote by $u^{-}_{\vartheta_{N,w}}(n,\varsigma)$ the expression on the right-hand side of \eqref{Sec8_ineq_min2-thm29} and by $u^{+}_{\vartheta_{N,w}}(n,\varsigma)$ the expression on the right-hand side of \eqref{Sec8_ineq_maj2-thm29} associated to this measure.

\begin{theorem}\label{Sec8_Thm30}
Let $N\geq2$, $n \in \Z^{+}$ and $h_{\varsigma}$ be defined by \eqref{def_h_varsigma}.
\smallskip
\begin{itemize}
\item[(i)] Let $\varsigma$ satisfy \eqref{Sec8_meas1}. If $\L(\x)$ is a polynomial of degree at most $n$ with real coefficients such that 
\begin{equation}\label{Sec8_ineq_min-thm30}
\L(\x) \leq h_{\varsigma}\left(\tfrac{1}{2\pi} \arccos( \x\cdot \v)\right)
\end{equation}
for all $\x \in \SS^{N-1}$ then 
\begin{equation}\label{Sec8_ineq_min2-thm30}
\int_{\SS^{N-1}} \L(\x) \,w( \x\cdot \v)\, \d\sigma_N(\x) \leq u^{-}_{\vartheta_{N,w}}(n,\varsigma).
\end{equation}
The equality is attained for $\L(\x) = \ell_{\vartheta_{N,w}}\left(n,\varsigma, \tfrac{1}{2\pi}\arccos (\x \cdot \v)\right)$.

\smallskip

\item[(ii)] Let $\varsigma$ satisfy \eqref{Sec8_meas2}. If $\M(\x)$ is a polynomial of degree at most $n$ with real coefficients such that 
\begin{equation}\label{Sec8_ineq_maj-thm30}
\M(\x) \geq h_{\varsigma}\left(\tfrac{1}{2\pi} \arccos( \x\cdot \v)\right)
\end{equation}
for all $\x \in \SS^{N-1}$ then 
\begin{equation}\label{Sec8_ineq_maj2-thm30}
\int_{\SS^{N-1}} \M(\x) \,w( \x\cdot \v)\, \d\sigma_N(\x) \geq u^{+}_{\vartheta_{N,w}}(n,\varsigma).
\end{equation}
The equality is attained for $\M(\x) = m_{\vartheta_{N,w}}\left(n,\varsigma, \tfrac{1}{2\pi}\arccos (\x \cdot \v)\right)$.
\end{itemize}
\end{theorem}


\subsubsection{Symmetrization lemma} Before we proceed to the proof of Theorem \ref{Sec8_Thm30} we present a symmetrization lemma from \cite{LV} in the same spirit of Lemma \ref{Lemma19_HV} above. Let $SO_{\v}(N)$ be the topological subgroup of all rotations of \,$\SS^{N-1}$ (real orthogonal $N \times N$ matrices $M$ with $\det M =1$) leaving $\v$ fixed. Let $\varrho$ be its left-invariant (and also right-invariant since $SO_{\v}(N)$ is compact) Haar measure, normalized so that $\varrho(SO_{\v}(N)) =1$. For every polynomial $\F(\x)$ define
\begin{equation}\label{Sec8_symm}
\breve{\F}(\x) = \int_{SO_{\v}(N)} \F(M\x)\,\d\varrho(M).
\end{equation}

\begin{lemma}\label{Sec8_Lem31}
The following propositions hold.
\smallskip
\begin{enumerate}
\item[(i)] Let $F(\theta)$ be a real trigonometric polynomial of degree $n$. If $F(\theta)$ is an even function of $\theta$, then $F \left(\tfrac{1}{2\pi} \arccos( \x\cdot \v)\right)$ is a polynomial of $\x \cdot \v$ of degree $n$ with real coefficients.
\smallskip
\item[(ii)] If $\F(\x)$ is a polynomial of $x_1, x_2, \ldots, x_N$ of degree $n$ with real coefficients, then $\breve{\F}(\x)$ is a polynomial of $\x\cdot\v$ of degree $n$ with real coefficients when $\x$ is restricted to the sphere $\SS^{N-1}$. Moreover, the identity
\begin{equation}\label{Sec8_eq30.5}
\int_{\SS^{N-1}} \breve{\F}(\x)\,w( \x\cdot \v)\, \d\sigma_N(\x) = \int_{\SS^{N-1}} \F(\x)\,w( \x\cdot \v)\, \d\sigma_N(\x)
\end{equation}
holds.
\end{enumerate}
\end{lemma}
\begin{proof}
This is \cite[Lemma 13 and Lemma 14]{LV}.
\end{proof}


\subsubsection{Proof of Theorem \ref{Sec8_Thm30}} Let $\L(\x)$ be a polynomial of degree at most $n$ satisfying \eqref{Sec8_ineq_min-thm30}. Using \eqref{Sec8_symm} we find that 
\begin{equation}\label{Sec8_eq31}
\breve{\L}(\x) \leq h_{\varsigma}\left(\tfrac{1}{2\pi} \arccos( \x\cdot \v)\right)
\end{equation}
for all $\x \in \SS^{N-1}$. From Lemma \ref{Sec8_Lem31} (ii) there exists an even real trigonometric polynomial $\ell$ of degree at most $n$ such that 
\begin{equation}\label{Sec8_eq32}
\breve{\L}(\x) = \ell\left(\tfrac{1}{2\pi} \arccos( \x\cdot \v)\right)
\end{equation}
for all $\x \in \SS^{N-1}$. From \eqref{Sec8_eq31} and \eqref{Sec8_eq32} we find that $\ell(\theta) \leq h_{\varsigma}(\theta)$ for every $\theta \in \R$. Inequality \eqref{Sec8_ineq_min2-thm30} now follows from \eqref{Sec8_eq30.5}, \eqref{Sec8_change_var} and Theorem \ref{Sec8_Thm29}.

\smallskip

From Lemma \ref{Sec8_Lem31} (i), the function $\L(\x) := \ell_{\vartheta_{N,w}}\left(n,\varsigma, \tfrac{1}{2\pi}\arccos (\x \cdot \v)\right)$ is a polynomial of $\x\cdot\v$ of degree at most $n$ that verifies \eqref{Sec8_ineq_min-thm30}. The fact that the equality in \eqref{Sec8_ineq_min2-thm30} is attained follows from \eqref{Sec8_change_var} and Theorem \ref{Sec8_Thm29}.

\smallskip

The proof for the majorant part is analogous.

\section*{Acknowledgements}
E. C. acknowledges support from CNPq-Brazil grants $473152/2011-8$ and $302809/2011-2$, and FAPERJ grant E-26/103.010/2012. Part of this work was completed during a visit of F. L. to IMPA, for which he gratefully acknowledges the support.


\begin{thebibliography}{99}

\bibitem{B} 
L. de Branges,
\newblock {\it Hilbert spaces of entire functions},
\newblock Prentice-Hall, 1968.

\bibitem{B2} 
L. de Branges,
\newblock Homogeneous and periodic spaces of entire functions,
\newblock Duke Math. Journal 29 (1962), 203--224.

\bibitem{BMV}
J.~T.~Barton, H.~L.~Montgomery, and J.~D.~Vaaler,
\newblock Note on a diophantine inequality in several variables,
\newblock Proc. Amer. Math. Soc. 129 (2001), 337--345.

\bibitem{Ber}
S. N. Bernstein,
\newblock Sur une propri\'{e}t\'{e} des fonctions entieres, 
\newblock C. R. Acad. Sci. 176 (1923), 1602--1605.

\bibitem{BV} R.\ Bojanic and R.\ DeVore,
\newblock On polynomials of best one sided approximation,
\newblock L'Enseignement Math\'ematique 12 (1966),
\newblock 139--164.

\bibitem{Car}
E. Carneiro,
\newblock Sharp approximations to the Bernoulli periodic functions by trigonometric polynomials,
\newblock J. Approx. Theory 154 (2008), 90--104.

\bibitem{CC}
E. Carneiro and V. Chandee,
\newblock Bounding $\zeta(s)$ in the critical strip, 
\newblock J. Number Theory 131 (2011), 363--384. 

\bibitem{CCM}
E. Carneiro, V. Chandee and M. Milinovich,
\newblock Bounding $S(t)$ and $S_1(t)$ on the Riemann hypothesis,
\newblock Math. Ann. 356 (2013), 939--968.

\bibitem{CCLM}
E. Carneiro, V. Chandee, F. Littmann and M. Milinovich,
\newblock Hilbert spaces and the pair correlation of zeros of the Riemann zeta-function,
\newblock preprint.

\bibitem{CL}
E. Carneiro and F. Littmann,
\newblock Bandlimited approximations to the truncated Gaussian and applications,
\newblock Constr. Approx. 38 (2013), 19--57. 

\bibitem{CL2}
E. Carneiro and F. Littmann,
\newblock Entire approximations for a class of truncated and odd functions,
\newblock  J. Fourier Anal. Appl. 19 (2013), 967--996. 

\bibitem{CLV}
E.\ Carneiro, F.\ Littmann, and J. D.\ Vaaler,
\newblock Gaussian subordination for the Beurling-Selberg extremal problem,
\newblock Trans.\ Amer.\ Math.\ Soc. 365 (2013), 3493--3534. 
\bibitem{CV2}
E.~Carneiro and J.~D.~Vaaler,
\newblock Some extremal functions in Fourier analysis, II,
\newblock Trans. Amer. Math. Soc. 362 (2010), 5803--5843.

\bibitem{CV3}
E.~Carneiro and J.~D.~Vaaler,
\newblock Some extremal functions in Fourier analysis, III,
\newblock Constr. Approx. 31, No. 2 (2010), 259--288.

\bibitem{CS}
V. Chandee and K. Soundararajan,
\newblock Bounding $|\zeta(\tfrac{1}{2} + it)|$ on the Riemann hypothesis,
\newblock Bull. London Math. Soc. 43 (2) (2011), 243--250.

\bibitem{ET} 
P.~Erd\"{o}s and P.~Tur\'{a}n,
\newblock On a problem in the theory of uniform distribution,
\newblock Indag. Math., 10, (1948), 370--378.

\bibitem{Ga}
P. X. Gallagher, 
\newblock Pair correlation of zeros of the zeta function, 
\newblock J. Reine Angew. Math. 362 (1985), 72--86.

\bibitem{Gan} M. I. Ganzburg,
\newblock $L$-approximation to non-periodic functions,
\newblock Journal of concrete and applicable mathematics 8, No. 2 (2010), 208--215.

\bibitem{GL}
M. I. Ganzburg and D. S. Lubinsky, 
\newblock Best approximating entire functions to $|x|^{\alpha}$ in $L^2$, 
\newblock Complex analysis and dynamical systems III,  93--107, Contemp. Math. 455, Amer. Math. Soc., Providence, RI, 2008.

\bibitem{GG}
D. A. Goldston and S. M. Gonek,
\newblock A note on $S(t)$ and the zeros of the Riemann zeta-function,
\newblock Bull. London Math. Soc. 39 (2007), 482--486.

\bibitem{G} 
L. Golinskii, 
\newblock Quadrature formula and zeros of para-orthogonal polynomials on the unit circle, \newblock Acta Math.\ Hungar. 96, no. 3 (2002), 169--186.

\bibitem{GV}
S.~W.~Graham and J.~D.~Vaaler,
\newblock A class of extremal functions for the Fourier transform,
\newblock Trans. Amer. Math. Soc. 265 (1981), 283--382.


\bibitem{HPL} 
G.~Hardy, G.~Polya and J.~Littlewood,
\newblock {\it Inequalities},
\newblock Cambridge University Press, 1967.

\bibitem{HW} 
I. I. Hirschman and D. V. Widder,
\newblock {\it The convolution transform},
\newblock Princeton Univ. Press, 1955.

\bibitem{HV} 
J.~Holt and J.~D.~Vaaler,
\newblock The Beurling-Selberg extremal functions for a ball in the Euclidean space,
\newblock Duke Math. Journal 83 (1996), 203--247.

\bibitem{Hor}
L. H\"{o}rmander,
\newblock {\it Linear Partial Differential Operators},
\newblock Springer-Verlag, 1964.

\bibitem{ILS}
H. Iwaniec, W. Luo and P. Sarnak,
\newblock Low lying zeros of families of $L$-functions,
\newblock Publ. Math. IHES. 91 (2000), 55--131.

\bibitem{JNT} 
W. B. Jones, O. Njastad and W. J. Thron, 
\newblock Moment theory, orthogonal polynomials, quadrature, and continued fractions associated with the unit circle, 
\newblock Bull. London Math. Soc. 21 (1989), 113--152.

\bibitem{KW}
M. Kaltenb\"{a}ck and H. Woracek, 
\newblock P\'{o}lya class theory for Hermite-Biehler functions of finite order,
\newblock J. London Math. Soc. (2) 68 (2003), no. 2, 338--354.

\bibitem{K} 
M. G. Krein,
\newblock A contribution to the theory of entire functions of exponential type,
\newblock Bull. Acad. Sci. URSS. Ser. Math. 11 (1947), 309--326.

\bibitem{K2}
M. G. Krein, 
\newblock On the best approximation of continuous differentiable functions on the whole real axis,
\newblock Dokl. Akad. Nauk SSSR 18 (1938), 615--624 (Russian).

\bibitem{La}
E. Laguerre, 
\newblock Sur les fonctions du genre z\'{e}ro et du genre un, 
\newblock Comptes Rendus l'Acad. Sci. 98 (1882), 828--831.

\bibitem{Li}
X. J. Li,
\newblock On reproducing kernel Hilbert spaces of polynomials,
\newblock Math. Nachr. 185 (1997), 115--148.

\bibitem{LV} 
X.~J.~Li and J.~D.~Vaaler,
\newblock Some trigonometric extremal functions and the Erd\"{o}s-Tur\'{a}n type inequalities,
\newblock Indiana Univ. Math. J. 48, No. 1 (1999), 183--236.

\bibitem{L1} 
F.~Littmann,
\newblock Entire approximations to the truncated powers,
\newblock Constr. Approx. 22, No. 2 (2005), 273--295.

\bibitem{L3}
F.~Littmann,
\newblock Entire majorants via Euler-Maclaurin summation,
\newblock Trans. Amer. Math. Soc. 358, No. 7 (2006), 2821--2836.

\bibitem{L4}
F. Littmann,
\newblock Quadrature and extremal bandlimited functions, 
\newblock SIAM J. Math. Anal. 45 (2013), no. 2, 732--747.

\bibitem{Lu}
D. Lubinsky,
\newblock On the Bernstein constants of polynomial approximation, 
\newblock Constr. Approx. 25 (2007), 303--366. 

\bibitem{M} 
H.~L.~Montgomery,
\newblock The analytic principle of the large sieve,
\newblock Bull. Amer. Math. Soc. 84, No. 4 (1978), 547--567.

\bibitem{M2}
H.~L.~Montgomery,
\newblock {\em Ten Lectures on the Interface Between Analytic Number Theory and Harmonic Analysis},
\newblock CBMS No. 84, Amer. Math. Soc., Providence, 1994.

\bibitem{MV} 
H.~L.~Montgomery and R.~C.~Vaughan,
\newblock Hilbert's Inequality,
\newblock J. London Math. Soc. 8 (2) (1974), 73--81.

\bibitem{Na}
B.\ Sz.-Nagy,
\newblock \"Uber gewisse Extremalfragen bei transformierten trigonometrischen Entwicklungen II,
\newblock Ber. Math.-Phys. Kl. S\"achs. Akad. Wiss. Leipzig 91 (1939), 3--24.

\bibitem{PP}
M.~Plancherel and G.~Polya,
\newblock Fonctions enti\'eres et int\'egrales de Fourier multiples, (Seconde partie)
\newblock Comment. Math. Helv. 10 (1938), 110--163.

\bibitem{P1}
G. P\'{o}lya, 
\newblock \"{U}ber Ann\"{a}herung durch Polynome mit lauter reelen Wurzeln, 
\newblock Rendiconti del Circolo Mat. Palermo 36 (1913), 279--295.

\bibitem{S1}
A. Selberg,
\newblock, Remarks on sieves,
\newblock Proc. 1972 Number Theory Conference, Univ. of Colorado, Boulder (1972), 205--216.

\bibitem{S2}
A.~Selberg,
\newblock Lectures on Sieves, {\em Atle Selberg: Collected Papers}, Vol. II,
\newblock Springer-Verlag, 1991.

\bibitem{S} 
B. Simon, 
\newblock OPUC on one foot, 
\newblock Bull. Amer. Math. Soc. 42, no. 4 (2005), 431--460.

\bibitem{St} 
E.~Stein, 
\newblock {\it Singular Integrals and Differentiability Properties of Functions},
\newblock Princeton University Press, 1970.

\bibitem{SW}
E. M. Stein and G. Weiss,
\newblock {\it Fourier Analysis on Euclidean spaces},
\newblock Princeton Univ. Press, 1971.

\bibitem{Sz}
G. Szeg\"{o},
\newblock {\it Orthogonal polynomials},
\newblock Colloquim Publications, Volume 23, AMS, 1939.

\bibitem{V}
J.~D.~Vaaler,
\newblock Some extremal functions in Fourier analysis,
\newblock Bull. Amer. Math. Soc. 12 (1985), 183--215.

\bibitem{W} 
G. N. Watson,
\newblock {\it A Treatise on the Theory of Bessel Functions},
\newblock Cambridge Univ. Press, 1922.

\bibitem{Z} 
A.~Zygmund,
\newblock {\em Trigonometric Series},
\newblock Cambridge University Press, 1959.


\end{thebibliography}
\end{document}